\documentclass[a4paper,12pt,oneside,reqno]{amsart}
\usepackage[latin1]{inputenc}
\usepackage[english]{babel}
\usepackage[dvips]{graphicx}
\usepackage{amssymb}
\usepackage{amsmath}
\usepackage{amsthm}
\usepackage{url}
\usepackage{hyperref}
\usepackage{a4wide}

\usepackage[hmargin=1.8cm,vmargin=2.3cm]{geometry} 

\def\[#1\]{\begin{align*}#1\end{align*}}
\def\be#1\ee{\begin{align}#1\end{align}}
\def\bea#1\eea{\begin{align}#1\end{align}}
\def\ben#1\een{\begin{align*}#1\end{align*}}

\newcommand{\n}[1]{\left\Vert #1\right\Vert}
\newcommand{\la}{\left\langle}
\newcommand{\ra}{\right\rangle}
\newcommand{\R}{\mathbb{R}}

\newcommand{\N}{\mathbb{N}}

\newcommand{\mc}[1]{\mathcal{#1}}
\newcommand{\s}{\mc}
\newcommand{\p}[2]{\frac{\partial #1}{\partial #2}}

\def\ol#1{\overline{#1}}

\newcommand{\qmatrix}[1]{ \left( \begin{matrix} #1 \end{matrix} \right) }

\def\[#1\]{\begin{align*}#1\end{align*}}
\def\be#1\ee{\begin{align}#1\end{align}}
\def\bea#1\eea{\begin{align}#1\end{align}}
\def\ben#1\een{\begin{align*}#1\end{align*}}

\newtheoremstyle{theorem}{0.5cm}{0.5cm}%
   {}
   {}
   {\bfseries}
   {}
   {2ex}
   {\thmname{#1}\thmnumber{ #2}\thmnote{ #3}}
\theoremstyle{theorem}
\newtheorem{theorem}{Theorem}[section]

\newtheorem{example}[theorem]{Example}
\newtheorem{corollary}[theorem]{Corollary}
\newtheorem{remark}[theorem]{Remark}

\newtheorem{lemma}[theorem]{Lemma}

\begin{document}
\title[On  linear hypersingular BTE and its variational formulation]
{On linear hypersingular Boltzmann transport equation and its variational formulation }


 
\author{Jouko Tervo}
\address{University of Eastern Finland, Department of Applied Physics, P.O.Box 1627, 70211 Kuopio, Finland}
\email{040tervo@gmail.com}
 
\date{\today}

\maketitle

\begin{abstract}

For charged particle transport the linear Boltzmann transport equation (BTE) turns out to be a partial hyper-singular  integro-differential operator. 
This is due to the fact that
the related differential cross-sections $\sigma(x,\omega',\omega,E',E)$
may have hyper-singularities. In these cases the energy integral appearing in the collision terms  must be interpreted as the  Hadamard finite part integral leading to hyper-singular integral operators.
The article considers a refined expression for the exact transport operator and related variational formulations of the inflow initial boundary value problem for one particle equation  containing hyper-singularities. 
We find that the exact BTE contains the first-order partial derivatives with respect to energy combined by partial Hadamard (first-order) singular integral operators. In addition, it contains the second-order partial derivatives with respect to angle and some mixed terms.
The analysis will be carried out only for the so called M\o ller-type interaction (scattering) which is a kind of prototype of hyper-singular interactions. The generalizations to other type of collisions, such as to Bremsstrahlung, go analogously.
We also expose a weak form (the variational formulation) of the hyper-singular transport problem. 
Another variant variational formulation 
decreases the level of singularities in the integration (appearing in the due bilinear form) containing only singularities of order one that is, singularities like ${1\over{E'-E}}dE' dE$.
The variational formulation is an essential step in order to show the existence of generalized solutions e.g. by Lions-Lax-Milgram Theorem based methods (proceedings for solution spaces and existence theory are omitted here). The corresponding approximative transport operator is deduced. It turns out to be a CSDA-Fokker-Planck type operator.
\end{abstract}

{\small \textbf{Keywords: } Linear Boltzmann equation,  hyper-singular integral operators, variational formulation, charged particle transport}

{\small \textbf {AMS-Classification:}  35Q20, 45E99, 35R09}

\section{Introduction}\label{intro}

The \emph{Boltzmann transport equation} (BTE) models changes of the number density of particles in phase space whose variables are position, velocity direction (angle) and energy.
For general theory of linear BTEs with relevant boundary conditions we refer to \cite{dautraylionsv6} and \cite{agoshkov}. In \cite{case}, \cite{cercignani}, \cite{duderstadt}, \cite{pomraning}  the subject is considered from a more physical point of view.
Some more recent issues (including spectral and certain inverse problems) is exposed in \cite{mokhtarkharroubi},
and  general non-linear transport theory e.g. in \cite{ukai}, \cite{bellamo}.
A  mathematical survey of non-linear collision theory of particle transport 
is given in \cite{villani}.

In the case of charged particle transport the linear BTE turns out to be a so called \emph{partial hyper-singular  integro-differential operator}. For theoretical and computational reasons it has  been approximated by more simple models, including e.g. in the Continuous Slowing Down Approximation (CSDA) (e.g. \cite{wilson}, \cite{frank10}, \cite{tervo18-up}) and the linear Fokker-Plank approximation
(e.g. \cite{pomraning92}). In these approximations the resulting operator is  a pure partial integro-differential operator without hyper-singular integrals. In \cite{tervo18} or more extensively in \cite{tervo18-up} we gave a (non-conventional) partial integro-pseudo-differential approximation for  certain kind of charged particle transports. In this paper we investigate more closely the exact partial hyper-singular integral operator related to the so called \emph{M\o ller scattering}.

The transport of relevant particles (among others in dose calculation for radiation therapy, \cite{lorence})
can be {\it formally}  modelled by the following \emph{linear coupled system of BTEs}
\bea\label{intro1}
(\omega\cdot\nabla_x\psi_j)(x,\omega,E)&+
\Sigma_j(x,\omega,E)\psi_j(x,\omega,E)\nonumber\\
&
-(K_j\psi)(x,\omega,E)=f_j(x,\omega,E),\ (x,\omega,E)\in G\times S\times I
\eea
for $j=1,\cdots, N$, combined with an \emph{inflow boundary condition}
\be\label{intro2}
{\psi_j}_{|\Gamma_-}=g_j,\quad j=1,\cdots,N
\ee
where 
\be\label{intro3}
(K_j\psi)(x,\omega,E)=\sum_{k=1}^N\int_{S_{n-1}'\times I'}\sigma_{kj}(x,\omega',\omega,E',E)\psi_k(x,\omega',E') d\omega' dE'.
\ee
Here $G\subset\R^n$ is the spatial domain. In applications $n$ is typically $1,\ 2$ or $3$.  $S_{n-1}\subset \R^n$ is the unit sphere and $I=[E_0,E_m]$ is the energy interval. There are some benefits to choose $G\times S\times I$ for the state space instead of $G\times V$ where $V$ is the velocity space. 
Nevertheless, the below analysis could be carried out in $G\times V$ as well.
The set $\Gamma_-$ is "the inflow boundary part of $\partial G$" (see section \ref{pre}). $N$ assigns the number of evolving particles.
Above $(\omega',E')$ refers to the angle and energy of the incoming particle whereas $(\omega,E)$ refers to the angle and energy of the leaving particle.
On the right in (\ref{intro1}), the functions $f_j$ represent (internal) sources and in (\ref{intro2}) $g_j$  are (inflow) boundary sources.
The system is coupled through the integral operators $K_j$. The solution $\psi=(\psi_1,\cdots,\psi_N)$ of the problem (\ref{intro1})-(\ref{intro2}) 
is a vector-valued function whose components describe the radiation fluxes of various particles under consideration.
The equation (\ref{intro1}) is a steady state counterpart of its dynamical (time-dependent) equation. In many applications
it is sufficient to consider only the steady state equations because the
flux $\psi$ reaches the steady state nearly instantly. 
For simplicity we restrict ourselves to the case $n=3$ and  we denote $S=S_2$.

The differential cross-sections $\sigma_{kj}(x,\omega',\omega,E',E)$
may have singularities, or even \\ \emph{hyper-singularities}, and in these cases the integral $\int_{ I'}$ appearing in the collision terms $K_j$ must be interpreted as the {\it Hadamard finite part integral} which we denote by ${\rm p.f.}\int_{ I'}$. Moreover, the differential cross sections may contain Dirac's $\delta$-distributions with respect to $\omega$.
In \cite{tervo18-up} we  presented some details of real, physical collision operators.
In particular, we found that certain differential cross sections may contain hyper-singularities like ${1\over{(E'-E)^2}}dE'$.
These kind of singularities lead to extra pseudo-differential like (or approximately partial differential) terms in the transport equation (\cite{tervo18-up}, section 2.3). We also remark that in these cases additionally an \emph{initial condition} (or conditions) with respect to $E$ must be imposed to obtain mathematically (and physically) well-posed problems.
The analysis presented in \cite{tervo18-up} revealed the exact form of transport operators for certain interactions (collisions).

The existence of solutions for the problem (\ref{intro1}), (\ref{intro2}),
(as well as for the time-dependent problem with the due initial condition) has been studied for  coupled systems  in \cite{tervo17-up}
(the results of which remaining valid, after slight modifications, for any $1\leq p<\infty$) and for  single equations e.g. in \cite{egger14}, \cite{dautraylionsv6}, \cite{agoshkov}.
In these references it is assumed that the 
collision operator $K$ satisfies the so-called \emph{(partial) Schur criterion} (for boundedness) which is not valid for all species of particle interactions.
In \cite{tervo17}, \cite{tervo18} (or more comprehensively in \cite{tervo18-up}) we studied systematically the existence of solutions for the CSDA-system which is an approximation of the hyper-singular BTE-equation.
It extended the results of \cite{frank10}, where a spatially homogeneous stopping power was assumed. 
In this article we derive a refined expression for the exact transport operator and the related variational formulation of the inflow initial boundary value problem for one particle equation ($N=1$) containing hyper-singularities. 
The analysis will be carried out only for the M\o ller-type scattering which is a kind of prototype of hyper-singular interactions. The generalizations to other type of collisions (such as to Bremsstrahlung) go analogously.

For the first instance,
in  section \ref{moller} we  consider shortly  the M\o ller scattering.  This interaction models the electron's (and positron's) inelastic collisions in particle transport. We expose the {\it hyper-singular partial integral} expression and its {\it pseudo-differential like} expression obtained in \cite{tervo18-up} for the corresponding collision operator $K$ (and transport operator).  
These expressions are still quite implicit and the refined expression of the transport operator, say $T$, will be given  in section \ref{exact}. It turns out that $T$ is of the form (Theorem \ref{ex-eq})
\bea\label{intro-exact}
&
(T\psi)(x,\omega,E)=
-{\partial\over{\partial E}}\Big(
{\s H}_1\big((\ol{\s K}_2\psi)(x,\omega,\cdot,E)\big)(E)\Big)
-
2\pi\ 
\hat\sigma_2(x,E,E){\p \psi{E}}(x,\omega,E)\nonumber\\
&
-\hat\sigma_2(x,E,E)
\sum_{|\alpha|\leq 2}a_{\alpha}(E,\omega)(\partial_{\omega}^\alpha\psi)(x,\omega,E)\nonumber\\
&
+{\rm p.f.}\int_E^{E_m}{1\over{E'-E}}
\hat\sigma_2(x,E',E)\int_{0}^{2\pi}\la \nabla_S\psi(x,\gamma(E',E,\omega)(s),E'),{\p \gamma{E}}(E',E,\omega)(s)\ra ds dE'
\nonumber\nonumber\\
&
+
{\rm p.f.}\int_E^{E_m}{1\over{E'-E}}
{\p {\hat\sigma_2}{E}}(x,E',E)\int_{0}^{2\pi}\psi(x,\gamma(E',E,\omega)(s),E')ds dE'
\nonumber\\
&
+
{\s H}_1\big((\ol{\s K}_1\psi)(x,\omega,\cdot,E)\big)(E)
-
2\pi\ 
{\p {\hat\sigma_2}{E'}}(x,E,E)\psi(x,\omega,E)\nonumber\\
&
+\omega\cdot\nabla_x\psi+\Sigma(x,\omega,E)\psi
-(K_r\psi)(x,\omega,E)
\eea
where for $j=1,2$ the (first-order) Hadamard finite part integral operator is defined by
\[
{\s H}_1\big((\ol{\s K}_j\psi)(x,\omega,\cdot,E)\big)(E)
:=
{\rm p.f.}\int_E^{E_m}{1\over{E'-E}}
\hat\sigma_j(x,E',E)\int_{0}^{2\pi}\psi(x,\gamma(E',E,\omega)(s),E')ds dE'.
\]
The second order partial differential operator 
$\sum_{|\alpha|\leq 2}a_{\alpha}(E,\omega)\partial_{\omega}^\alpha\psi$ can be given explicitly for M\o ller scattering and its coordinate free form is $c(E)\Delta_S$ where $\Delta_S$ is the \emph{Laplace-Beltrami operator} on $S$ (see Remark \ref{rem-aij}).
The operator $K_r$ above is the so called  {\it restricted collision operator}. Roughly speaking it is the residual when the singular part is separated from the collision operator.
In \cite{tervo18} we showed  that $K_r$ is a \emph{bounded operator} in $L^2(G\times S\times I)$. The result
is based on the  Schur criterion. In addition, we showed that $\Sigma-K_r$ is \emph{coercive (accretive) operator} in $L^2(G\times S\times I)$ under relevant assumptions.
The operator (\ref{intro-exact}) reminds of the CSDA-Fokker-Plank operator but it contains additionally the Hadamard finite part integral operators which can be  considered as pseudo-differential-like  operators. The expression (\ref{intro-exact}) of $T$ gives a solid background for the use of CSDA-Fokker-Plank operator as an approximation.

The relevant transport problem requires besides of the operator equation $T\psi=f$ the due initial and inflow boundary conditions and therefore the total transport problem is
\be\label{tp-intro}
T\psi=f,\ 
\psi_{|\Gamma_-}=g,\
\psi(.,.,E_m)=0
\ee
where $E_m$ is the so called \emph{cut-off energy}.
In  section \ref{var-for} we expose a weak form of the hyper-singular transport problem. 
We use the strongly hyper-singular operator as a starting point of derivations although the expression (\ref{intro-exact}) could be applied equally well (in fact somewhat more simply).
The obtained weak form is the so called \emph{variational equation} 
\[
B(\psi,v)=Fv,\ v\in \mc D_{(0)}.
\]
for the transport problem. 
Here $\mc D_{(0)}$ is a relevant space of test functions and $B(.,.)$ and $F$ are the due bilinear  and linear forms, respectively. The variational formulation given in section \ref{lowering}
decreases the level of singularities in the integration and it contains only singularities of order one that is, singularities like ${1\over{E'-E}}dE' dE$.
The variational formulation is an essential step in order to define generalized solutions and to show the existence of solutions e.g. by Lions-Lax-Milgram Theorem based methods. It also gives a platform needed for Galerkin finite element methods. 
The consideration of singular integral operators by element methods is well known in the field of boundary element methods (BEM), e.g. \cite{costabel}. So in principle, no approximations (such as CSDA or Fokker-Plank) are needed for numerical solutions. 
In section \ref{ad-sec} we compute  related formal adjoint operators which are beneficial to existence theory as well.
In this paper we, however, omit  proceedings for solution spaces and existence theory instead of some notes in section \ref{bilin-bound}. In the final section \ref{tr-app} the corresponding approximative transport operator is deduced which turns out to be a CSDA-Fokker-Planck operator in nature.

\section{Preliminaries}\label{pre}

\subsection{Basic notations}\label{b-f-spaces}

We assume that $G$ is an open bounded 
 set in $\R^3$ such that $\ol{G}$ is a $C^1$-manifold with boundary (as a submanifold of $\R^3$; cf. \cite{lee}).
In particular, it follows from this definition that $G$ lies on one side of its boundary.
The unit outward (with respect to $G$) pointing normal on $\partial G$ is denoted by $\nu$ 
and the surface measure (induced by the Lebesgue measure $dx$) on $\partial G$ is denoted as $d\sigma$.
We let $S=S_2$ be the unit sphere in $\R^3$ equipped with the standard surface measure $d\omega$. 
Furthermore, let $I=[E_0,E_{\rm m}]$  where
$0\leq E_0<E_{\rm m}<\infty$. 
 We could replace $I$ by $I=[E_0,\infty[$ but we neglect this case here.
We shall denote by $I^\circ$ the interior  of $I$.
The interval $I$ is equipped with the Lebesgue measure $dE$.
All functions considered in this paper are real-valued, and all linear (Hilbert, Banach) {\it spaces are real}.

For $(x,\omega)\in G\times S$ the \emph{escape time (in the direction $\omega$)} $t(x,\omega)=t_-(x,\omega)$ is defined by 
$ t(x,\omega):=\inf\{s>0\ |\ x-s\omega\not\in G\}  =\sup\{T>0\ |\ x-s\omega\in G\ {\rm for\ all}\ 0<s<T\}$.
We define
\[
\Gamma':=(\partial G)\times S,\quad
\Gamma:=\Gamma'\times I,
\]
and their subsets
\begin{alignat*}{3}
\Gamma_0' :={}&\{(y,\omega)\in \Gamma'\ |\ \omega\cdot\nu(y)=0\},
\quad &&\Gamma_0:=\Gamma'_0\times I, \\
\Gamma_{-}':={}&\{(y,\omega)\in \Gamma'\ |\ \omega\cdot\nu(y)<0\}, && \Gamma_{-}:=\Gamma_{-}'\times I,\\
\Gamma_{+}':={}&\{(y,\omega)\in \Gamma'\ |\ \omega\cdot\nu(y)>0\}, &&\Gamma_{+}:=\Gamma_{+}'\times I.
\end{alignat*}
Note that $\Gamma=\Gamma_0\cup \Gamma_-\cup \Gamma_+$.

In the sequel we denote for $k\in\N_0$, 
\[
C^k(\ol G\times S\times I):=\{\psi\in C^k(G\times S\times I^\circ)\ |\ \psi=f_{|G\times S\times I^\circ},\ f\in C_0^k(\R^3\times S\times\R)\},
\]
where for a $C^k$-manifold $M$ without boundary, the set $C_0^k(M)$ denotes
 the set of all $C^k$-functions on $M$ with compact support.
Define the (Sobolev) space $W^2(G\times S\times I)$  by
\[
W^2(G\times S\times I)
:=\{\psi\in L^2(G\times S\times I)\ |\  \omega\cdot\nabla_x \psi\in L^2(G\times S\times I) \}.
\]
The  space $W^2(G\times S\times I)$ is a Hilbert space when equipped  with the inner product
\[
\la {\psi},v\ra_{W^2(G\times S\times I)}:=\la {\psi},v\ra_{L^2(G\times S\times I)}+
\la\omega\cdot\nabla_x\psi,\omega\cdot\nabla_x v\ra_{L^2(G\times S\times I)}.
\]
The space $C^1(\ol G\times S\times I)$ is a dense subspace of  $W^2(G\times S\times I)$ (e.g. \cite{friedrich44}). 

Let $T^2(\Gamma)$ be the weighted Lebesgue space $L^2(\Gamma,|\omega\cdot\nu|d\sigma d\omega dE)$.
The trace $\gamma(\psi):=\psi_{|\Gamma}$ (for a detailed study of inflow trace theory see e.g. \cite{cessenat84} and \cite{tervo18-up}, section 2.2)
is well-defined in the  space
\[
\widetilde{W}^2(G\times S\times I):=\{\psi\in W^2(G\times S\times I)\ |\ \gamma(\psi)\in T^2(\Gamma) \}
\]
which is a Hilbert space when equipped with the inner product
\[
\la {\psi},v\ra_{\widetilde{W}^2(G\times S\times I)}:=\la {\psi},v\ra_{W^2(G\times S\times I)}+ \la {\gamma(\psi)},\gamma(v)\ra_{T^2(\Gamma)},
\]
where
\[
\la h_1,h_2\ra_{T^2(\Gamma)}:=\int_\Gamma h_1(x,\omega,E)h_2(x,\omega,E)|\omega\cdot\nu| d\sigma d\omega dE.
\]
We recall that the  Green's formula (\cite{dautraylionsv6}, p. 225)
\be\label{green}
\int_{G\times S\times I}(\omega\cdot \nabla_x \psi)v\, dxd\omega dE
+\int_{G\times S\times I}(\omega\cdot \nabla_x v)\psi\, dxd\omega dE=
\int_{\partial G\times S\times I}(\omega\cdot \nu) v\, \psi\, d\sigma d\omega dE,
\ee
is valid  for every $\psi,\ v\in \widetilde{W}^2( G\times S\times I)$.

\subsection{Some  tools from analysis}\label{taylor-S}

We recall the following standard concepts from analysis which we  shall frequently 
need.
The Taylor's expansion (of order $r\in\N_0$) for sufficiently smooth functions
$f:U\to\R$ on an open set $U\subset \R^N$ is given by
\be\label{taylor}
f(x)=\sum_{|\alpha|\leq r}{1\over{\alpha !}}{{\partial^\alpha f}\over{\partial x^\alpha}}(x_0)(x-x_0)^\alpha + \sum_{|\alpha|=r+1}
R_\alpha(x)(x-x_0)^\alpha
\ee
where the residual term (one of its variant forms) is
\[
R_\alpha(x):={{|\alpha|}\over{\alpha !}}\int_0^1(1-t)^{|\alpha|-1}
{{\partial^\alpha f}\over{\partial x^\alpha}}(x_0+t(x-x_0))dt.
\]
Recall also 
the definitions of Hadamard finite part integrals for discontinuous functions $f:[a,b]\to\R$ by \cite{martin-rizzo}, pp. 5 and 32, formulas (14) and (32) therein or \cite[p. 104]{schwarz}.
Applying these definitions (for a fixed $x$) to the function $F_x(t):=\chi_{[x,b]}(t)f(t)$, where $f\in C([a,b])$ and
where $\chi_{[x,b]}(t)$ is the characteristic function of the interval $[x,b]$, we have
\be \label{def-h1}
{\rm p.f.}\int_a^b{{F_x(t)}\over{t-x}}dt
={\rm p.f.}\int_x^{b}{{f(t)}\over{t-x}}dt
=\lim_{\epsilon\to 0}\Big(\int_{x+\epsilon}^{b}{{f(t)}\over{t-x}}dt
+f(x^+)\ln(\epsilon)\Big)
\ee
and
\be \label{def-h2}
{\rm p.f.}\int_a^b{{F_x(t)}\over{(t-x)^2}}dt
=&{\rm p.f.}\int_x^{b}{{f(t)}\over{(t-x)^2}}dt \nonumber \\
=&\lim_{\epsilon\to 0}\Big(\int_{x+\epsilon}^{b}{{f(t)}\over{(t-x)^2}}dt
+f'(x^+)\ln(\epsilon)-{1\over\epsilon}f(x^+)\Big).
\ee
Analogously we define
\be \label{def-h1-a}
{\rm p.f.}\int_a^{x}{{f(t)}\over{t-x}}dt
=\lim_{\epsilon\to 0}\Big(\int_a^{x-\epsilon}{{f(t)}\over{t-x}}dt
-f(x^-)\ln(\epsilon)\Big)
\ee
or equivalently
\[
{\rm p.f.}\int_a^{x}{{f(t)}\over{x-t}}dt
=\lim_{\epsilon\to 0}\Big(\int_a^{x-\epsilon}{{f(t)}\over{x-t}}dt
+f(x^-)\ln(\epsilon)\Big)
\] 
and
\be \label{def-h2-a}
{\rm p.f.}\int_a^{x}{{f(t)}\over{(t-x)^2}}dt 
=\lim_{\epsilon\to 0}\Big(\int_a^{x-\epsilon}{{f(t)}\over{(t-x)^2}}dt
-f'(x^-)\ln(\epsilon)-{1\over\epsilon}f(x^-)\Big)
\ee

In particular,
these formulas give
\be\label{hada1}
{\rm p.f.}\int_x^{b}{1\over{t-x}}dt
=\ln(b-x),
\ee
\be\label{hada2}
{\rm p.f.}\int_x^{b}{1\over{(t-x)^2}}dt
=-{1\over{b-x}}.
\ee
Note that ${\rm p.f.}\int_x^{b}{{f(t)}\over{t-x}}dt$ is well-defined (at least) for all $f\in C^\alpha([a,b]),\ \alpha>0$ and (cf. \cite{chan})
\be\label{c-0-a}
{\rm p.f.}\int_x^{b}{{f(t)}\over{t-x}}dt
=\int_{x}^{b}{{f(t)-f(x)}\over{t-x}}dt+
f(x)\ln({b-x}) .
\ee
We recall from \cite{tervo18-up}, Lemma 3.2

\begin{lemma}\label{hadale}
Suppose that $f\in C^2([a,b]\times [a,b])$.  Then for $x\in [a,b]$
\be\label{ch-id}
{d\over{dx}}\Big({\rm p.f.}\int_x^{b}{{f(x,t)}\over{t-x}}dt\Big)
={\rm p.f.}\int_x^{b}{{f(x,t)}\over{(t-x)^2}}dt+ 
{\rm p.f.}\int_x^{b}{{{\p f{x}}(x,t)}\over{t-x}}dt
-{\p f{t}}(x,x)
\ee
and
\be\label{ch-id-a}
{d\over{dx}}\Big({\rm p.f.}\int_a^{x}{{f(x,t)}\over{x-t}}dt\Big)
=-{\rm p.f.}\int_a^{x}{{f(x,t)}\over{(t-x)^2}}dt+ 
{\rm p.f.}\int_a^{x}{{{\p f{x}}(x,t)}\over{x-t}}dt
-{\p f{t}}(x,x).
\ee 
\end{lemma}

We remark that
the derivation result of the previous lemma \ref{hadale} is valid  for more general $f$ and that under appropriate assumptions $Hf\in W^{1,p}(]a,b[)$ where
\[
(Hf)(x):={\rm p.f.}\int_x^b{{f(x,t)}\over{t-x}}dt.
\]
For example, we have

\begin{lemma}\label{h1-lemma}
Suppose that $f\in C^{1+\alpha}([a,b]\times  [a,b]),\ \alpha>0$ and let 
\[
F_1(x):={{f(x,x)}\over{b-x}},\ F_2(x):=\ln(b-x){\p f{t}}(x,x),\
F_3(x):=\ln(b-x){\p f{x}}(x,x)
 .
 \]
Furthermore, suppose that
\be\label{h1-le-2}
F_j\in L^1(]a,b[),\ j=1,2,3.
\ee
Then  $Hf\in W^{1,1}(]a,b[)$ and 
\be\label{gen-hadale}
(Hf)'(x)=
{\rm p.f.}\int_x^{b}{{f(x,t)}\over{(t-x)^2}}dt+ 
{\rm p.f.}\int_x^{b}{{{\p f{x}}(x,t)}\over{t-x}}dt
-{\p f{t}}(x,x)
.
\ee
\end{lemma}

\begin{proof}
By the assumptions one sees (by the Taylor's formula) that the function
\[
h(x):= 
{\rm p.f.}\int_x^{b}{{f(x,t)}\over{(t-x)^2}}dt+ 
{\rm p.f.}\int_x^{b}{{{\p f{x}}(x,t)}\over{t-x}}dt
-{\p f{t}}(x,x)
\]
is in $L^1(]a,b[)$. Applying the techniques used in Lemmas \ref{ad-le-1} and \ref{ad-le-3-a} below one can verify that for $\varphi\in C_0^\infty(]a,b[)$
\[
\int_a^b (Hf)(x)\varphi'(x) dx=
-\int_a^b h(x)\varphi(x) dx.
\]
We omit the details of the proof.
\end{proof}

The assumption $f\in C^{1+\alpha}([a,b]\times  [a,b]),\ \alpha>0$ in the above lemma can be replaced with the weaker condition:

$f\in C^{1}([a,b]\times  [a,b]),\ \alpha>0$ such that 
\[
\Big|{\p f{t}}(x,t')-{\p f{t}}(x,t)\Big|\leq C|t'-t|^\alpha,\ t'\geq t.
\]

We finally mention that generally for $f\in L^2(]a,b[)$ the Hadamard finite part integrals  
\[
(H_jf)(x):={\rm p.f.}\int_x^b{{f(t)}\over{(t-x)^j}}dt,\ j=1,2
\]
are interpreted as distributions defined by 
\be\label{hada-dist}
(H_jf)(\varphi):=\int_a^b f(t)\Big({\rm p.f.}\int_a^t{{\varphi(x)}\over{(t-x)^j}}dx\Big)dt\ {\rm for}\ \varphi\in C_0^\infty(]a,b[).
\ee
Using the Taylor's expansion
\[
\varphi(x)=\varphi(t)+\varphi'(t)(t-x)+\int_0^1(1-s)\varphi''(t+s(t-x))ds\cdot (t-x)^2
\]
one sees that $H_2f$ defined by (\ref{hada-dist}) is really a distribution (and similarly  $H_1f$). Note that this generalization does not work for more general $f=f(x,t)$.

We need additionally the Taylor's expansion  for a sufficiently smooth function $f:S\to\R$.
Since $S$ is a manifold the expansion requires some explanation. The detailed presentation of the subject is outside of this paper and so we give only some essential technicalities. For some additional formulations see e.g. \cite{mukherjee}, pp. 185-186.

The first order Taylor's expansion of a function $f:S\to\R$ which is $C^3$-function around $\omega\in S$ is of the form
\be\label{tay-ex}
f(\omega')=f(\omega)+\la(\nabla_S f)(\omega),\zeta\ra +
(R_\omega f)(\omega')(\zeta,\zeta),\ \zeta\in T_\omega(S)
\ee
where $\nabla_S$ is the gradient on $S$ and  $R_\omega f$  is the residual.
The inner product $\la .,. \ra$ is the Riemannian inner product on $T_\omega(S)$ induced by the euclidean inner product on $\R^3$. In addition, there exists a constant $C\geq 0$ such that 
\be\label{tay-app}
\n{\zeta}\leq C\n{\omega'-\omega}.
\ee
Leaving the residue $R_\omega f$ away we get Taylor's approximations for $f(\omega')$ near $\omega$.

The basic principle in deriving (\ref{tay-ex}) is to apply an appropriate pull-back $H_\omega:V\to U_\omega$
where $U_\omega\subset S$ and $V\subset T_\omega(S)$ are open neighbourhoods such that $\omega\in U_\omega,\ 0\in V$ and $H_\omega(0)=\omega$. One assumes that $H_\omega$
is a sufficiently smooth diffeomorphism and so the (smooth) inverse mapping $H_\omega^{-1}:U_\omega\to V$ exists. Let $\omega'\in U_\omega$ and let
\[
H_\omega^{-1}(\omega')=\zeta=\xi_1\ol\Omega_1'+\xi_2\ol\Omega_2'
\]
where
$\ol\Omega_1'=\ol\Omega_1(\omega'),\ \ol\Omega_2'=\ol\Omega_2(\omega')$. Here $\ol\Omega_1,\ \ol\Omega_2$ are the (locally defined) canonical tangent vectors of $S$ at $\omega\in S$ that is,
\[
\ol\Omega_1={1\over{\sqrt{\omega_1^2+\omega_2^2}}}(-\omega_1,\omega_2,0),
\]
\[
\ol\Omega_2=\big({{\omega_1\omega_3}\over{\sqrt{\omega_1^2+\omega_2^2}}},
{{\omega_2\omega_3}\over{\sqrt{\omega_1^2+\omega_2^2}}}.
-\sqrt{1-\omega_3^2}\big),\ \sqrt{\omega_1^2+\omega_2^2}\not=0.
\]
One often chooses $H_\omega$ to be the (differential geometry's) exponential mapping $H_\omega=\exp_\omega$ (for exponential mapping see \cite{carmo} and the Example \ref{ex-map} below). The pull-back obeys (at least locally)
\be\label{e4}
\n{H_\omega^{-1}(\omega')}=\n{\zeta}\leq C\n{\omega'-\omega}.
\ee

The tangent space $T_\omega(S)$ can be isomorphically (and isometrically) identified with $\R^2$ by
\[
\zeta=\xi_1\ol\Omega_1+\xi_2\ol\Omega_2\sim (\xi_1,\xi_2)=:\xi.
\]
In fact we can define
\[
J(\zeta)=J(\xi_1\ol\Omega_1+\xi_2\ol\Omega_2)= (\xi_1,\xi_2).
\]
Let $J(V)=V'\subset\R^2$.
Using this identification we find that the mapping $f\circ H_\omega:V'\to \R$ is well-defined (and as smooth as $f$). Note that actually
\[
(f\circ H_\omega)(\xi)=(f\circ H_\omega)(\xi_1\ol\Omega_1+\xi_2\ol\Omega_2).
\]
Hence we are able to write the Taylor's expansion (the expansion (\ref{taylor}) with $r=1$) near $0\in\R^2$ 
\bea\label{t-ex}
&
(f\circ H_\omega)(\zeta)=(f\circ H_\omega)(\xi)=(f\circ H_\omega)(0)
+\sum_{j=1}^2\partial_j(f\circ H_\omega)(0)\xi_j
\nonumber\\
&
+\sum_{|\alpha|=2}{{|\alpha|}\over{\alpha !}}\int_0^1(1-t)\partial_\xi^\alpha
(f\circ H_\omega)(t\xi)dt \cdot \xi^\alpha. 
\eea
By (\ref{e4}) for $\zeta=H_\omega^{-1}(\omega')$
\be\label{e7}
\n{\xi}=\n{\zeta}\leq C\n{\omega'-\omega}.
\ee
Finally,  one has
\be\label{nabla-a}
\partial_j(f\circ H_\omega)(0)=(\partial_{\omega_j}f)(\omega)
\ee 
and
\be\label{e8}
\sum_{j=1}^2\partial_j(f\circ H_\omega)(0)\xi_j
=\sum_{|\alpha|=1}{1\over{\alpha !}}\partial_\xi^\alpha(f\circ H_\omega)(0)\xi^\alpha=\la(\nabla_S f)(\omega),\zeta\ra .
\ee
and so (\ref{tay-ex}) can be seen.

The gradient $\nabla_S f$ on sphere $S$ can be shortly  depicted as follows (for $n=3$). 
Suppose that $f$ is defined and smooth in a neighbourhood of $S\subset\R^3$. Then
\be\label{gradS}
\nabla_S f=\la\nabla f,\ol\Omega_1\ra\ol\Omega_1+
\la\nabla f,\ol\Omega_2\ra\ol\Omega_2
\ee
where $\nabla f$ is the gradient of $f$ in the ambient space $\R^3$. Note that the right hand side of (\ref{gradS}) is the projection of $\nabla f$ onto the tangent space.
Hence
\be\label{nab0} 
\omega\cdot\nabla_S f=0.
\ee

The first order Taylor's approximation of $\psi$ with respect to $\omega$ around $\omega$ is
\be\label{ess-app-1}
\psi(x,\omega',E)\approx \psi(x,\omega,E)
+\la(\nabla_{S}\psi)(x,\omega,E),\zeta\ra 
\ee
where $\zeta=H_\omega^{-1}(\omega')\in T_\omega(S)$ and it satisfies by (\ref{e7})
\be\label{tay-app-a}
\n{\zeta}\leq C\n{\omega'-\omega}.
\ee
The residual, for example, for the approximation (\ref{ess-app-1}) is by (\ref{t-ex})
\be\label{t-residual}
R_\omega(\psi(x,.,E))(\omega')(\zeta,\zeta)
=
\sum_{|\alpha|=2}{{|\alpha|}\over{\alpha !}}\int_0^1(1-t)\partial_\xi^\alpha
(\psi(x,.,E)\circ H_\omega)(t\xi)dt \cdot \xi^\alpha. 
\ee
where $\xi=J(\zeta)$.

\begin{remark}\label{inner-pr}

Let $\la .,.\ra_r$ be the Riemannian inner product on $T(S)$ that is,
\[
\la \zeta_1,\zeta_2\ra_r =\xi_1^1\xi_1^2+\xi_2^1\xi_2^2
\]
for $\zeta_j=\xi_j^1\ol\Omega_1(\omega)+\xi_j^2\ol\Omega_2(\omega)\in T_{\omega}(S)$.
Furthermore, let $a_{{\rm pr},\omega},\ \omega\in S$ be the projection of a vector $a\in\R^3$ onto $T_{\omega}(S)$ that is,
\[
a_{{\rm pr},\omega}:=
\sum_{j=1}^2\la a,\ol\Omega_j(\omega)\ra\ol\Omega_j(\omega).
\]
Then that for $\zeta\in T_{\omega}(S)$ and $a\in\R^3$
\be\label{same}
\la \zeta,a_{{\rm pr},\omega}\ra_r=\la \zeta,a\ra_{\R^3}=\la \zeta,a\ra.
\ee
In the sequel the inner product $\la \zeta,a\ra$ is interpreted by (\ref{same}).
\end{remark}

\begin{example}\label{ex-map}

Recall that the geodesics of the sphere $S=S_2$ are great circles. 
Let $\omega\in S$ and let $\zeta=\xi_1\ol\Omega_1+\xi_2\ol\Omega_1\in T_\omega(S)$ be the tangent vector of $S$ at $\omega$. Then there exists a geodesic $\gamma_\omega:]-2,2[\to S$ which satisfies the initial conditions 
\[
\gamma_\omega(0)=\omega,\ \gamma_\omega'(0)=\zeta.
\]
In fact in the case of $S$ 
\be\label{geo-s}
\gamma_\omega(t)=\cos(\n{\zeta}t)\omega+\sin(\n{\zeta}t){\zeta\over{\n{\zeta}}}.
\ee

The differential geometry's exponential mapping $\exp_\omega:T_\omega(S)\to S$ at $\omega$ is defined by
\[
\exp_\omega(\zeta):=\gamma_\omega(1).
\]
Note that the exponential map is dependent on the parametrization $\gamma_\omega$ of the geodesic. Its basic properties are, however independent 
of $\gamma_\omega$. 
It can be shown that there exist open neighbourhoods $U_\omega\subset S$ and $V\subset T_\omega(S)$ such that $\omega\in U_\omega,\ 0\in V$  and for which
\[
\exp_\omega:V\to U_\omega
\]
is a diffeomorphism. Hence the (smooth) inverse mapping $\exp_\omega^{-1}:U_\omega\to V$ exists. Let $\omega'\in U_\omega$ and let
\[
\exp_\omega^{-1}(\omega')=\zeta=\xi_1\ol\Omega_1+\xi_2\ol\Omega_2.
\]
From (\ref{geo-s}) we immediately get that for the sphere the exponential map using the above parametrization is
\be\label{exp-s}
\exp_\omega(\zeta)=
\gamma_\omega(1)=\cos(\n{\zeta})\omega+\sin(\n{\zeta}){\zeta\over{\n{\zeta}}}.
\ee
To see that $\exp_\omega(\zeta)=\exp_\omega(\xi_1\ol\Omega_1+\xi_2\ol\Omega_2)$ has the needed differentiability properties, apply Taylor's expansions
\[
\cos(\n{\zeta})=1+{1\over{2!}}\n{\zeta}^2+\cdots\
 {\rm and}\ \sin(\n{\zeta)}=\n{\zeta}-{1\over{3!}}\n{\zeta}^3+\cdots\ .
\]

The inverse $\exp_\omega^{-1}$ can be computed explicitly. In fact, let 
$\exp_\omega(\zeta)=\omega'$ that is,
\be\label{e0}
\cos(\n{\zeta})\omega+\sin(\n{\zeta}){\zeta\over{\n{\zeta}}}=\omega'.
\ee
Since $\omega\perp \zeta$ we find that
\be\label{e1}
\cos(\n{\zeta})=\la\omega',\omega\ra ,\ \sin(\n{\zeta})=\sqrt{1-\la\omega',\omega\ra^2}
\ee
and then
\be\label{e2}
\n{\zeta}={\rm \ol{arc}cos}(\la\omega',\omega\ra ).
\ee
From (\ref{e0}), (\ref{e1}) and (\ref{e2}) we obtain
\be\label{e3}
\exp_\omega^{-1}(\omega')=\zeta={{{\rm \ol{arc}cos}(\la\omega',\omega\ra}\over{
\sqrt{1-\la\omega',\omega\ra^2}}}(\omega'-\la\omega',\omega\ra\omega).
\ee

We find that
there exists a constant $C>0$ such that
\be\label{e4-a}
\n{\exp_\omega^{-1}(\omega')}=\n{\zeta}\leq C\n{\omega'-\omega}.
\ee
Actually, 
from the Hospital's rule it follows that
\[
\lim_{x\to 1}{{{\rm \ol{arc}cos(x)}}\over{\sqrt{1-x^2}}}=1
\]
and so (since $S\times S$ is compact) there exists a constant $C'$ such that 
\be\label{e5}
\Big|{{\rm \ol{arc}cos(\la\omega',\omega\ra}\over{
\sqrt{1-\la\omega',\omega\ra^2}}}\Big|\leq C'.
\ee
Furthermore, we have
\be\label{e6}
\omega'-\la\omega',\omega\ra\omega=\omega'-\la\omega'-\omega,\omega\ra\omega
-\omega.
\ee
Hence the assertion follows from (\ref{e3}), (\ref{e5}), (\ref{e6}).

\end{example}

\section{ On hyper-singular collision operators related to charged particle transport}\label{h-s-c-op}

The differential cross-sections
may have singularities, or even hyper-singularities, which would lead to extra  pseudo-differential-like terms in the transport equation.
In the case where $\sigma(x,\omega',\omega,E',E)$ has hyper-singularities (like in the case of M\o ller-type  differential cross 
sections analysed below) the integral $\int_{I'}$ occurring in the collision operator must be understood in the 
sense of {\it Cauchy principal value} ${\rm p.v.}\int_{I'}$ or more generally in the sense of {\it Hadamard finite part integral} ${\rm p.f.}\int_{I'}$ (\cite[Sec. 3.2]{hsiao},  \cite[pp. 104-105]{schwarz}, \cite{estrada}, sections 1.5 and 1.6). Hyper-singular integral operators form a subclass of pseudo-differential operators (\cite{hsiao}, Chapter 7, \cite{lifanov}, Chapter 3).
Moreover, the $(\omega',\omega)$-dependence in differential cross-sections typically contain Dirac's $\delta$-distributions (on $\R$). 
More precisely, in  $\sigma(x,\omega'\omega,E',E)$ there may occur terms
like $\delta(\omega\cdot\omega'-\mu(E',E))$ or $\delta(E-E')$ which require special treatment.

Consider the following partial singular integral operator,
\be\label{coll-1}
(K\psi)(x,\omega,E)={\rm p.f.}\int_{I'}\int_{S'}\sigma(x,\omega',\omega,E',E)\psi(x,\omega',E')d\omega' dE'.
\ee
The simplest is the case  where  $\sigma$ has at most a so-called \emph{weak singularity} with respect to energy. This means that
$\sigma=\sigma_0(x,\omega',\omega,E',E)$ 
is a measurable non-negative function $G\times S\times S\times (I\times I\setminus D)\to\R$,
where $D=\{(E,E)\ |\ E\in I\}$ is the diagonal of $I\times I$,
obeying for $E\neq E'$ the estimates
\bea
&
{\rm ess\ sup}_{(x,\omega)}\int_{S'}\sigma_0(x,\omega',\omega,E',E)d\omega'\leq  {C\over{|E-E'|^\kappa}},
\label{coll-3} \\
&
{\rm ess\ sup}_{(x,\omega)}\int_{S'}\sigma_0(x,\omega,\omega',E,E')d\omega'\leq  {C\over{|E-E'|^\kappa}},
\label{coll-3a}
\eea
where $\kappa<1$.
The corresponding collision operator 
\be 
(K\psi)(x,\omega,E)
=
\int_{I'}\int_{S'}\sigma_0(x,\omega',\omega,E',E)\psi(x,\omega',E')d\omega' dE',
\ee
is a usual partial Schur  integral operator that is, $\sigma_0(x,\omega',\omega,E',E)$ satisfies the Schur criterion for the boundedness (\cite{halmos}, p. 22)
and so $K$ is a  bounded operator $L^2(G\times S\times I)\to L^2(G\times S\times I)$ (see \cite{tervo18-up}, section 5).  

Nevertheless, the  collision operator $K$ is not generally of the above form. $(E',E)$-dependence in differential cross section 
$\sigma(x,\omega',\omega,E',E)$ may contain hyper-singularities of higher order, such as
${1\over{(E'-E)}^j}$, for $j=1,2$. 
Below we shall consider in more detail the M\o ller scattering. We remark that  the analysis e.g. for Bremsstrahlung goes quite similarly (but it is more simple).

\subsection{M\o ller scattering as a prototype for hyper-singular collision operators}\label{moller}

So called M\o ller scattering is a kind of prototype of interactions of charged particles leading to hyper-singular integral operators. Hence we express our analysis in the frames of it.
In \cite{tervo18-up} we verified that the cross section $\sigma$ for the 
M\o ller interaction is of the form

\begin{multline}
\sigma(x,\omega',\omega,E',E)
=\chi(E',E)\Big(
{1\over{(E'-E)^2}}\sigma_2(x,\omega',\omega,E',E)\\
-{1\over{E'-E}}\sigma_1(x,\omega',\omega,E',E)+\sigma_0(x,\omega',\omega,E',E)\Big)
\label{coll-2}
\end{multline}
where 
\[
\chi(E',E):=\chi_{\R_+}(E-E_0)\chi_{\R_+}(E_m-E)\chi_{\R_+}(E'-E).
\]
Here each of $\sigma_j(x,\omega',\omega,E',E)$, $j=0,1,2$ may contain the above mentioned $\delta$-distributions,
and hence they are not necessarily measurable functions on $G\times S\times S\times I\times I$.
Denote for $j=0,1,2$,
\[
(\ol {\s K}_j\psi)(x,\omega,E',E):={}&\int_{S'}\sigma_j(x,\omega',\omega,E',E)\psi(x,\omega',E') d\omega', \\[2mm]
(\widehat {\s K}_j\psi)(x,\omega,E',E):={}&\chi(E',E)(\ol {\s K}_j\psi)(x,\omega,E',E).
\]
Here the integral $\int_{S}$ is originally interpreted as a distribution. 
However, it can be shown (\cite{tervo18-up}, section 3.2)   that $\ol{\s K}_j$ is of the form 
\bea\label{eq:ol_s_K_22_j}
(\ol {\s K}_{j}\psi)(x,\omega,E',E)
=
{}&
\hat\sigma_{j}(x,E',E)
\int_{S'}
\delta(\omega'\cdot\omega - \mu(E,E'))\psi(x,\omega',E')d\omega'
\nonumber\\
={}&\hat{\sigma}_{j}(x,E',E)\int_{0}^{2\pi}\psi(x,\gamma(E',E,\omega)(s),E')ds,
\eea
where $\mu(E',E):=\sqrt{{{E(E'+2)}\over{E'(E+2)}}}$ and
where $\gamma=\gamma(E',E,\omega):[0,2\pi]\to S$
is a parametrization of the curve
\[
\Gamma(E',E,\omega)=\{\omega'\in S\ |\ \omega'\cdot\omega-\mu(E',E)=0\}.
\]
For example,  we can choose
\be
\gamma(E',E,\omega)(s)=R(\omega)\big(\sqrt{1-\mu^2}\cos(s),\sqrt{1-\mu^2}\sin(s),\mu\big),\quad s\in [0,2\pi],
\ee
where $\mu=\mu(E',E)$, and $R(\omega)$ is any rotation (unitary) matrix which maps the vector $e_3=(0,0,1)$ into $\omega$. We choose
\[
R(\omega)=\qmatrix{{{\omega_1\omega_3}\over{\sqrt{\omega_1^2+\omega_2^2}}}&
{{-\omega_1}\over{\sqrt{\omega_1^2+\omega_2^2}}} &\omega_1\\
{{\omega_2\omega_3}\over{\sqrt{\omega_1^2+\omega_2^2}}}&
{{\omega_2}\over{\sqrt{\omega_1^2+\omega_2^2}}} &\omega_2\\   
-\sqrt{1-\omega_3^2}&0&\omega_3}    
=\qmatrix{\ol\Omega_2&\ol\Omega_1&\omega}.
\]

Combining the above treatments we found 
in \cite{tervo18-up}, section 3.2  that $K$  is of the
form 
\begin{multline}
({K}\psi)(x,\omega,E)
=
{\s H}_2\big((\ol {\s K}_2\psi)(x,\omega,\cdot,E)\big)(E) 
\\
-
{\s H}_1\big((\ol {\s K}_1\psi)(x,\omega,\cdot,E)\big)(E)
+\int_{I}(\widehat {\s K}_0\psi)(x,\omega,E',E) dE',
\label{co-bb}
\end{multline}
where ${\s H}_j$, $j=1,2$, are the {\it Hadamard finite part integral operators} with respect to $E'$-variable defined by
\[
({\s H}_j u)(E):={\rm p.f.}\int_{E}^{E_m}{1\over{(E'-E)^j}}u(E')dE'.
\]
The expression \eqref{co-bb} is the {\it hyper-singular integral operator form} of $K$.

Moreover, we in \cite{tervo18-up} verified that \eqref{co-bb} can be equivalently given in the "\emph{pseudo-differential operator-like form}" by
\bea\label{co-cc}
({K}\psi)(x,\omega,E)
={}&
{\partial\over{\partial E}}\Big(
{\s H}_1\big((\ol{\s K}_2\psi)(x,\omega,\cdot,E)\big)(E)\Big)
-
{\s H}_1\big(({\p {(\ol{\s K}_2\psi)}E}(x,\omega,\cdot,E)\big)(E)
\nonumber\\
{}&
+
{\p {(\ol{\s K}_2\psi)}{E'}}(x,\omega,E',E)_{|E'=E}\nonumber\\
{}&
-
{\s H}_1\big((\ol{\s K}_1\psi)(x,\omega,\cdot,E)\big)(E)
+\int_{I}(\widehat{\s K}_0\psi)(x,\omega,E',E) dE'
\eea
where only the "first-order" Hadamard finite part integral operator ${\s H}_1$ appears.  
We neglect the details but we recall that the derivation of (\ref{co-cc}) founded on the use of Lemma \ref{hadale}.

As a conclusion we see that some 
interactions produce the first-order partial derivatives
with respect to energy $E$ combined with the Hadamard part operator (which is
a pseudo-differential-like operator; see Remark 3.3 in \cite{tervo18-up}).
For instance, in  dose calculation (radiation therapy) these problematic interactions are the primary electron-electron, primary positron-positron collisions and Bremsstrahlung. 
The exact transport operator for M\o ller scattering is
\bea\label{ex+1}
&
(T\psi)(x,\omega,E):=
-{\s H}_2\big((\ol{\s K}_2\psi)(x,\omega,\cdot,E)\big)(E) 
+
{\s H}_1\big((\ol{\s K}_1\psi)(x,\omega,\cdot,E)\big)(E)\nonumber\\
&
+\omega\cdot\nabla_x\psi+\Sigma(x,\omega,E)\psi
-(K_r\psi)(x,\omega,E)
\eea
where 
\[
(K_r\psi)(x,\omega,E):=
\int_{I'}(\widehat{\s K}_0\psi)(x,\omega,E',E) dE'.
\]
$K_r$ is called a \emph{restricted collision operator} which (by the Schur criterion) is a bounded operator $L^2(G\times S\times I)\to L^2(G\times S\times I)$. Basic properties of (more general) restricted collision operators 
are exposed in \cite{tervo17} and more widely in \cite{tervo18-up}, section 5.4.
Equivalently, in virtue of (\ref{co-cc}) $T$ can be given by
\bea\label{ex-0}
&
(T\psi)(x,\omega,E):=
-{\partial\over{\partial E}}\Big(
{\s H}_1\big((\ol{\s K}_2\psi)(x,\omega,\cdot,E)\big)(E)\Big)
+
{\s H}_1\big(({\p {(\ol{\s K}_2\psi)}E}(x,\omega,\cdot,E)\big)(E)
\nonumber\\
{}&
-
{\p {(\ol{\s K}_2\psi)}{E'}}(x,\omega,E',E)_{\big|E'=E}\nonumber\\
{}&
+\omega\cdot\nabla_x\psi+\Sigma(x,\omega,E)\psi
+
{\s H}_1\big((\ol{\s K}_1\psi)(x,\omega,\cdot,E)\big)(E)
-(K_r\psi)(x,\omega,E).
\eea
We finally mention that 
the pseudo-differential-like parts can be approximated 
by partial differential operators
(\cite{tervo18-up}, section 4) to obtain pure partial integro-differential operator approximations, so called continuous slowing down approximations (CSDA).

\section{Exact transport operator and formal adjoints}\label{exact-eq}

We shall give more details about the transport operator (\ref{ex+1}) or (\ref{ex-0}). The obtained refined expression reveals that hyper-singular transport operators generate partial differential terms with respect to angle and energy variables as well.
We start by verifying some essential tools for Hadamard finite part integrals.
The below Lemmas are needed frequently in the rest of the paper.

\subsection{Auxiliary lemmas for Hadamard finite part integrals}\label{var-for-1}

Suppose that $f\in C(\ol G,C^\alpha(I^2))$ where $\alpha>0$. Then we find  that
\[
{\rm p.f.}\int_E^{E_m}{1\over{E'-E}}f(x,E',E)dE'
=
\int_E^{E_m}{{f(x,E',E)-f(x,E,E)}\over{E'-E}}dE'
+f(x,E.E)\ln(E_m-E)
\]
since ${\rm p.f.}\int_E^{E_m}{1\over{E'-E}}dE'=\ln(E_m-E)$. 
Noting that $\int_{I}|\ln(E_m-E)|dE<\infty$ and that
\[
\Big|\int_E^{E_m}{{f(x,E',E)-f(x,E,E)}\over{E'-E}}dE'\Big|
\leq 
\int_E^{E_m}{{C_\alpha |E'-E|^\alpha}\over{|E'-E|}}dE'=C_\alpha {1\over\alpha}(E_m-E)^\alpha
\]
where $C_\alpha:=\n{f}_{C(\ol G,C^\alpha(I^2))}$
we see that the function 
\[
E\to {\rm p.f.}\int_E^{E_m}{1\over{E'-E}}f(x,E',E)dE'
\]
is integrable on $I$ (in the sense of ordinary improper Riemann integrals). Analogously we see that the function 
\[
E'\to {\rm p.f.}\int_{E_0}^{E'}{1\over{E'-E}}f(x,E',E)dE
\]
is integrable on $I'$.

We start with the next Fubin-type lemma

\begin{lemma}\label{ad-le-1}
For $f\in C(\ol G,C^\alpha( I^2)),\ \alpha>0$ we have
\be\label{vf-5}
\int_I\Big({\rm p.f.}\int_E^{E_m}{1\over{E'-E}}f(x,E',E)dE'\Big)dE
=
\int_{I'}\Big({\rm p.f.}\int_{E_0}^{E'}{1\over{E'-E}}f(x,E',E)dE\Big)dE'
\ee
\end{lemma}

\begin{proof}
Define integrals
\[
I_\epsilon^1:=
\int_I\Big(\int_{E+\epsilon}^{E_m}{1\over{E'-E}}f(x,E',E)dE'+f(x,E,E)\ln(\epsilon)\Big)dE,
\]
\[
I_\epsilon^2:=
\int_{I'}\Big(\int_{E_0}^{E'-\epsilon}{1\over{E'-E}}f(x,E',E)dE
+f(x,E',E')\ln(\epsilon)\Big)dE'.
\]
By the Fubin's Theorem
\bea\label{vf-9}
&
I_\epsilon^1:=
\int_I\Big(\int_{I'}\chi_{\R_+}(E'-(E+\epsilon)){1\over{E'-E}}f(x,E',E)dE'\Big)dE+\int_If(x,E,E)\ln(\epsilon)dE\nonumber\\
&
=
\int_{I'}\Big(\int_{I}\chi_{\R_+}(E'-(E+\epsilon)){1\over{E'-E}}f(x,E',E)dE\Big)dE'+\int_{I}f(x,E,E)\ln(\epsilon)dE\nonumber\\
&
=
\int_{I'}\Big(\int_{E_0}^{E'-\epsilon}{1\over{E'-E}}f(x,E',E)dE\Big)dE'+\int_{I'}f(x,E',E')\ln(\epsilon)dE'=:I_\epsilon^2.\nonumber\\
\eea

Let
\[
F_\epsilon^1(x,E):=
\int_{E+\epsilon}^{E_m}{1\over{E'-E}}f(x,E',E)dE'+f(x,E,E)\ln(\epsilon).
\]
Then we have
\[
&
F_\epsilon^1(x,E):=
\int_{E+\epsilon}^{E_m}{{f(x,E',E)-f(x,E,E)}\over{E'-E}}dE'
+\int_{E+\epsilon}^{E_m}{{f(x,E,E)}\over{E'-E}}dE'
+f(x,E,E)\ln(\epsilon)\\
&
=\int_{E+\epsilon}^{E_m}{{f(x,E',E)-f(x,E,E)}\over{E'-E}}dE'
+f(x,E,E)\ln(E_m-E)
\]
since $\int_{E+\epsilon}^{E_m}{1\over{E'-E}}dE'=\ln(E_m-E)-\ln(\epsilon)$.
Noting that (here $C_\alpha$ is as above)
\[
\Big|{{f(x,E',E)-f(x,E,E)}\over{E'-E}}\Big|
\leq C_\alpha |E'-E|^{\alpha-1}
\]
we find that 
\bea\label{vf-10}
&
|F_\epsilon^1(x,E)|\leq 
\int_{E+\epsilon}^{E_m}
C_\alpha|E'-E|^{\alpha-1}dE'+|f(x,E,E)\ln(E_m-E)|
\nonumber\\
&
=C_\alpha{1\over\alpha}\big((E_m-E)^\alpha-\epsilon^\alpha\big)+|f(x,E,E)\ln(E_m-E)|.
\eea
Hence the sequence $\{F_\epsilon^1(x,.)\}$ is bounded by an integrable function.
By the definition of the Hadamard finite part integrals
\be\label{vf-11}
\lim_{\epsilon\to 0}F_\epsilon^1(x,E)
={\rm p.f.}\int_E^{E_m}{1\over{E'-E}}f(x,E',E)dE'
\ee
Similarly we see that the sequence of functions
\[
F_\epsilon^2(x,E'):=
\int_{E_0}^{E'-\epsilon}{1\over{E'-E}}f(x,E',E)dE
+f(x,E',E')\ln(\epsilon)
\]
is bounded by an integrable function 
and
\be\label{vf-12}
\lim_{\epsilon\to 0}F_\epsilon^2(x,E')
={\rm p.f.}\int_{E_0}^{E'}{1\over{E'-E}}f(x,E',E)dE.
\ee

In virtue of (\ref{vf-9})
\be\label{vf-13}
\int_IF_\epsilon^1(x,E)dE=I_\epsilon^1=I_\epsilon^2=\int_{I'}F_\epsilon^2(x,E')dE'.
\ee
The Lebesgue Dominated Convergence Theorem implies by (\ref{vf-11}), (\ref{vf-12}), (\ref{vf-13}) that
\bea\label{vf-14-a}
&
\int_I\Big({\rm p.f.}\int_E^{E_m}{1\over{E'-E}}f(x,E',E)dE'\Big)dE
=\int_I\lim_{\epsilon\to 0}F_\epsilon^1(x,E)dE\nonumber\\
&
=
\lim_{\epsilon\to 0}
\int_IF_\epsilon^1(x,E)dE=\lim_{\epsilon\to 0}\int_{I'}F_\epsilon^2(x,E')dE'
\nonumber\\
&
=\int_{I'}\lim_{\epsilon\to 0}F_\epsilon^2(x,E')dE'
=
\int_{I'}\Big({\rm p.f.}\int_{E_0}^{E'}{1\over{E'-E}}f(x,E',E)dE'\Big)dE'
\eea
which completes the proof.
\end{proof}

We prove also the next partial integration-like results.

\begin{lemma}\label{ad-le-3-a}
Suppose that $f\in C(\ol G\times I,C^{\alpha}( I')),\ \alpha>0$.
Then   
\bea\label{adle3-1-a}
{\rm p.f.}\int_E^{E_m}{1\over{E'-E}}&f(x,E',E)dE'
\nonumber\\
&
=
-
\int_E^{E_m}\ln(E'-E){\p f{E'}}(x,E',E)dE'
+\ln(E_m-E)f(x,E_m,E).
\eea 
\end{lemma} 

\begin{proof}
Using partial integration we have
\bea\label{adle3-2-a}
&
\int_{E+\epsilon}^{E_m}{1\over{E'-E}}f(x,E',E)dE'
+\ln(\epsilon)f(x,E,E)
\nonumber\\
&
=
\Big|_{E+\epsilon}^{E_m}\ln(E'-E)f(x,E',E)
-
\int_{E+\epsilon}^{E_m}\ln(E'-E){\p f{E'}}(x,E',E)dE'
+\ln(\epsilon)f(x,E,E)
\nonumber\\
&
=
-
\int_{E+\epsilon}^{E_m}\ln(E'-E){\p f{E'}}(x,E',E)dE'
+\ln(E_m-E)f(x,E_m,E)
\eea
from which the assertion follows by letting $\epsilon\to 0$.

\end{proof}

\begin{lemma}\label{ad-le-3}
Suppose that $f\in C(\ol G\times I,C^{1+\alpha}( I')),\ \alpha>0$.
Then   for $E\not=E_m$
\bea\label{adle3-1}
&
{\rm p.f.}\int_E^{E_m}{1\over{(E'-E)^2}}f(x,E',E)dE'
\nonumber\\
&
=
{\rm p.f.}\int_E^{E_m}{1\over{E'-E}}{\p f{E'}}(x,E',E)dE'+{\p f{E'}}(x,E,E)_+
-{1\over{E_m-E}}f(x,E_m,E)
\eea
where ${\p f{E'}}(x,E,E)_+$ denotes the right hand partial derivative (actually here ${\p f{E'}}(x,E,E)_+={\p f{E'}}(x,E,E)$). 
\end{lemma} 

\begin{proof}
Using partial integration we have for $\epsilon>0$
\bea\label{adle3-2}
&
\int_{E+\epsilon}^{E_m}{1\over{(E'-E)^2}}f(x,E',E)dE'
-{1\over\epsilon}f(x,E,E)+\ln(\epsilon){\p f{E'}}(x,E,E)
\nonumber\\
&
=
\Big|_{E+\epsilon}^{E_m}-{1\over{E'-E}}f(x,E',E)
+
\int_{E+\epsilon}^{E_m}{1\over{E'-E}}{\p f{E'}}(x,E',E)dE'
\nonumber\\
&
-{1\over\epsilon}f(x,E,E)+\ln(\epsilon){\p f{E'}}(x,E,E)
\nonumber\\
&
=
\Big(\int_{E+\epsilon}^{E_m}{1\over{E'-E}}{\p f{E'}}(x,E',E)dE'
+\ln(\epsilon){\p f{E'}}(x,E,E)\Big)
\nonumber\\
&
+{{f(x,E+\epsilon,E)-f(x,E,E)}\over\epsilon}-{1\over{E_m-E}}f(x,E_m,E)
\eea
from which the assertion follows by letting $\epsilon\to 0^+$.

\end{proof}

Similarly we have 

\begin{lemma}\label{ad-le-3-aa}
Suppose that $f\in C(\ol G\times I',C^{1+\alpha}( I)),\ \alpha>0$.
Then   for $E\not=E_0$
\bea\label{adle3-1-b}
&
{\rm p.f.}\int_{E_0}^{E'}{1\over{(E'-E)^2}}f(x,E',E)dE
\nonumber\\
&
=
-{\rm p.f.}\int_{E_0}^{E'}{1\over{E'-E}}{\p f{E}}(x,E',E)dE-{\p f{E}}(x,E',E')_-
-{1\over{E'-E_0}}f(x,E',E_0)
\eea
where ${\p f{E}}(x,E',E')_-$ denotes the left hand partial derivative. 
\end{lemma}

The previous lemma \ref{ad-le-3}  immediately gives

\begin{lemma}\label{ad-le-4}
Suppose that $f\in C(\ol G,C^{1+\alpha}(I^2))$. Then
\bea\label{adth-1}
&
\int_I\Big({\rm p.f.}\int_E^{E_m}{1\over{(E'-E)^2}}f(x,E',E)dE'\Big)dE
=
\int_{I'}\Big({\rm p.f.}\int_{E_0}^{E'}{1\over{E'-E}}{\p f{E'}}(x,E',E)dE\Big)dE'
\nonumber\\
&
+\int_{I'}{\p f{E'}}(x,E',E')_+dE'
-\int_{I'}{1\over{E_m-E'}}f(x,E_m,E')dE'
\eea
when the integral $\int_{I'}{1\over{E_m-E'}}f(x,E_m,E')dE'$ exists.

\end{lemma}

\begin{proof}
In virtue of Lemma \ref{ad-le-3}
\bea\label{adth-2}
&
\int_I\Big({\rm p.f.}\int_E^{E_m}{1\over{(E'-E)^2}}f(x,E',E)dE'\Big)dE
=
\int_I\Big({\rm p.f.}\int_E^{E_m}{1\over{E'-E}}{\p f{E'}}(x,E',E)dE'\Big)dE
\nonumber\\
&
+\int_I{\p f{E'}}(x,E,E)_+dE
-\int_I{1\over{E_m-E}}f(x,E_m,E)dE.
\eea
Hence the assertion follows from Lemma \ref{ad-le-1}.

\end{proof}

Moreover, we have the following Fubin-type theorem

\begin{lemma}\label{ad-le-5}
For $f\in C(\ol G\times S\times I,C^{1+\alpha}(I')),\ \alpha>0$ and for $E\not=E_m$
\be\label{var-for-9}
\int_S\Big({\rm p.f.}\int_E^{E_m}{1\over{(E'-E)^2}}f(x,E',E,\omega)dE'\Big)d\omega
=
{\rm p.f.}\int_E^{E_m}{1\over{(E'-E)^2}}\Big(\int_Sf(x,E',E,\omega)d\omega\Big)dE'.
\ee
\end{lemma}

\begin{proof}
By Fubin's Theorem we have
\bea\label{var-for-10}
&
\int_S\Big[\int_{E+\epsilon}^{E_m}{1\over{(E'-E)^2}}f(x,E',E,\omega)dE'
-{1\over\epsilon}f(x,E,E,\omega)+\ln(\epsilon){\p f{E'}}(x,E,E,\omega)\Big]d\omega
\nonumber\\
&
=
\int_{E+\epsilon}^{E_m}{1\over{(E'-E)^2}}\Big(\int_Sf(x,E',E,\omega)d\omega\Big)dE'
\nonumber\\
&
-{1\over\epsilon}\int_Sf(x,E,E,\omega)d\omega
+\ln(\epsilon)\int_S{\p f{E'}}(x,E,E,\omega)d\omega\nonumber\\
&
=
\int_{E+\epsilon}^{E_m}{1\over{(E'-E)^2}}\Big(\int_Sf(x,E',E,\omega)d\omega\Big)dE'
\nonumber\\
&
-{1\over\epsilon}\int_Sf(x,E,E,\omega)d\omega
+\ln(\epsilon)
{\partial\over{\partial E'}}\Big(\int_S f(x,.,E,\omega)d\omega\Big)_{|E'=E}.
\eea
Denote
\[
F_\epsilon(x,E,\omega):=
\int_{E+\epsilon}^{E_m}{1\over{(E'-E)^2}}f(x,E',E,\omega)dE'
-{1\over\epsilon}f(x,E,E,\omega)+\ln(\epsilon){\p f{E'}}(x,E,E,\omega).
\]
Applying the Taylor's formula
\[
f(x,E',E,\omega)=f(x,E,E,\omega)+\int_0^1{\p f{E'}}(x,E+t(E'-E),E,\omega)dt\ \cdot(E'-E)
\]
and the identity
\[
&
\int_0^1{1\over{E'-E}}{\p f{E'}}(x,E+t(E'-E),E,\omega)dt=
{1\over{E'-E}}{\p f{E'}}(x,E,E,\omega)
\\
&
+
\int_0^1{1\over{E'-E}}\Big({\p f{E'}}(x,E+t(E'-E),E,\omega)-{\p f{E'}}(x,E,E,\omega)\Big)dt
\]
we see analogously to the proof of Lemma \ref{ad-le-1}  that by the Lebesgue Dominated Convergence Theorem 
\be\label{var-for-11}
\int_S\lim_{\epsilon\to 0}F_\epsilon(x,E,\omega) d\omega
=
\lim_{\epsilon\to 0}\int_SF_\epsilon(x,E,\omega) d\omega.
\ee
In virtue of definition (\ref{def-h2}) and (\ref{var-for-10}), (\ref{var-for-11})
\[
&
\int_S\Big({\rm p.f.}\int_E^{E_m}{1\over{(E'-E)^2}}f(x,E',E,\omega)dE'\Big)d\omega\\
&
=
\int_S\lim_{\epsilon\to 0}\Big(
\int_{E+\epsilon}^{E_m}{1\over{(E'-E)^2}}f(x,E',E,\omega)dE'
-{1\over\epsilon}f(x,E,E,\omega)+\ln(\epsilon){\p f{E'}}(x,E,E,\omega)\Big)
d\omega\\
&
=\int_S\lim_{\epsilon\to 0}F_\epsilon(x,E,\omega) d\omega
=
\lim_{\epsilon\to 0}\int_SF_\epsilon(x,E,\omega) d\omega\\
&
= 
{\rm p.f.}\int_E^{E_m}{1\over{(E'-E)^2}}\Big(\int_Sf(x,E',E,\omega)d\omega\Big)dE'
\]
which completes the proof.
\end{proof}

In the same way (but more easily) we find that
for $f\in C(\ol G\times S\times I,C^{\alpha}(I')),\ \alpha>0$ and $E\not=E_m$
\be\label{var-for-12}
\int_S\Big({\rm p.f.}\int_E^{E_m}{1\over{E'-E}}f(x,E',E,\omega)dE'\Big)d\omega
=
{\rm p.f.}\int_E^{E_m}{1\over{E'-E}}\Big(\int_Sf(x,E',E,\omega)d\omega\Big)dE'.
\ee

The integral $\int_S\int_0^{2\pi}\psi(x,\gamma(E',E,\omega)(s),E')v(x,\omega,E)ds d\omega$ below (in sections \ref{ad-sec} and \ref{var-for}) emerging from M\o ller collision operator needs a special treatment which is yielded by the next lemma.

\begin{lemma}\label{le-m:0}
Let $-1<t<1$, and let for any $\omega\in S$ 
\[
\Gamma_{t}^\omega=\{\omega'\in S\ |\ \omega\cdot\omega'=t\}.
\]
Then for any $f\in C(S\times S')$ one has
\begin{align}\label{eq:le-m:0}
\int_S \int_{\Gamma_t^\omega} f(\omega,\omega')d\ell(\omega') d\omega
=
\int_{S'} \int_{\Gamma_t^{\omega'}} f(\omega,\omega')d\ell(\omega) d\omega'
\end{align}
where $\int_{\Gamma_t^\omega} h(\omega')d\ell(\omega')$ is the path integral of $h$ along $\Gamma_t^\omega$.
\end{lemma}

\begin{proof}

We apply the idea given in \cite{tervo18-up}), Lemma 5.4.
Let $-1<t_0<t<1$.
For every $\omega\in S$,
the usual surface measure $d\omega$ on $S$
disintegrates into a family of measures
$(1-r^2)^{-1/2} d\ell_r\otimes dr$, $-1<r<1$,
where $\ell_r$ is the path length measure along the curve
$\Gamma^\omega_r$. 
Let $S_{\omega,[t_0,t]}$ be the spherical zone
\[
S_{\omega,[t_0,t]}
:=\{\omega'\in S\ |\ t_0\leq \omega'\cdot\omega\leq t\}.
\]
Then
we have by Fubin's Theorem
\bea\label{dis-int}
&
\int_S \int_{S_{\omega,[t_0,t]}} f(\omega,\omega')d\omega' d\omega
= 
\int_S \int_{t_0}^{t} \int_{\Gamma_r^\omega} f(\omega,\omega')(1-r^2)^{-1/2} d\ell_r(\omega')dr d\omega\nonumber \\
&
=
\int_{t_0}^{t} (1-r^2)^{-1/2}\int_S \int_{\Gamma_r^\omega} f(\omega,\omega')d\ell_r(\omega')d\omega dr .
\eea
We  use the parametrization
\[
s\to R(\omega)\big(\sqrt{1-r^2}\cos(s),\sqrt{1-r^2}\sin(s),r\big)
\]
for $\Gamma^\omega_r$.

Furthermore, let $\chi_{\R_+}$ be the characteristic function of $\R_+$. 
Then again by the Fubini's Theorem 
\[
&
\int_S \int_{S_{\omega,[t_0,t]}} f(\omega,\omega')d\omega' d\omega
\\
&
=
\int_S \int_{S'}\chi_{\R_+}(\omega'\cdot\omega-t_0)
\chi_{\R_+}(t-\omega'\cdot\omega)
 f(\omega,\omega')d\omega' d\omega\\
&
=
\int_{S'} \int_{S}\chi_{\R_+}(\omega\cdot\omega'-t_0)
\chi_{\R_+}(t-\omega'\cdot\omega)
 f(\omega,\omega')d\omega d\omega'\\
& 
=
\int_{S'} \int_{S}\chi_{\R_+}(\omega'\cdot\omega-t_0)
\chi_{\R_+}(t-\omega\cdot\omega')
 f(\omega,\omega')d\omega d\omega'\\
& 
=
\int_{S'} \int_{S_{\omega',[t_0,t]}} f(\omega,\omega')d\omega d\omega'.
\]
Thus by (\ref{dis-int})
\be\label{ccc}
\int_{t_0}^{t} (1-r^2)^{-1/2}\int_S \int_{\Gamma_r^\omega} f(\omega,\omega')d\ell_r(\omega')d\omega dr 
=
\int_{t_0}^{t} (1-r^2)^{-1/2}\int_{S'} \int_{\Gamma_r^{\omega'}} f(\omega,\omega')d\ell_r(\omega)d\omega' dr.  
\ee
Noting that the function 
\[
r\to \int_S \int_{\Gamma_r^\omega} f(\omega,\omega')d\ell_r(\omega')d\omega 
=\sqrt{1-r^2}\int_S\int_0^{2\pi}f\big(\omega,R(\omega)(\sqrt{1-r^2}\cos(s),\sqrt{1-r^2}\sin(s),r)\big)dr d\omega
\]
is  continuous 
we get the assertion by taking (in formula (\ref{ccc})) the derivative with respect to $t$ on each side. 

\end{proof}

\begin{corollary}\label{sp-ch-le}

For $\psi,\ v\in C(\ol G\times S\times I)$ 
\be\label{vf-3}
\int_S\int_0^{2\pi}\psi(x,\gamma(E',E,\omega)(s),E')v(x,\omega,E)ds d\omega
=
\int_{S'}\int_0^{2\pi}\psi(x,\omega',E')v(x,\gamma(E',E,\omega')(s),E)ds d\omega'
\ee
where $\gamma(E',E,\omega)$ is as in section \ref{moller}.
\end{corollary}

\begin{proof}
Let for fixed $x, \ E',\ E$ 
\[
f(\omega,\omega'):=\psi(x,\omega',E')v(x,\omega,E).
\]
Recall that
$\gamma(E',E,\omega)$ is a parametrization of the curve
$\Gamma_{\mu}^\omega:=\{\omega'\in S\ |\ \omega\cdot\omega'=\mu(E',E)\}$ and
note that $\n{\gamma(E',E,\omega)'(s)}=\sqrt{1-\mu(E',E)^2}$.
Hence we have by Lemma \ref{le-m:0} 
\bea\label{vf-3-a}
&
\int_S\int_0^{2\pi}\psi(x,\gamma(E',E,\omega)(s),E')v(x,\omega,E)ds d\omega
=
{1\over{\sqrt{1-\mu(E',E)^2}}}\int_S\int_{\Gamma_{\mu}^\omega} f(\omega,\omega') d\ell(\omega') d\omega\nonumber\\
&
=
{1\over{\sqrt{1-\mu(E',E)^2}}}\int_{S'}\int_{\Gamma_{\mu}^{\omega'}} f(\omega,\omega') d\ell(\omega)d\omega'
\nonumber\\
&
=
\int_{S'}\int_{0}^{2\pi}\psi(x,\omega',E')v(x,\gamma(E',E,\omega')(s),E)ds d\omega',
\eea
as desired.

\end{proof}

Finally we prove the following result.

\begin{theorem}\label{ad-le-6}
Let $\psi\in C(\ol G,C^2(I,C^3(S)))$. 
Then for fixed $x,\ \omega,\ E$ the mapping 
\[
F(E'):= \int_0^{2\pi}\psi(x,\gamma(E',E,\omega)(s),E')ds
\]
is differentiable and 
\begin{enumerate}
\item 
for $E'\not=E$
\bea\label{var-for-13-a}
F'(E')&={\partial\over{\partial E'}}\big(\int_0^{2\pi}\psi(x,\gamma(E',E,\omega)(s),E')ds\big)\nonumber\\
&
=
\int_0^{2\pi}\la (\nabla_{S}\psi)(x,\gamma(E',E,\omega)(s),E'),
{\p {\gamma}{E'}}(E',E,\omega)(s)\ra ds
\nonumber\\
&
+
\int_0^{2\pi}{\p {\psi}{E}}(x,\gamma(E',E,\omega)(s),E')ds
\eea
\item 
for $E'=E$
\bea\label{var-for-13-b}
&
F'(E)={\partial\over{\partial E'}}\big(\int_0^{2\pi}\psi(x,\gamma(E',E,\omega)(s),E')ds\big)_{\Big|E'=E}
=
2\pi\ (\partial_{E'}\mu)(E,E)(\omega\cdot\nabla_{S}\psi)(x,\omega,E)\nonumber\\
&
+
\sum_{|\alpha|\leq 2}a_{\alpha}(E,\omega)(\partial_{\omega}^\alpha\psi)(x,\omega,E)
+
2\pi {\p {\psi}{E}}(x,\omega,E).
\eea
Here (below we denote $(\psi\circ H_\omega)(x,\xi,E)=\psi(x,H_\omega(\xi),E)$)
\bea\label{aa}
&
\sum_{|\alpha|\leq 2}a_\alpha (\omega,E)
(\partial_{\omega}^\alpha\psi)(x,\omega,E)\nonumber\\
&
:=\lim_{E'\to E}\int_0^{2\pi}\sum_{j=1}^2\partial_j \big(\la (\nabla_S\psi\circ  H_\omega)(x,.,E'),\eta(E',E,\omega,s)\ra\big)(0) \xi_j((E',E,\omega,s) ds.
\eea
where $\partial_j=\partial_{\xi_j}$ and where the limit is given in Lemma \ref{aij} below.
\end{enumerate}
Above $\xi(E',E,\omega,s):=J(\zeta(E',E,\omega,s)),\ \zeta(E',E,\omega,s)=H_\omega^{-1}(\gamma(E',E,\omega)(s))$ (recall $J$ and $\zeta$ from section \ref{taylor-S}) and $\eta(E',E,\omega,s):={\p {\gamma}{E'}}(E',E,\omega)(s)$.

\end{theorem}

\begin{proof}
A.
Note that the mapping $E'\to \gamma(E',E,\omega)(s)$ is a curve on $S$ 
and so
${\p {\gamma}{E'}}(E',E,\omega)(s)\in T_{\gamma(E',E,\omega)(s)}(S)$.
Therefore formula (\ref{var-for-13-a}) follows from the chain rule on manifolds.

B.
Writing  $\mu=\mu(E',E),\ \partial_{E'}\mu={\p {\mu}{E'}}(E',E)$ we find that
\be\label{pg}
{\p {\gamma}{E'}}(E',E,\omega)(s)=
R(\omega)\big(-{{(\partial_{E'}\mu)\mu}\over{\sqrt{1-\mu^2}}}\cos(s),
-{{(\partial_{E'}\mu)\mu}\over{\sqrt{1-\mu^2}}}\sin(s),\partial_{E'}\mu\big)
\ee
The  term
\[
\int_0^{2\pi}\la (\nabla_{S}\psi)(x,\gamma(E',E,\omega)(s),E'),
{\p {\gamma}{E'}}(E',E,\omega)(s)\ra ds
\] 
in (\ref{var-for-13-a}) requires further study for $E'=E$ since the partial derivative ${\p {\gamma}{E'}}(E',E,\omega)(s)$ does not exist for $E'=E$.

Because the mapping $F$ is continuous it suffices (by the consequence of L'Hospital's rule) to show that the limit $\lim_{E'\to E}F'(E')$ exists
and that
\bea\label{limF'}
&
\lim_{E'\to E}F'(E')
=2\pi\ (\partial_{E'}\mu)(E,E)(\omega\cdot\nabla_{S}\psi)(x,\omega,E)\nonumber\\
&
+
\sum_{|\alpha|\leq 2}a_{\alpha}(E,\omega)(\partial_{\omega}^\alpha\psi)(x,\omega,E)+2\pi {\p {\psi}{E}}(x,\omega,E).
\eea
Since (note that  $\gamma(E,E,\omega)(s)=\omega$)
\[
\lim_{E'\to  E}\int_0^{2\pi}{\p {\psi}{E}}(x,\gamma(E',E,\omega)(s),E')ds
=2\pi {\p {\psi}{E}}(x,\omega,E)
\]
we have to prove that
\bea\label{var-for-14}
&
\lim_{E'\to E}\Big(
\int_0^{2\pi}\la (\nabla_{S}\psi)(x,\gamma(E',E,\omega)(s),E'),
{\p {\gamma}{E'}}(E',E,\omega)(s)\ra ds\Big)\nonumber\\
&
=
2\pi\ (\partial_{E'}\mu)(E,E)(\omega\cdot\nabla_{S}\psi)(x,\omega,E)
+
\sum_{|\alpha|\leq 2}a_{\alpha}(E,\omega)(\partial_{\omega}^\alpha\psi)(x,\omega,E)
.
\eea
We proceed as follows.

In virtue of the Taylor's formula 
(recall $\eta={\p {\gamma}{E'}}(E',E,\omega)(s)$)
\bea\label{ex-9}
&
\la (\nabla_{S}\psi)(x,\gamma(E',E,\omega)(s),E'),{\p {\gamma}{E'}}(E',E,\omega)(s)\ra
=\la (\nabla_S\psi)(x,\omega,E'),{\p {\gamma}{E'}}(E',E,\omega)(s)\ra\nonumber\\
&
+
\sum_{j=1}^2 \partial_j\big(\la (\nabla_S\psi\circ  H_\omega)(x,.,E'),\eta(E',E,\omega,s)\ra\big)(0) \xi_j(E',E,\omega,s) 
\nonumber\\
&
+
(R_\omega\psi(x,\cdot,E'))(\omega')\Big(\zeta(E',E,\omega, s),\zeta(E',E,\omega,\eta(E',E,\omega,s)\Big)
\eea
where 
the residual is 
\bea\label{t-residual-a}
&
R_\omega(\psi(x,.,E))(\omega')(\zeta(E',E,\omega, s),\zeta(E',E,\omega, s),\eta(E',E,\omega,s))
\nonumber\\
&
=
\sum_{|\alpha|=2}{{|\alpha|}\over{\alpha !}}\int_0^1(1-t)\partial_\xi^\alpha
\big(\la (\nabla_S\psi\circ H_\omega)(x,.,E'),\eta(E',E,\omega,s)\ra\big)(t\xi) dt \cdot \xi^\alpha. 
\eea

We have for $E'\not=E$
\bea\label{ex-10}
&
\int_{0}^{2\pi}
\la (\nabla_S\psi)(x,\omega,E'),
{\p {\gamma}{E'}}(E',E,\omega)(s)\ra ds
=
\la (\nabla_S\psi)(x,\omega,E'),
\int_{0}^{2\pi}{\p {\gamma}{E'}}(E',E,\omega)(s)ds\ra\nonumber\\
&
=2\pi\ (\partial_{E'}\mu)(E',E)(\omega\cdot\nabla_S\psi)(x,\omega,E')
\eea
where we recalled that $R(\omega)e_3=\omega$,
and so 
\bea\label{lim-1}
&
\lim_{E'\to E}\Big(
\int_{0}^{2\pi}
\la (\nabla_S\psi)(x,\omega,E'),
{\p {\gamma}{E'}}(E',E,\omega)(s)\ra ds\Big)
\nonumber\\
&
=2\pi\ (\partial_{E'}\mu)(E,E)(\omega\cdot\nabla_S\psi)(x,\omega,E).
\eea

For the residual we have
\be\label{lim-3}
\lim_{E'\to E}\int_0^{2\pi}
(R_\omega\psi(x,\cdot,E'))(\omega')\Big(\zeta(E',E,\omega,s) ,\zeta(E',E,\omega,s),
\eta(E',E,\omega,s)\Big)ds=0
\ee
This is seen as follows. By (\ref{e7}) 
\be
\n{\zeta(E',E,\omega, s)}\leq C\n{\gamma(E',E,\omega)(s)-\omega}.
\ee
Letting $\gamma_0(E',E)(s):=(\sqrt{1-\mu^2}\cos(s),\sqrt{1-\mu^2}\sin(s),\mu)$
and recalling that $\omega=R(\omega)e_3$ we obtain
\bea\label{ex-21}
&
\n{\gamma(E',E,\omega)(s)-\omega}^2
=\n{R(\omega)\gamma_0(E',E)(s)-\omega}^2=\n{\gamma_0(E',E)(s)-R(\omega)^{-1}\omega}^2
\nonumber\\
&
=
\n{\gamma_0(E',E)(s)-e_3}^2=2(1-\mu(E',E))
\eea
and then
we get
\bea\label{ex-12-b}
&
\Big|(R_\omega\psi(x,\cdot,E')(\omega')\Big(\zeta(E',E,\omega,s) ,\zeta(E',E,\omega,s),
\eta(E',E,\omega,s)\Big)|
\nonumber\\
&
\leq 
\sum_{|\alpha|=2}{{|\alpha|}\over{\alpha !}}\int_0^1(1-t)\n{\partial_\xi^\alpha
(\nabla_S\psi\circ H_\omega)(x,t\xi,E')}\n{\eta(E',E,\omega,s)} \n{\xi}^{|\alpha|} 
\nonumber\\
&
\leq
C_1\n{\psi}_{C(\ol G\times I,C^3(S))}\n{\zeta(E',E,\omega,s)}^2
\n{{\p {\gamma}{E'}}(E',E,\omega)(s)}\nonumber\\
&
\leq 
C_1C^2\n{\psi}_{C(\ol G\times I,C^3(S))}\n{\gamma(E',E,\omega)(s)-\omega}^2
\n{{\p {\gamma}{E'}}(E',E,\omega)(s)}\nonumber\\
&
\leq 
2C_1C^2\n{\psi}_{C(\ol G\times I,C^3(S))}(1-\mu){{|\partial_{E'}\mu|}\over{\sqrt{1-\mu^2}}}
=2C_1C^2\n{\psi}_{C(\ol G\times I,C^3(S))}|\partial_{E'}\mu|{{\sqrt{1-\mu}}\over{\sqrt{1+\mu}}}.
\eea
In (\ref{ex-12-b}) the right hand side convergences to zero since $\mu\to 1$ as $E'\to E$.
Hence the assertion (\ref{var-for-13-b}) follows by combining (\ref{ex-9}), (\ref{lim-1}), (\ref{lim-3}).
 
\end{proof}

\begin{lemma}\label{aij}
Let $\xi=J(\zeta),\ \zeta:=H_\omega^{-1}(\gamma(E',E,\omega)(s))\in T_\omega(S)$ and $\eta={\p \gamma{E'}}(E',E,\omega)(s)$. Then
\bea\label{aij-1}
&
\lim_{E'\to E}\int_0^{2\pi}\sum_{j=1}^2\partial_j\big(\la  (\nabla_S\psi\circ  H_\omega)(x,.,E'),\eta(E',E,\omega,s)\ra\big)(0) \xi_j((E',E,\omega,s) ds
\nonumber\\
&
=-(\partial_{E'}\mu)(E,E)\sum_{j=1}^2\int_0^{2\pi}
 \partial_j\big(\la (\nabla_S\psi\circ  H_\omega)(x,.,E),R(\omega)\big(\cos(s),\sin(s),0\big)\ra\big)(0)\nonumber\\
&
\cdot
\la R(\omega)\big(\cos(s),\sin(s),0\big)
,\ol\Omega_j\ra  ds.
\eea
\end{lemma}

\begin{proof}
Recall that $H_\omega=\exp_\omega$ is given in Example \ref{ex-map}.
Denote $\omega':=\gamma(E',E,\omega)(s)$ and $\mu=\mu(E',E)$. 
In due to (\ref{e3}) we find that
\bea\label{aij-2}
&
\zeta(E',E,\omega,s)=\exp_\omega^{-1}(\omega')
=
{{{\rm \ol{arc}cos}(\la\omega',\omega\ra)}\over{
\sqrt{1-\la\omega',\omega\ra^2}}}(\omega'-\la\omega',\omega\ra\omega)
\nonumber\\
&
=
{{{\rm \ol{arc}cos}(\la\gamma(E',E,\omega)(s),\omega\ra}\over{
\sqrt{1-\la\gamma(E',E,\omega)(s),\omega\ra^2}}}(\gamma(E',E,\omega)(s)-\la\gamma(E',E,\omega)(s),\omega\ra\omega)
.
\eea
Furthermore, we  find that
\[
\la\gamma(E',E,\omega)(s),\omega\ra &=\la R(\omega)\gamma_0(E',E)(s),\omega\ra
=\la \gamma_0(E',E)(s),R(\omega)^*\omega\ra\\
&
=\la \gamma_0(E',E)(s),e_3\ra=\mu(E',E)
\]
where we noticed that $R(\omega)^*\omega=R(\omega)^{-1}\omega=e_3$.
Hence
\bea\label{aij-3}
&
\zeta(E',E,\omega,s) 
=
{{{\rm \ol{arc}cos}(\mu)}\over{
\sqrt{1-\mu^2}}}(\gamma(E',E,\omega)(s)-\mu\omega)
\nonumber\\
&
=
{{{\rm \ol{arc}cos}(\mu)}\over{
\sqrt{1-\mu^2}}}R(\omega)(\gamma_0(E',E)(s)-(0,0,\mu))
\eea
since $R(\omega)(0,0,\mu)=\mu\omega$ (recall $R(\omega)$ from section \ref{moller}).

In virtue of Example \ref{ex-map}
\[
\lim_{E'\to E}{{{\rm \ol{arc}cos}(\mu(E',E))}\over{
\sqrt{1-\mu(E',E)^2}}}=1.
\]
In addition,
\[
R(\omega)(\gamma_0(E',E)(s)-(0,0,\mu))=R(\omega)\big(\sqrt{1-\mu^2}\cos(s),\sqrt{1-\mu^2}\sin(s),0\big).
\]
Hence (because $\mu(E,E)=1$)
\bea\label{aij-4}
&
\lim_{E'\to E}\int_0^{2\pi}\sum_{j=1}^2\partial_j\big(\la  (\nabla_S\psi\circ  H_\omega)(x,.,E'),\eta(E',E,\omega,s)\ra\big)(0) \xi_j((E',E,\omega,s) ds
\nonumber\\
&
=
\lim_{E'\to E}\sum_{j=1}^2\int_0^{2\pi}\partial_j\big( \la (\nabla_S\psi\circ  H_\omega)(x,.,E'),
R(\omega)\big(-{{(\partial_{E'}\mu)\mu}\over{\sqrt{1-\mu^2}}}\cos(s),
-{{(\partial_{E'}\mu)\mu}\over{\sqrt{1-\mu^2}}}\sin(s),\partial_{E'}\mu\big)
\ra  \big)(0)
\nonumber\\
&
\cdot
\Big(
R(\omega)\big(\sqrt{1-\mu^2}\cos(s),\sqrt{1-\mu^2}\sin(s),0\big)
\Big)_j ds
=
\lim_{E'\to E}\sum_{j=1}^2\int_0^{2\pi} \nonumber\\
&
\partial_j\big(\la (\nabla_S\psi\circ  H_\omega)(x,.,E'),
R(\omega)\big(-(\partial_{E'}\mu)\mu\ \cos(s),
-(\partial_{E'}\mu)\mu\ \sin(s),\sqrt{1-\mu^2}\partial_{E'}\mu\big)
\ra\big)(0)  \nonumber\\
&
\cdot
\Big(
R(\omega)\big(\cos(s),\sin(s),0\big)
\Big)_j ds\nonumber\\
&
=-(\partial_{E'}\mu)(E,E)\sum_{j=1}^2\int_0^{2\pi}
 \partial_j\big(\la (\nabla_S\psi\circ  H_\omega)(x,.,E),
R(\omega)\big(\cos(s),\sin(s),0\big)
\ra \big)(0) \nonumber\\
&
\cdot
\Big(R(\omega)\big(\cos(s),\sin(s),0\big)
\Big)_j ds
\eea
Since
\[
\la R(\omega)\big(\cos(s),\sin(s),0\big),\omega\ra=
\la \big(\cos(s),\sin(s),0\big),R^*(\omega)\omega\ra
=\la \big(\cos(s),\sin(s),0\big),e_3\ra =0
\]
the vector $R(\omega)\big(\cos(s),\sin(s),0\big)$ lies in $T_\omega(S)$ (as it should be).
In addition (since $\ol\Omega_1,\ \ol\Omega_2$ are orthogonal),
\[
\Big(R(\omega)\big(\cos(s),\sin(s),0\big)
\Big)_j=\la R(\omega)\big(\cos(s),\sin(s),0\big),\ol\Omega_j\ra,\ j=1,2.
\]
Hence  we conclude the assertion from
(\ref{aij-4}).
\end{proof}

\begin{remark}\label{rem-aij}

A. The result of the limit in (\ref{aij-1}) can be further computed. We notice that the limit (\ref{var-for-14}) can be computed in the coordinate free way (by personal communication with Dr. Petri Kokkonen) but our below treatises are based on local ones.

Let
\[
b_j(\omega):=\int_0^{2\pi}
\la R(\omega)\big(\cos(s),\sin(s),0\big)
,\ol\Omega_j\ra R(\omega)\big(\cos(s),\sin(s),0\big)ds.
\]
By using the matrix partition $R(\omega)=\qmatrix{\ol\Omega_2&\ol\Omega_1&\omega}$ we have
\bea\label{pd-term-a}
&
\la R(\omega)\big(\cos(s),\sin(s),0\big)
,\ol\Omega_j\ra =\la \big(\cos(s),\sin(s),0\big)
,R^*(\omega)\ol\Omega_j\ra\nonumber\\
&
=\la \big(\cos(s),\sin(s),0\big)
,\qmatrix{\ol\Omega_2\cr\ol\Omega_1\cr\omega}\ol\Omega_j\ra
=
\la \big(\cos(s),\sin(s),0\big)
,\big(\la \ol\Omega_2,\ol\Omega_j\ra,\la\ol\Omega_1,\ol\Omega_j\ra,0\big)\ra.
\eea
Hence we find that
\[
\la R(\omega)\big(\cos(s),\sin(s),0\big)
,\ol\Omega_1\ra=\sin(s),
\]
\[
\la R(\omega)\big(\cos(s),\sin(s),0\big)
,\ol\Omega_2\ra =\cos(s)
\] 
and so
\[
b_1(\omega)=
 R(\omega)\big(\int_0^{2\pi}\sin(s)\cos(s)ds,\int_0^{2\pi}\sin^2(s)ds,0\big)
=\pi R(\omega)e_2=\pi\ol\Omega_1(\omega).
\]
Similarly $b_2(\omega)=\pi\ol\Omega_2(\omega)$. 
Hence
\bea\label{r-1}
&
-(\partial_{E'}\mu)(E,E)\sum_{j=1}^2\int_0^{2\pi}
 \partial_j\big(\la (\nabla_S\psi\circ  H_\omega)(x,.,E),R(\omega)\big(\cos(s),\sin(s),0\big)\ra\big)(0)\nonumber\\
&
\cdot
\la R(\omega)\big(\cos(s),\sin(s),0\big)
,\ol\Omega_j\ra  ds\nonumber
\\
&
=
-\pi\ 
(\partial_{E'}\mu)(E,E)\sum_{j=1}^2
 \partial_j\big(\la (\nabla_S\psi\circ  H_\omega)(x,.,E)
,\ol\Omega_j\ra  
\eea

We compute the result of (\ref{r-1}) in spherical coordinates.
Recall that the spherical coordinates for $S$ are given by the parametrization
\[
g(\phi,\theta)=(\cos\phi\sin\theta,\sin\phi\sin\theta,\cos\theta)
\]
and that
\[
{\p f{\omega_1}}(\omega)={\p {(f\circ g)}{\phi}}(\phi,\theta)_{| g(\phi,\theta)=\omega},\ 
{\p f{\omega_2}}(\omega)={\p {(f\circ g)}{\theta}}(\phi,\theta)_{| g(\phi,\theta)=\omega}.
\]
By a routine computation we see that 
\[
\la\ol\Omega_1(\omega),\partial_{\omega_1}\ol\Omega_1(\omega)\ra=
\la\ol\Omega_2(\omega),\partial_{\omega_2}\ol\Omega_2(\omega)\ra
=
\la\ol\Omega_1(\omega),\partial_{\omega_2}\ol\Omega_2(\omega)\ra=0,
\]
\[
\la\ol\Omega_2(\omega),\partial_{\omega_1}\ol\Omega_1(\omega)\ra=-{1\over{\sqrt{\omega_1^2+\omega_2^2}}}\omega_3.
\]

The elements $G_{ij}=\la {\partial\over{\partial \omega_i}},{\partial\over{\partial \omega_j}}\ra$ of the Riemannian metric tensor are
\[
G_{11}=\omega_1^2+\omega_2^2,\ G_{12}=G_{21}=0,\ G_{22}=1.
\]
Hence the
Laplace-Beltrami operator in these coordinates is (here $|G|={\rm det}(G)$)
\[
&
\Delta_S\psi ={1\over{\sqrt{|G|}}}
\sum_{i=1}^2\sum_{j=1}^2\partial_{\omega_i}\Big(\sqrt{|G|}(G^{-1})_{ij}\partial_{\omega_j}\psi\Big)\\
&
=
{1\over{\omega_1^2+\omega_2^2}}\partial_{\omega_1}^2\psi
+\partial_{\omega_2}^2\psi+{{\omega_3}\over{\sqrt{\omega_1^2+\omega_2^2}}}\partial_{\omega_2}\psi.
\]
Moreover, note that
\[
\ol\Omega_1\sim {1\over{\sqrt{\omega_1^2+\omega_2^2}}}{\partial\over{\partial\omega_1}},\
\ol\Omega_2\sim {\partial\over{\partial\omega_2}}
\]
and
\[
\nabla_S\psi=\la\nabla_S\psi,\ol\Omega_1\ra \ol\Omega_1+\la\nabla_S\psi,\ol\Omega_2\ra \ol\Omega_2.
\]

Consider the derivatives
\[
\partial_j\big(\la (\nabla_S\psi\circ H_\omega)(x,.,E),\ol\Omega_j(\omega)\ra\big)(0),\ j=1,2.
\]
Since $(\partial_jH_\omega)(0)=\ol\Omega_j(\omega)$
we find that
\be\label{ex-par}
\partial_j(f\circ H_\omega)(0)=\la(\nabla_S f)(\omega),\ol\Omega_j(\omega)\ra
,\ j=1,2
\ee
that is,  
\[
&
\partial_1(f\circ H_\omega)(0)={1\over{\sqrt{\omega_1^2+\omega_2^2}}}
\la(\nabla_S f)(\omega),{\partial\over{\partial\omega_1}}\ra
={1\over{\sqrt{\omega_1^2+\omega_2^2}}}(\partial_{\omega_1}f)(\omega),\\
&
\partial_2(f\circ H_\omega)(0)=\la(\nabla_S f)(\omega),{\partial\over{\partial\omega_2}}\ra
=(\partial_{\omega_2}f)(\omega).
\]

Hence we see by the above formulas (somewhat formally) that
\[
&
\partial_1\big(\la (\nabla_S\psi\circ H_\omega)(x,.,E),\ol\Omega_1(\omega)\ra\big)(0)
=
\\
&
=
\la \partial_1\big(\nabla_S\psi\circ H_\omega\big)(x,.,E),\ol\Omega_1(\omega)\ra
\\
&
=
{1\over{\sqrt{\omega_1^2+\omega_2^2}}}
\la \partial_{\omega_1}\big(\nabla_S\psi\big)(x,\omega,E),\ol\Omega_1(\omega)\ra\\
&
=
{1\over{\sqrt{\omega_1^2+\omega_2^2}}}
\partial_{\omega_1}\la(\nabla_S\psi)(x,\omega,E),\ol\Omega_1(\omega)\ra
-{1\over{\sqrt{\omega_1^2+\omega_2^2}}}\la(\nabla_S\psi)(x,\omega,E),(\partial_{\omega_1}\ol\Omega_1)(\omega)\ra\\
&
=
{1\over{\sqrt{\omega_1^2+\omega_2^2}}}
(\partial_{\omega_1}^2\psi)(x,\omega,E)
+{{\omega_3}\over{\sqrt{\omega_1^2+\omega_2^2}}}
(\partial_{\omega_2}\psi)(x,\omega,E) 
\]
since $\partial_{\omega_1}\Big({1\over{\sqrt{\omega_1^2+\omega_2^2}}}\Big)=
{\partial\over{\partial\phi}}\Big({1\over{\sin\theta}}\Big)=0$,
and similarly
\[
&
\partial_2\big(\la (\nabla_S\psi\circ H_\omega)(x,.,E),\ol\Omega_2(\omega)\ra\big)(0)
=
\la \partial_{\omega_2}\big(\nabla_S\psi\big)(x,\omega,E),\ol\Omega_2(\omega)\ra\\
&
=
\partial_{\omega_2}\la(\nabla_S\psi)(x,\omega,E),\ol\Omega_2(\omega)\ra
-\la(\nabla_S\psi)(x,\omega,E),(\partial_{\omega_2}\ol\Omega_2)(x,\omega,E)\ra\\
&
=
(\partial_{\omega_2}^2\psi)(x,\omega,E).
\] 

As a conclusion we find that 
\[
&
\sum_{j=1}^2
\partial_j\big(\la (\nabla_S\psi\circ H_\omega)(x,.,E),\ol\Omega_j(\omega)\ra\big)(0)\\
&
=
{1\over{\omega_1^2+\omega_2^2}}(\partial_{\omega_1}^2\psi)(x,\omega,E)
+(\partial_{\omega_2}^2\psi)(x,\omega,E)+{{\omega_3}\over{\sqrt{\omega_1^2+\omega_2^2}}}(\partial_{\omega_2}\psi)(x,\omega,E)\\
&
=
(\Delta_S\psi)(x,\omega,E).
\]

Combining the above computations we see that 
\bea\label{pd-term}
\sum_{|\alpha|\leq 2}
a_{\alpha}(\omega,E)(\partial_{\omega}^\alpha \psi)(x,\omega,E)
=
-\pi\ (\partial_{E'}\mu)(E,E)
(\Delta_S\psi)(x,\omega,E).
\eea

B. We remark also that by (\ref{nab0}) in (\ref{var-for-13-b})
\be\label{van}
2\pi\ (\partial_{E'}\mu)(E,E)(\omega\cdot\nabla_S\psi)(x,\omega,E)=0.
\ee

\end{remark}

\subsection{A refined form of the exact equation}\label{exact}

We derive a refined form  for the exact  transport  operator under consideration.
Recall that for $j=1,\ 2$
\be\label{K-12}
(\ol {\s K}_{j}\psi)(x,\omega,E',E)
=
\hat{\sigma}_{j}(x,E',E)\int_{0}^{2\pi}\psi(x,\gamma(E',E,\omega)(s),E')ds,
\ee
and so 
\be\label{e-1}
{\s H}_j\big((\ol{\s K}_j\psi)(x,\omega,\cdot,E)\big)(E)
=
{\rm p.f.}\int_E^{E_m}{1\over{(E'-E)^j}}
\hat{\sigma}_{j}(x,E',E)\int_{0}^{2\pi}\psi(x,\gamma(E',E,\omega)(s),E')ds dE'
\ee

We begin with 

\begin{lemma}\label{e-le-2}
Suppose that $f\in C(\ol G\times S\times I, C^\alpha(I')),\ \alpha>0$. Then 
\be\label{e-0-a}
{\rm p.f.}\int_E^{E_m}{{f(x,\omega,E,E')}\over{E'-E}}dE'=\ln(E_m-E)f(x,\omega,E,E)+u(x,\omega,E)
\ee
where $u\in C(\ol G\times S\times I)$.
\end{lemma}

\begin{proof}

We have
\[
&
{\rm p.f.}\int_E^{E_m}{{f(x;\omega,E,E')}\over{E'-E}}dE'
={\rm p.f.}\int_E^{E_m}{{f(x,\omega,E,E)}\over{E'-E}}dE'\\
&
+\int_E^{E_m}{{f(x,\omega,E,E')-f(x,\omega,E,E)}\over{E'-E}}dE'
=\ln(E_m-E)f(x,\omega,E,E)+u(x,\omega,E)
\]
where
\[
u(x,\omega,E):=\int_E^{E_m}{{f(x,\omega,E,E')-f(x,\omega,E,E)}\over{E'-E}}dE'.
\]

We find that $u(x,\omega,E)=\int_{I'}\chi_{\R_+}(E'-E){{f(x,\omega,E,E')-f(x,\omega,E,E)}\over{E'-E}}dE'$
where
\[
\Big|\chi_{\R_+}(E'-E){{f(x,\omega,E,E')-f(x,\omega,E,E)}\over{E'-E}}\Big|
\leq C_\alpha \chi_{\R_+}(E'-E)|E'-E|^{\alpha-1}.
\]
In addition,
\[
\int_{I'}\chi_{\R_+}(E'-E)|E'-E|^{\alpha-1}dE'={1\over\alpha}(E_m-E)^\alpha
\leq {1\over\alpha}E_m^\alpha
\]
and so the continuity of $u$ follows from the Lebesgue Dominated Convergence Theorem. This implies the assertion.

\end{proof}

The next Lemmas gives   information about the domains of the Hadamard finite part integral terms ${\s H}_1\big((\ol{\s K}_j\psi)(x,\omega,\cdot,E)\big)(E),\ j=1,2$.

\begin{lemma}\label{e-th-1} 
Suppose that $\hat\sigma_1\in C(\ol G\times I,C^1( I'))$ and that $0<\alpha<1$, $0<\beta<1$. Then for $ \psi\in C(\ol G\times I,C^\alpha(S)))\cap C(\ol G\times S,C^{\beta}(I))$
\be\label{e-2}
{\s H}_1\big((\ol{\s K}_1\psi)(x,\omega,\cdot,E)\big)(E)
=2\pi\ \ln(E_m-E)\hat{\sigma}_{1}(x,E,E)\psi(x,\omega,E)+u(x,\omega,E)
\ee
where
$u\in C(\ol G\times S\times I)$.
\end{lemma}

\begin{proof}

We have
\bea\label{e-3}
&
{\s H}_1\big((\ol{\s K}_1\psi)(x,\omega,\cdot,E)\big)(E)
=
{\rm p.f.}\int_E^{E_m}{1\over{E'-E}}
\hat{\sigma}_{1}(x,E',E)\int_{0}^{2\pi}\psi(x,\gamma(E',E,\omega)(s),E')ds dE' 
\nonumber\\
&
={\rm p.f.}\int_E^{E_m}{1\over{E'-E}}f(x,\omega,E,E')dE'
\eea
where
\[
f(x,\omega,E,E'):=
\hat{\sigma}_{1}(x,E',E)\int_{0}^{2\pi}\psi(x,\gamma(E',E,\omega)(s),E')ds.
\]
We  show that $f\in C(\ol G\times S\times I,C^{\delta}(I'))$ where $0<\delta\leq\min\{{\alpha\over 2},\beta\}$.

Since $\hat\sigma_1\in  C(\ol G\times I,C^1( I'))$ it suffices to treat only the term 
\[
h_1(x,\omega,E,E'):=\int_{0}^{2\pi}\psi(x,\gamma(E',E,\omega)(s),E')ds.
\]
We have for $E',\ E''\in I$
\bea\label{e-6}
&
h_1(x,\omega,E,E')-h_1(x,\omega,E,E'')
\nonumber\\
&
=
\int_{0}^{2\pi}[\psi(x,\gamma(E',E,\omega)(s),E')-\psi(x,\gamma(E',E,\omega)(s),E'')]ds\nonumber\\
&
+
\int_{0}^{2\pi}[\psi(x,\gamma(E',E,\omega)(s),E'')-\psi(x,\gamma(E'',E,\omega)(s),E'')]ds=:I_1+I_2.
\eea
Since $\psi\in C(\ol G\times S,C^{\beta}(I))$   we obtain
\be\label{e-7}
|I_1|\leq 
\int_{0}^{2\pi}|\psi(x,\gamma(E',E,\omega)(s),E')-\psi(x,\gamma(E',E,\omega)(s),E'')|ds
\leq C_\beta |E'-E''|^{\beta}.
\ee
Similarly since $ \psi\in C(\ol G\times I,C^\alpha(S)))$ we have
\bea\label{e-8}
&
|I_2|\leq 
\int_{0}^{2\pi}|\psi(x,\gamma(E',E,\omega)(s),E'')-\psi(x,\gamma(E'',E,\omega)(s),E'')|ds\nonumber\\
&
\leq C_\alpha \n{\gamma(E',E,\omega)(s)-\gamma(E'',E,\omega)(s)}^{\alpha}
\leq 
C_\alpha C |E'-E''|^{{\alpha\over 2}}.
\eea
Here we used the estimate
\bea\label{e-10}
&
\n{\gamma(E',E,\omega)(s)-\gamma(E'',E,\omega)(s)}^2
\nonumber\\
&
=(\sqrt{1-\mu(E',E)^2}-\sqrt{1-\mu(E'',E)^2})^2+(\mu(E',E)-\mu(E'',E))^2
\leq C|E'-E''|
\eea
which can be seen by elementary computations.

Combining (\ref{e-6}), (\ref{e-7}) and (\ref{e-8}) we get
\[
|h_1(x,\omega,E,E')-h_1(x,\omega,E,E'')|\leq |I_1|+|I_2|\leq 
C_\beta |E'-E''|^\beta +C_\alpha C|E'-E''|^{{\alpha\over 2}}
\]
from which (since $I$ is bounded) it follows that
\[
|h_1(x,\omega,E,E')-h_1(x,\omega,E,E'')|\leq  C_\delta|E'-E''|^\delta
\]
Hence the assertion is a consequence of Lemma \ref{e-le-2}. 

\end{proof}

The proof of the above lemma \ref{e-th-1} gives immediately
\begin{corollary}\label{x-le-cor}
Suppose that $\psi\in C(\ol G,C^\alpha(S\times I)),\ \alpha>0$. Let
\[
h_1(x,\omega,E,E'):=\int_{0}^{2\pi}\psi(x,\gamma(E',E,\omega)(s),E')ds.
\]
Then $h_1$ obeys  
\be\label{le-reg-1-a}
|h_1(x,\omega,E',E)-h_1(x,\omega,E,E)|\leq C_1\n{\psi}_{C(\ol G,C^\alpha(S\times I)}(E'-E)^{{\alpha\over 2}}, \ E'\geq E
\ee
where $C_1$ is independent of $\psi$.
\end{corollary}

\begin{proof}
Choose $\alpha=\beta$ in the above lemma.
It needs only to note that in the above proof $C_\beta=2\pi\ \n{\psi}_{C(\ol G\times S,C^\beta(I))}$
and $C_\alpha=2\pi\ \n{\psi}_{C(\ol G\times I,C^\alpha(I))}$.
\end{proof}

Let for $E'\not=E$
\[
h(x,\omega,E',E):=
\int_{0}^{2\pi}\la \nabla_S\psi(x,\gamma(E',E,\omega)(s),E'),{\p \gamma{E}}(E',E,\omega)(s)\ra ds .
\]
By the proof of Theorem \ref{ad-le-6} $h$ is defined also for $E'=E$ and
\[
&
h(x,\omega,E,E)=\lim_{E'\to E}h(x,\omega,E',E)
= 
\sum_{|\alpha|\leq 2}\widehat a_{\alpha}(E,\omega)(\partial_{\omega}^\alpha\psi)(x,\omega,E)\\
&
:=\lim_{E'\to E}\int_0^{2\pi}
\sum_{j=1}^2\partial_j\big(\la (\nabla_S\psi\circ H_\omega)(x,.,E),\hat\eta(E',E,\omega,s)\ra\big)(0)\xi_j(E',E,\omega,s)ds
\]
where now $\hat\eta={\p \gamma{E}}(E',E,\omega)(s)$ and where we used that
\[
2\pi\ (\partial_{E}\mu)(E,E)(\omega\cdot\nabla_S\psi)(x,\omega,E)=0.
\]

We have the following technical lemma

\begin{lemma}\label{le-reg}
Suppose that 
$\psi\in C(\ol G,C^{1}(I,C^{3}(S)))$. Then the mapping $h$ obeys 
\be\label{le-reg-1}
|h(x,\omega,E',E)-h(x,\omega,E,E)|\leq C(E'-E)^{1\over 2}, \ E'\geq E.
\ee
Analogous result is valid for the mapping
\[
 h_2(x,\omega,E',E):=
\int_{0}^{2\pi}\la \nabla_S\psi(x,\gamma(E',E,\omega)(s),E'),{\p \gamma{E'}}(E',E,\omega)(s)\ra ds .
\]
\end{lemma}

\begin{proof}

A. We have for $E'\not=E$
\bea\label{le-reg-2}
&
h(x,\omega,E',E)=
\int_{0}^{2\pi}\la \nabla_S\psi(x,\gamma(E',E,\omega)(s),E'),{\p \gamma{E}}(E',E,\omega)(s)\ra ds \nonumber\\
&
=
\int_{0}^{2\pi}\la \nabla_S\psi(x,\gamma(E',E,\omega)(s),E),{\p \gamma{E}}(E',E,\omega)(s)\ra ds \nonumber\\
&
+
\int_{0}^{2\pi}\Big(\la \nabla_S\psi(x,\gamma(E',E,\omega)(s),E')
- \nabla_S\psi(x,\gamma(E',E,\omega)(s),E),{\p \gamma{E}}(E',E,\omega)(s)\ra\Big) ds\nonumber\\
&
=:h_1(x,\omega,E',E)+h_2(x,\omega,E',E).
\eea
Since $\psi\in C(\ol G,C^{1}(I,C^1(S)))$
\bea\label{le-reg-2a}
&
|h_2(x,\omega,E',E)|\leq C_1|E'-E| \int_0^{2\pi}\n{{\p \gamma{E}}(E',E,\omega)(s)}ds\nonumber\\
&
=2\pi\ C_1|E'-E|\ |(\partial_E\mu)(E',E)|{{1}\over{\sqrt{1-\mu(E',E)^2}}}
\leq C_1'\ |E'-E|^{1\over 2}
\eea
where we noticed that $1-\mu(E',E)^2=1-{{E(E'+2)}\over{E'(E+2)}}={{2(E'-E))}\over{E'(E+2)}}$. Note that
\be\label{apu}
h(x,\omega,E',E)-h(x,\omega,E,E)=h_1(x,\omega,E',E)-h(x,\omega,E,E)+h_2(x,\omega,E',E)
\ee
and so it suffices to consider the term $h_1(x,\omega,E',E)-h(x,\omega,E,E)$. This is done in the next paragraph.
 
B.
Let $\zeta=\zeta(E',E,\omega,s)=H_\omega^{-1}(\omega')\sim \xi(E',E,\omega,s)=\xi,\ \omega'=\gamma(E',E,\omega)(s)$ and
$\hat\eta=\hat\eta(E',E,\omega,s)={\p {\gamma}{E}}(E',E,\omega)(s)$.
In virtue of the Taylor's formula (cf. (\ref{t-ex}))  we  have
for $E'\not=E$ 
\bea\label{ex-9-a}
&
h_1(x,\omega,E',E)\nonumber\\
&
=
\int_0^{2\pi}
\la (\nabla_S\psi)(x,\gamma(E',E,\omega)(s),E),{\p {\gamma}{E}}(E',E,\omega)(s)\ra ds
\nonumber\\
&
=
\int_0^{2\pi}
\la (\nabla_S\psi\circ H_\omega)(x,.,E),{\p {\gamma}{E}}(E',E,\omega)(s)\ra(\xi) ds
\nonumber\\
&
=\int_0^{2\pi}\la (\nabla_S\psi)(x,\omega,E),{\p {\gamma}{E}}(E',E,\omega)(s)\ra ds\nonumber\\
&
+ 
\sum_{j=1}^2\int_0^{2\pi}\partial_j\big(\la(\nabla_S\psi\circ H_\omega)(x,.,E),\hat\eta(E',E,\omega,s)\ra\big)(0)\xi_j(E',E,\omega,s)ds
\nonumber\\
&
+
\int_0^{2\pi}
(R_\omega\psi(x,\cdot,E))(\omega')\Big(\zeta(E',E,\omega, s),\zeta(E',E,\omega,\hat\eta(E',E,\omega,s)\Big) ds
\eea
where the residual $(R_\omega\psi(x,\cdot,E))(\omega')$ is 
\bea\label{t-residual-b}
&
R_\omega(\psi(x,.,E))(\omega')(\zeta(E',E,\omega, s),\zeta(E',E,\omega, s),\hat\eta(E',E,\omega,s)
\nonumber\\
&
=
\sum_{|\alpha|=2}{{|\alpha|}\over{\alpha !}}\int_0^1(1-t)\partial_\xi^\alpha
\big(\la (\nabla_S\psi)(x,.,E)\circ H_\omega),\hat\eta\ra\big) (t\xi)dt \cdot \xi^\alpha. 
\eea  
Hence
\bea\label{u1}
&
h_1(x,\omega,E',E)-h(x,\omega,E,E)
=
\int_{0}^{2\pi}
\la (\nabla_S\psi)(x,\omega,E),
{\p {\gamma}{E}}(E',E,\omega)(s)\ra ds\nonumber\\
&
+
\sum_{j=1}^2\int_0^{2\pi}\partial_j\big(\la(\nabla_S\psi\circ H_\omega)(x,.,E),\hat\eta(E',E,\omega,s)\ra\big)(0)\xi_j(E',E,\omega,s)ds\nonumber\\
&
-
\lim_{E'\to E}
\sum_{j=1}^2\int_0^{2\pi}\partial_j\big(\la(\nabla_S\psi\circ H_\omega)(x,.,E),\hat\eta(E',E,\omega,s)\ra\big)(0)\xi_j(E',E,\omega,s)ds
\nonumber\\
&
+
\int_0^{2\pi}
R_\omega(\psi(x,.,E))(\omega')(\zeta(E',E,\omega, s),\zeta(E',E,\omega, s),\hat\eta(E',E,\omega,s) ds.
\eea
Similarly as in the proof of Theorem \ref{ad-le-6} we find that
\bea\label{u2}
&
\Big|\int_0^{2\pi}
R_\omega(\psi(x,.,E))(\omega')\Big(\zeta(E',E,\omega, s),\zeta(E',E,\omega, s),\hat\eta(E',E,\omega,s)\Big) ds\Big|
\nonumber\\
&
\leq
C\sqrt{1-\mu(E',E)}\leq C'|E'-E|^{1\over 2}
\eea
and that
\be\label{ex-10-a}
\int_{0}^{2\pi}
\la (\nabla_S\psi)(x,\omega,E),
{\p {\gamma}{E}}(E',E,\omega)(s)\ra ds=0.
\ee

Consider the term 
\[
&
\sum_{j=1}^2\int_0^{2\pi}\partial_j\big(\la(\nabla_S\psi\circ H_\omega)(x,.,E),\hat\eta(E',E,\omega,s)\ra\big)(0)\xi_j(E',E,\omega,s)ds\nonumber\\
&
-
\lim_{E'\to E}
\sum_{j=1}^2\int_0^{2\pi}\partial_j\big(\la(\nabla_S\psi\circ H_\omega)(x,.,E),\hat\eta(E',E,\omega,s)\ra\big)(0)\xi_j(E',E,\omega,s)ds.
\]
In due to proof of Lemma \ref{aij} (here we again denote shortly $\mu=\mu(E',E)$) 
\bea\label{le-reg-6}
&
\int_0^{2\pi}\partial_j\big(\la(\nabla_S\psi\circ H_\omega)(x,.,E),\hat\eta(E',E,\omega,s)\ra\big)(0)\xi_j(E',E,\omega,s)ds
=
{{\arccos(\mu)}\over{\sqrt{1-\mu^2}}}
\int_0^{2\pi}\nonumber\\
&
\partial_j\big(\la(\nabla_S\psi\circ H_\omega)(x,.,E),
R(\omega)\big(-(\partial_{E}\mu)\mu\ \cos(s),
-(\partial_{E}\mu)\mu\ \sin(s),\sqrt{1-\mu^2}\partial_{E}\mu\big)
\ra\big)(0)ds
\nonumber
\\
&
\cdot
\big(
R(\omega)\big(\cos(s),\sin(s),0\big)
\big)_j ds
\nonumber\\
&
=
{{\arccos(\mu)}\over{\sqrt{1-\mu^2}}}\int_0^{2\pi}b_j(\omega,s)q_{j}(E',E,\omega,s) ds
\eea
where 
\[
b_j(\omega,s)
:=
\big(
R(\omega)\big(\cos(s),\sin(s),0\big)
\big)_j=\la R(\omega)\big(\cos(s),\sin(s),0\big),\ol\Omega_j\ra
\]
and
\[
&
q_{j}(E',E,\omega,s)\\
&
:=
\partial_j\big(\la(\nabla_S\psi\circ H_\omega)(x,.,E),
R(\omega)\big(-(\partial_{E}\mu)\mu\ \cos(s),
-(\partial_{E}\mu)\mu\ \sin(s),\sqrt{1-\mu^2}\partial_{E}\mu\big)
\ra\big)(0)
\]

Note that
\[
&
\lim_{E'\to E}\sum_{j=1}^2\int_0^{2\pi}\partial_j\big(\la(\nabla_S\psi\circ H_\omega)(x,.,E),\hat\eta(E',E,\omega,s)\ra\big)(0)\xi_j(E',E,\omega,s)ds
\\
&
=
\int_0^{2\pi}b_j(\omega,s)q_{j}(E,E,\omega,s) ds
.
\]
Hence we have
\bea\label{le-reg-20}
&
\Big|\int_0^{2\pi}\partial_j\big(\la(\nabla_S\psi\circ H_\omega)(x,.,E),\hat\eta(E',E,\omega,s)\ra\big)(0)\xi_j(E',E,\omega,s)ds\nonumber\\
&
-
\lim_{E'\to E}
\int_0^{2\pi}\partial_j\big(\la(\nabla_S\psi\circ H_\omega)(x,.,E),\hat\eta(E',E,\omega,s)\ra\big)(0)\xi_j(E',E,\omega,s)ds\Big|\nonumber\\
&
=
\Big|{{\arccos(\mu(E',E))}\over{\sqrt{1-\mu(E',E)^2}}}\int_0^{2\pi}b_j(\omega,s)q_{j}(E',E,\omega,s)ds-\int_0^{2\pi}b_j(\omega,s)q_{j}(E,E,\omega,s)ds\Big|\nonumber\\
&
\leq
\Big|{{\arccos(\mu(E',E))}\over{\sqrt{1-\mu(E',E)^2}}}-1\Big|\ \Big|\int_0^{2\pi}b_j(\omega,s)q_j(E',E,\omega,s)ds\Big|\nonumber\\
&
+\Big|\int_0^{2\pi}b_j(\omega,s)q_{j}(E',E,\omega,s)ds-\int_0^{2\pi}b_j(\omega,s)q_{j}(E,E,\omega,s)ds\Big|.
\eea
Recall that $\arccos(\mu)=\arcsin(\sqrt{1-\mu^2})$ and
\[
\arcsin(x)
=x+{1\over {2}}{{x^3}\over 3}+{{1\cdot 3}\over{2\cdot 4}}{{x^5}\over 5}+...,\ |x|<1.
\]
Hence we find that 
\be\label{le-reg-11}
\Big|{{\arccos(\mu(E',E))}\over{\sqrt{1-\mu(E',E)^2}}}-1\Big|
\leq C_4 (1-\mu(E',E)^2)\leq C_4'|E'-E|.
\ee
Since $\psi\in C(\ol G,C^{1}(I,C^{3}(S)))$
\[
\Big|\int_0^{2\pi}b_j(\omega,s)q_j(E',E,\omega,s)ds\Big|\leq C_5
\]
and
\bea\label{le-reg-18}
&
\Big|\int_0^{2\pi}b_j(\omega,s)q_{j}(E',E,\omega,s)ds-\int_0^{2\pi}b_j(\omega,s)q_{j}(E,E,\omega,s)ds\Big|\nonumber\\
&
\leq 
\int_0^{2\pi}C_6|b_{j}(\omega,s)|
\Big(
[(\partial_{E}\mu)(E',E)\mu(E',E)-(\partial_{E}\mu)(E,E)]^2+|(\partial_{E}\mu)(E',E)|^2(1-\mu(E',E)^2)\Big)^{{1\over 2}}ds.
\eea
Since $E_0>0$ we see that $\mu\in C^2(I^2)$ and so
\be\label{le-reg-12}
|(\partial_{E}\mu)(E',E)\mu(E',E)-(\partial_{E}\mu)(E,E)|
=|(\partial_{E}\mu\cdot \mu)(E',E)-(\partial_{E}\mu\cdot \mu)(E,E)|
\leq C_7|E'-E|.
\ee
Moreover,
\be 
|1-\mu(E',E)^2|\leq C_8|E'-E|.
\ee
Hence by (\ref{le-reg-20}), (\ref{le-reg-11}) and (\ref{le-reg-18})
\bea\label{le-reg-23}
&
\Big|\sum_{j=1}^2\int_0^{2\pi}\partial_j\big(\la(\nabla_S\psi\circ H_\omega)(x,.,E),\hat\eta(E',E,\omega,s)\ra\big)(0)\xi_j(E',E,\omega,s)ds\nonumber\\
&
-
\lim_{E'\to E}
\sum_{j=1}^2\int_0^{2\pi}\partial_j\big(\la(\nabla_S\psi\circ H_\omega)(x,.,E),\hat\eta(E',E,\omega,s)\ra\big)(0)\xi_j(E',E,\omega,s)ds\Big|
\nonumber\\
& 
\leq  C |E'-E|^{1\over 2}.
\eea
As a conclusion we see that the assertion follows from  (\ref{le-reg-2a}), (\ref{apu}), (\ref{u1}), (\ref{u2}), (\ref{ex-10-a}) and (\ref{le-reg-23}).
This completes the proof.

\end{proof}

Tracking similarly as in Corollary \ref{x-le-cor} the constants appearing in the estimates used in the proof of the above lemma
we get the following corollary

\begin{corollary}\label{le-reg-cor}
Suppose that 
$\psi\in C(\ol G,C^{1}(I,C^{3}(S)))$. Then the mapping $h_2$ obeys 
\be\label{le-reg-1-k}
|h_2(x,\omega,E',E)-h_2(x,\omega,E,E)|\leq C_2\n{\psi}_{ C(\ol G,C^{1}(I,C^{3}(S)))}(E'-E)^{{1\over 2}}, \ E'\geq E
\ee
where $C_2$ is independent of $\psi$. 
\end{corollary}

The previous Lemma \ref{le-reg} guarantees that the below p.f.-integrals exit (we skip  detailed treatments).

We prove the following refined form for the exact transport operator related to M\o ller interaction.

\begin{theorem}\label{ex-eq}
Let $\hat\sigma_1\in C(\ol G\times I,C^1(I')),\ \hat\sigma_2\in C(\ol G, C^{2}(I'\times  I))$ and let $\psi\in C(\ol G,C^2(I,C^3(S)))$.
Then the exact transport operator 
\bea\label{exth-1}
&
T\psi=-
{\s H}_2\big((\ol{\s K}_2\psi)(x,\omega,\cdot,E)\big)(E) 
+
{\s H}_1\big((\ol{\s K}_1\psi)(x,\omega,\cdot,E)\big)(E)\nonumber\\
&
+\omega\cdot\nabla_x\psi+\Sigma\psi
-K_r\psi
\eea
can be expressed by
\bea\label{exth-2}
&
(T\psi)(x,\omega,E)=
-{\partial\over{\partial E}}\Big(
{\s H}_1\big((\ol{\s K}_2\psi)(x,\omega,\cdot,E)\big)(E)\Big)
-
2\pi\ 
\hat\sigma_2(x,E,E){\p \psi{E}}(x,\omega,E)\nonumber\\
&
-\hat\sigma_2(x,E,E)
\sum_{|\alpha|\leq 2}a_{\alpha}(E,\omega)(\partial_{\omega}^\alpha\psi)(x,\omega,E)\nonumber\\
&
+{\rm p.f.}\int_E^{E_m}{1\over{E'-E}}
\hat\sigma_2(x,E',E)\int_{0}^{2\pi}\la \nabla_S\psi(x,\gamma(E',E,\omega)(s),E'),{\p \gamma{E}}(E',E,\omega)(s)\ra ds dE'
\nonumber\nonumber\\
{}&
+
{\rm p.f.}\int_E^{E_m}{1\over{E'-E}}
{\p {\hat\sigma_2}{E}}(x,E',E)\int_{0}^{2\pi}\psi(x,\gamma(E',E,\omega)(s),E')ds dE'\nonumber\\
&
+
{\s H}_1\big((\ol{\s K}_1\psi)(x,\omega,\cdot,E)\big)(E)
-
2\pi\ 
{\p {\hat\sigma_2}{E'}}(x,E,E)\psi(x,\omega,E)
\nonumber\\
&
+\omega\cdot\nabla_x\psi+\Sigma(x,\omega,E)\psi
-(K_r\psi)(x,\omega,E)
\eea
where the term
$\sum_{|\alpha|\leq 2}a_{\alpha}(E,\omega)(\partial_{\omega}^\alpha\psi)(x,\omega,E)$ is as above in (\ref{pd-term}).
\end{theorem}

\begin{proof}

A. At first we show that (cf. \cite{tervo18-up}, section 3.2)
\bea\label{exth-3}
&
{\s H}_2\big((\ol{\s K}_2\psi)(x,\omega,\cdot,E)\big)(E) =
{\partial\over{\partial E}}\Big(
{\s H}_1\big((\ol{\s K}_2\psi)(x,\omega,\cdot,E)\big)(E)\Big)
-
{\s H}_1\big(({\p {(\ol{\s K}_2\psi)}E}(x,\omega,.,E)\big)(E)
\nonumber\\
{}&
+
{\p {(\ol{\s K}_2\psi)}{E'}}(x,\omega,E',E)_{\big|E'=E}.
\eea
Let
\[
f(x,\omega,E',E):=(\ol{\s K}_2\psi)(x,\omega,E',E).
\]
Then
\bea\label{exth-4}
&
{\p f{E'}}
={\p {\hat\sigma_2}{E'}}(x,E',E)\int_{0}^{2\pi}\psi(x,\gamma(E',E,\omega)(s),E')ds\nonumber\\
&
+
\hat\sigma_2(x,E',E)\int_{0}^{2\pi}\la \nabla_S\psi(x,\gamma(E',E,\omega)(s),E'),{\p \gamma{E'}}(E',E,\omega)(s)\ra ds\nonumber\\
&
+
\hat\sigma_2(x,E',E)\int_{0}^{2\pi}{\p \psi{E'}}(x,\gamma(E',E,\omega)(s),E')ds.
\eea
Since $\hat\sigma_2\in C(\ol G, C^{2}(I'\times  I))$ and $\psi\in C(\ol G,C^2(I,C^3(S)))$ we find
in due to Theorem \ref{ad-le-6} and Lemma \ref{e-th-1}  that $f\in C(\ol G\times S,C^1(I'\times I))$ and for $E'>E$
\[
\Big|{\p f{E'}}(x,\omega,E',E)-{\p f{E'}}(x,\omega,E,E)\Big|\leq C(E'-E)^{1\over 2}
\]
(we omit here further details).
Thus  the relation (\ref{exth-3}) is valid (see the note after  Lemma \ref{h1-lemma}).

B.
We have for $E'>E$
\bea\label{exth-4-a}
&
{\p {(\ol{\s K}_2\psi)}E}(x,\omega,E',E)
={\p {\hat\sigma_2}{E}}(x,E',E)\int_{0}^{2\pi}\psi(x,\gamma(E',E,\omega)(s),E')ds\nonumber\\
&
+
\hat\sigma_2(x,E',E)\int_{0}^{2\pi}\la \nabla_S\psi(x,\gamma(E',E,\omega)(s),E'),{\p \gamma{E}}(E',E,\omega)(s)\ra ds
\eea
and similarly recalling that $\gamma(E,E,\omega)(s)=\omega$
\bea\label{exth-5}
&
{\p {(\ol{\s K}_2\psi)}{E'}}(x,\omega,E',E)_{\big|E'=E}
=
2\pi\ 
{\p {\hat\sigma_2}{E'}}(x,E,E)\psi(x,\omega,E')\nonumber\\
&
+
\hat\sigma_2(x,E,E)\Big(\int_{0}^{2\pi}\la \nabla_S\psi(x,\gamma(E',E,\omega)(s),E'),{\p \gamma{E'}}(E',E,\omega)(s)\ra ds\Big)_{|E'=E}
\nonumber\\
&
+
2\pi\ 
\hat\sigma_2(x,E,E){\p \psi{E}}(x,\omega,E).
\eea
In virtue of (\ref{var-for-14})
\bea\label{exth-6}
&
\Big(\int_{0}^{2\pi}\la \nabla_S\psi(x,\gamma(E',E,\omega)(s),E'),{\p \gamma{E'}}(E',E,\omega)(s)\ra ds\Big)_{|E'=E}
\nonumber\\
&
=
2\pi\ (\partial_{E'}\mu)(E,E)(\omega\cdot\nabla_S\psi)(x,\omega,E)
+
\sum_{|\alpha|\leq 2}a_{\alpha}(E,\omega)(\partial_{\omega}^\alpha\psi)(x,\omega,E).
\eea
Recall that by (\ref{nab0})
\[
(\omega\cdot\nabla_S\psi)(x,\omega,E)=0.
\]
Hence the assertion follows by combining
(\ref{exth-3}), (\ref{exth-4-a}), (\ref{exth-5}) and (\ref{exth-6}).

\end{proof}

\vspace{4\baselineskip}

\subsection{Related formal adjoints}\label{ad-sec}

\subsubsection{Formal adjoints of restricted collision and convection-scattering   operators}

We assume
that  the restricted collision operator is  the sum
\be\label{esols1} 
K_r=K^1+K^2+K^3.
\ee
Here $K^1$ is of the form 
\[
(K^1\psi)(x,\omega,E)=\int_{S'\times I'}\sigma^1(x,\omega',\omega,E',E)\psi(x,\omega',E')d\omega' dE',  
\]
where $\sigma^1:G\times S^2\times I^2\to\R$ is a non-negative measurable function such that 
\bea\label{ass5-a}
&\int_{S'\times I'}\sigma^1(x,\omega',\omega,E',E)d\omega' dE'\leq M_1,\nonumber\\
&\int_{S'\times I'}\sigma^1(x,\omega,\omega',E,E')d\omega' dE'\leq M_2,
\eea
for a.e. $(x,\omega,E)\in G\times S\times I$. $K^1$ is related e.g. to the Bremsstrahlung.

The operator
$K^2$ is of the form 
\[
(K^2\psi)(x,\omega,E)=\int_{ S'}\sigma^2(x,\omega',\omega,E)\psi(x,\omega',E) d\omega',  
\]
where $\sigma^2:G\times S^2\times I\to\R$ is a non-negative  measurable function  such that
\bea\label{ass7}
&\int_{S'}\sigma^2(x,\omega',\omega,E)d\omega'\leq M_1,\nonumber\\
&\int_{S'}\sigma^2(x,\omega,\omega',E) d\omega'\leq M_2,
\eea
for a.e. $(x,\omega,E)\in G\times S\times I$. $K^2$ models the elastic scattering.

Finally, $K^3$ is of the form 
\[
(K^3\psi)(x,\omega,E)
=
\int_{I'}\int_{0}^{2\pi}
\hat\sigma^3(x,E',E)
\psi(x,\gamma(E',E,\omega)(s),E')ds dE'
\]
where
$\hat{\sigma}^3:G\times I^2\to\R$ is a non-negative measurable function such that
\bea\label{ass-8}
&\int_{I'}\hat{\sigma}^3(x,E',E)dE'\leq M_1, \nonumber\\
&\int_{I'}\hat{\sigma}^3(x,E,E')dE'\leq M_2,
\eea
for a.e. $(x,E)\in G\times I$. 
$K^3$ relates e.g. to the M\o ller scattering.

Recall that the restricted collision operator $K_r$ is a bounded operator
$L^2(G\times S\times I)\to L^2(G\times S\times I)$ since it obeys the (partial) \emph{Schur criterion for boundedness} (\cite{tervo18}, \cite{tervo18-up}). Furthermore, we have

\begin{theorem}\label{Kr-adjoint}
The adjoint $K_r^*$ of $K_r=K^1+K^2+K^3$ is given by
\bea
& 
(K_r^*v)(x,\omega,E)=
\int_{S'\times I'}\sigma^1(x,\omega,\omega',E,E')v(x,\omega',E')d\omega' dE'
+
\int_{S'}\sigma^2(x,\omega,\omega',E)v(x,\omega',E)d\omega' \nonumber\\
&
+
\int_{ I'}\int_0^{2\pi}\hat\sigma^3(x,E,E')v(x,\gamma(E,E',\omega)(s),E')ds  dE'.
\eea

\end{theorem}

\begin{proof}
The adjoint $K_r^*$ is $(K^1)^*+(K^2)^*+(K^3)^*$. We compute $(K^3)^*$. 
The adjoints $(K^1)^*,\ (K^2)^*$ are similarly computed.
We have by Corollary \ref{sp-ch-le} for $\psi,\ v \in L^2(G\times S\times I)$
\bea\label{Kr-3}
&
\la K^3\psi,v\ra_{L^2(G\times S\times I)}
=
\int_G\int_S\int_I(K^3\psi)(x,\omega,E)v(x,\omega,E) dE d\omega dx\nonumber
\\
&
=
\int_G\int_S\int_I\int_{I'}\int_0^{2\pi}\hat\sigma^3(x,E',E)\psi(x,\gamma(E',E,\omega)(s),E')v(x,\omega,E)ds dE' dE d\omega dx\nonumber\\
&
=
\int_G\int_I\int_{I'}\hat\sigma^3(x,E',E)\Big(\int_S\int_0^{2\pi}\psi(x,\gamma(E',E,\omega)(s),E')v(x,\omega,E)ds d\omega\Big)dE' dE  dx\nonumber\\
&
=
\int_G\int_I\int_{I'}\hat\sigma^3(x,E',E)\Big(\int_{S'}\int_0^{2\pi}\psi(x,\omega',E')v(x,\gamma(E',E,\omega')(s),E)ds d\omega'\Big)dE' dE  dx\nonumber\\
&
=
\int_G\int_{S'}\int_{I'}\psi(x,\omega',E')\Big(\int_{I}\int_0^{2\pi}\hat\sigma^3(x,E',E)v(x,\gamma(E',E,\omega')(s),E)ds dE \Big)dE' d\omega'  dx
\eea
which completes the proof.
\end{proof}

Let
\be\label{sigma}
\Sigma\in L^\infty(G\times S\times I).
\ee
Define
\[
(A_0\psi)(x,\omega,E):=\omega\cdot\nabla_x\psi+\Sigma\psi-K_r\psi.
\]
Then the formal adjoint of $A_0$ that is,
\[
(A_0^*v)(x,\omega,E)=-\omega\cdot\nabla_xv+\Sigma^*v-K_r^*v
\]
where $\Sigma^*=\Sigma$ and $v\in C^1_0(G\times S\times I^\circ)$.

\subsubsection{Formal adjoints of Hadamard finite part integral operators}\label{ad-hada-op}

Let
\[
&
(A_1\psi)(x,\omega,E):={\s H}_1\big((\ol {\s K}_1\psi)(x,\omega,\cdot,E)\big)(E)\nonumber\\
&
=
{\rm p.f.}\int_E^{E_m}{1\over{E'-E}}\hat\sigma_1(x,E',E)\int_0^{2\pi}\psi(x,\gamma(E',E,\omega)(s),E')ds dE'
\]
where $\psi\in C(\ol G\times S,C^\alpha(I))\cap C(\ol G\times I,C^\alpha(S))$.

\begin{theorem}\label{adjointA1-th} 
Suppose that $\hat\sigma_1\in C(\ol G,C^\alpha(I^2)),\ \alpha>0$. Then the formal adjoint of $A_1$,
\be\label{var-for-16-d}
(A_1^*v)(x,\omega',E'):=
{\rm p.f.}\int_{E_0}^{E'}{1\over{E'-E}}\hat\sigma_1(x,E',E)
\int_0^{2\pi}
v(x,\gamma(E',E,\omega')(s),E)dsdE
\ee
for $v\in C(\ol G,C^\alpha(S\times I))$. 
\end{theorem}

\begin{proof}

By    (\ref{var-for-12}) and by Corollary \ref{sp-ch-le}  we have
\bea\label{v-3}
&
\la  A_1\psi
,v \ra_{L^2(G\times S\times I)}\nonumber\\
&
=
\int_G\int_S\int_I\Big({\rm p.f.}\int_E^{E_m}{1\over{E'-E}}\hat\sigma_1(x,E',E) \nonumber\\
&
\cdot
\int_0^{2\pi}\psi(x,\gamma(E',E,\omega)(s),E')ds dE'\Big)
v(x,\omega,E) dE d\omega dx\nonumber\\
&
=
\int_G\int_I\Big({\rm p.f.}\int_E^{E_m}{1\over{E'-E}}\hat\sigma_1(x,E',E) \nonumber\\
&
\cdot
\Big[\int_S\int_0^{2\pi}\psi(x,\gamma(E',E,\omega)(s),E')v(x,\omega,E)dsd\omega\Big] dE'\Big)
 dE  dx\nonumber\\
&
=
\int_G\int_I\Big({\rm p.f.}\int_E^{E_m}{1\over{E'-E}}\hat\sigma_1(x,E',E) \nonumber\\
&
\cdot
\Big[\int_{S'}\int_0^{2\pi}\psi(x,\omega',E')v(x,\gamma(E',E,\omega')(s),E)dsd\omega'\Big] dE'\Big)
 dE  dx\nonumber\\ 
&
=
\int_G\int_I\Big({\rm p.f.}\int_E^{E_m}{1\over{E'-E}}\hat\sigma_1(x,E',E)f(x,E',E) dE'\Big)
dE  dx 
\eea
where
\[
f(x,E',E):=
\int_{S'}\int_0^{2\pi}\psi(x,\omega',E')v(x,\gamma(E',E,\omega')(s),E)dsd\omega'.
\]
Hence by Lemma \ref{ad-le-1} and by (\ref{var-for-12}) we obtain 
\bea\label{v-3a}
&
\la A_1\psi
,v \ra_{L^2(G\times S\times I)}
=
\int_G\int_{I'}\Big({\rm p.f.}\int_{E_0}^{E'}{1\over{E'-E}}\hat\sigma_1(x,E',E)f(x,E',E) 
dE\Big) dE'  dx\nonumber\\
&
=
\int_G\int_{I'}\Big({\rm p.f.}\int_{E_0}^{E'}{1\over{E'-E}}\hat\sigma_1(x,E',E) \nonumber \\
&
\cdot
\int_{S'}\int_0^{2\pi}
\psi(x,\omega',E')v(x,\gamma(E',E,\omega')(s),E)dsd\omega' dE\Big) dE' dx
\nonumber\\
&
=
\int_G\int_{S'}\int_{I'}\psi(x,\omega',E')\Big({\rm p.f.}\int_{E_0}^{E'}{1\over{E'-E}}\hat\sigma_1(x,E',E)
\nonumber \\
&
\cdot
\int_0^{2\pi}
v(x,\gamma(E',E,\omega')(s),E)dsdE\Big) dE' d\omega' dx
\nonumber\\
&
=
\la \psi
,A_1^* v \ra_{L^2(G\times S\times I)}
\eea
where $A_1^*$ is 
\[
(A_1^*v)(x,\omega',E'):=
{\rm p.f.}\int_{E_0}^{E'}{1\over{E'-E}}\hat\sigma_1(x,E',E)
\int_0^{2\pi}
v(x,\gamma(E',E,\omega')(s),E)dsdE.
\]
This completes the proof
\end{proof}

Finally, define (the rest of the transport operator (\ref{exth-2}))
\bea\label{A2}
&
(A_2\psi)(x,\omega,E):= -
{\partial\over{\partial E}}\Big(
{\s H}_1\big((\ol{\s K}_2\psi)(x,\omega,\cdot,E)\big)(E)\Big)
-
2\pi\ 
\hat\sigma_2(x,E,E){\p \psi{E}}(x,\omega,E)\nonumber\\
&
-\hat\sigma_2(x,E,E)
\sum_{|\alpha|\leq 2}a_{\alpha}(E,\omega)(\partial_{\omega}^\alpha\psi)(x,\omega,E)\nonumber\\
&
+{\rm p.f.}\int_E^{E_m}{1\over{E'-E}}
\hat\sigma_2(x,E',E)\int_{0}^{2\pi}\la \nabla_S\psi(x,\gamma(E',E,\omega)(s),E'),{\p \gamma{E}}(E',E,\omega)(s)\ra ds dE'
\nonumber\nonumber\\
&
+
{\rm p.f.}\int_E^{E_m}{1\over{E'-E}}
{\p {\hat\sigma_2}{E}}(x,E',E)\int_{0}^{2\pi}\psi(x,\gamma(E',E,\omega)(s),E')ds dE'\nonumber\\
&
-
2\pi\ 
{\p {\hat\sigma_2}{E'}}(x,E,E)\psi(x,\omega,E').
\eea
The formal adjoint $A_2^*$ can be computed by techniques utilized in the next section \ref{var-for} and it is (see Theorem \ref{var-th-a})
\be\label{ad-a2}
(A_2^*v)(x,\omega',E')
=
-
{\rm p.f.}\int_{E_0}^{E'}{1\over{(E'-E)^2}}
\hat{\sigma}_{2}(x,E',E) 
\cdot
\int_{0}^{2\pi}v(x,\gamma(E',E,\omega)(s),E)ds  dE
\ee
for appropriates $v$.

The formal adjoint of $T$ is given by
\[
T^*=A_0^*+A_1^*+A_2^*.
\]

\section{Variational formulation of the transport problem}\label{var-for}

In this section we shall give a weak form, so called variational formulation,  of the hyper-singular transport problem for M\o ller-type interactions. The obtained final weak form (the combination of Theorems \ref{var-th-a} and \ref{final-bi-form}) decreases the level of singularities in integration containing only singularities of order one that is, the singularities of the form ${1\over{E'-E}}dE dE'$. The variational formulation is an essential step in order to show the existence of solutions by Lions-Lax-Milgram Theorem based methods. Additionally, it gives a platform needed for Galerkin finite element  methods.

We treatise a variational formulation of the initial inflow boundary value transport problem
\be\label{var-for-1-a}
T\psi =f,
\ee
\be\label{var-for-2-a}
\psi_{|\Gamma_-}=g,
\ee
\be\label{var-for-3-a}
\psi(.,.,E_m)=0
\ee
where $T$ is given by (\ref{exth-1}) or equivalently by (\ref{exth-2}).
For some partial integration techniques we additionally require that
\be\label{var-for-3-b}
\lim_{E\to E_m}\ln(E_m-E)\psi(x,\omega,E)=0
\ee
which is a stronger requirement than (\ref{var-for-3-a}). Assuming that $\psi\in C(\ol G\times S,C^1(I))$ for which $\psi(.,.,E_m)=0$ we have
\[
&
\lim_{E\to E_m}\ln(E_m-E)\psi(x,\omega,E)=
\lim_{E\to E_m}(E_m-E)\ln(E_m-E){{\psi(x,\omega,E)-\psi(x,\omega,E_m)}\over{E_m-E}}\\
&
=-0\cdot {\p \psi{E}}(.,.,E_m)
=0
\]
since by the L'Hospital's rule $\lim_{x\to 0^+}x\ \ln(x)=\lim_{x\to 0^+} {{\ln(x)}\over{1\over x}}=0$. Hence that condition (\ref{var-for-3-b}) holds for sufficient regular  $\psi$ which obeys (\ref{var-for-3-a}).
Moreover, we assume that $f\in L^2(G\times S\times I)$ and $g\in T^2(\Gamma)$.
The basic principle is to require that
\be\label{v-1}
\la T\psi,v\ra_{L^2(G\times S\times I)}=\la f,v\ra_{L^2(G\times S\times I)},\
v\in \mc D_{(0)}
\ee
for the solution $\psi$ which satisfies the conditions (\ref{var-for-2-a}) and (\ref{var-for-3-a}).  Here $\mc D_{(0)}$ is  the space of test functions 
\[
\mc D_{(0)}:=\{v\in C^1(\ol G,C^2(I,C^3(S)))|\ v(x,\omega,E_0)=0\}.
\]
Since $v\in C(\ol G\times S,C^1(I))$ we see that even the condition 
\be\label{limv}
\lim_{E'\to E_0}\ln(E'-E_0)v(x,\omega,E')=0
\ee
holds for $v$.
At first, the requirement (\ref{v-1}) is imposed in the classical sense that is, $\psi$ and $v$ are assumed to be smooth enough. The final weak formulation follows by appropriate extensions to Sobolev spaces (cf. \cite{tervo18-up}, section 6).

\begin{remark}
We emphasize that in real problems the modelling often comprises a coupled system of transport equations. However, the below computations (for a single equation) form an essential part of variational formulations.

\end{remark}

\subsection{Basic variational formulation}\label{basic-var-for}

Let $f_-$ ($f_+$) be the negative (positive) part of a function. Recall that
\be
f=f_+-f_-\quad {\rm and}\quad |f|=f_++f_-.
\ee
We prove

\begin{theorem}\label{var-th-a}
Suppose that the assumptions (\ref{sigma}), (\ref{ass5-a}), (\ref{ass7}), (\ref{ass-8}) are valid and that $\hat\sigma_1\in C(\ol G,C^1(I\times I'))$, $\hat\sigma_2\in C(\ol G, C^{2}(I\times I')$. 
Furthermore,
suppose that  $\psi\in C^1(\ol G,C^2( I,C^3(S)))$ 
is a solution of the problem (\ref{var-for-1-a}),  (\ref{var-for-2-a}), (\ref{var-for-3-a}).
Then 
\be\label{var-form} 
B(\psi,v)=F(v),\ v\in \mc D_{(0)}.
\ee
Here the bilinear form $B(.,.,)$ is
\be\label{fb}
B(\psi,v)=B_2(\psi,v)+B_1(\psi,v)+B_0(\psi,v)
\ee
where
\be\label{b0-a}
B_0(\psi,v):=
\la\psi,-\omega\cdot\nabla_x v+\Sigma^*v-K_r^* v\ra_{L^2(G\times S\times I)}
+\int_{\Gamma_+}(\omega\cdot\nu)_+\psi v\ d\sigma d\omega dE,
\ee
\bea\label{b1-a}
B_1(\psi,v)&:=
\int_G\int_{S'}\int_{I'}\psi(x,\omega',E')\cdot\nonumber\\
&
\Big({\rm p.f.}\int_{E_0}^{E'}{1\over{E'-E}}\hat\sigma_1(x,E',E)
\int_0^{2\pi}
v(x,\gamma(E',E,\omega')(s),E)dsdE\Big) dE' d\omega' dx
\eea
and
\bea\label{b2-a}
&
B_2(\psi,v)
:=
-\int_G\int_S\int_{I'}\psi(x,\omega,E')
\Big(
{\rm p.f.}\int_{E_0}^{E'}{1\over{(E'-E)^2}}
\hat{\sigma}_{2}(x,E',E) 
\nonumber \\
&
\cdot
\int_{0}^{2\pi}v(x,\gamma(E',E,\omega)(s),E)ds  dE\Big)dE' d\omega dx.
\eea
The linear form $F$ is
\be\label{lin-form}
F(v):=\la { f},v\ra_{L^2(G\times S\times I)}+\int_{\Gamma_-}(\omega\cdot\nu)_-{ g} v\ d\sigma d\omega dE.
\ee
\end{theorem}

\begin{proof}
A.
Utilizing the above notations we have $T=A_0+A_1+A_2$ and so
\bea\label{fb-a}
&
\la T\psi
, v \ra_{L^2(G\times S\times I)}=\la A_0\psi,v\ra_{L^2(G\times S\times I)}
+\la A_1\psi
, v \ra_{L^2(G\times S\times I)}+
\la A_2\psi
, v \ra_{L^2(G\times S\times I)}\nonumber\\
&
=B_0(\psi,v)+B_1(\psi,v)+B_2(\psi,v)
\eea
where $B_j(\psi,v):=\la A_j\psi,v\ra_{L^2(G\times S\times I)}$. Note that for $j=1,2$
\be\label{Aj-a}
(A_j\psi)(x,\omega,E)=-{\s H}_j\big((\ol{\s K}_j\psi)(x,\omega,\cdot,E)\big)(E).
\ee
As in \cite{tervo18-up}, section 6,  we see that
\bea\label{b0-b}
&
B_0(\psi,v)=
\la A_0\psi,v\ra_{L^2(G\times S\times I)}
=\la\psi,-\omega\cdot\nabla_x v+\Sigma^*v-K_r^* v\ra_{L^2(G\times S\times I)}\nonumber\\
&
+\int_{\Gamma_+}(\omega\cdot\nu)_+\psi v\ d\sigma d\omega dE 
-\int_{\Gamma_-}(\omega\cdot\nu)_-g v\ d\sigma d\omega dE
\eea
By the proof of Theorem \ref{adjointA1-th} (see (\ref{v-3a}))
\bea\label{b1-b}
&
B_1(\psi,v)=
\la A_1\psi,v\ra_{L^2(G\times S\times I)}
=\la \psi,A_1^*v\ra_{L^2(G\times S\times I)}\nonumber\\
&
=
\int_G\int_{S'}\int_{I'}\psi(x,\omega',E')\Big({\rm p.f.}\int_{E_0}^{E'}{1\over{E'-E}}\hat\sigma_1(x,E',E)
\nonumber\\
&
\cdot
\int_0^{2\pi}
v(x,\gamma(E',E,\omega')(s),E)dsdE\Big) dE' d\omega' dx,
\eea

We show in Part B below that $B_2(.,.)$ has the expression (\ref{b2-a}). Hence noting that 
\[
\la T\psi,v\ra_{L^2(G\times S\times I)}=\la f,v\ra_{L^2(G\times S\times I)}
\]
and combining (\ref{fb-a}), (\ref{b1-b}), (\ref{b0-b}) we obtain the assertion (\ref{var-form}).

B. Consider the term $B_2(\psi,v)$. We utilize the expression (\ref{Aj-a}) for $A_2$ (alternatively we could use the expression (\ref{A2})).
Denote
\[
f(x,E',E):=\hat{\sigma}_{2}(x,E',E)
\int_S\int_{0}^{2\pi}\psi(x,\omega,E')v(x,\gamma(E',E,\omega)(s),E)ds d\omega.
\]
Applying  the  formulas (\ref{var-for-9}), (\ref{vf-3}),   (\ref{adle3-1})
we obtain (the existence of the integral $\int_{G\times S\times I}$ below becomes clear in a course of computations)
\bea\label{var-for-8}
&
B_2(\psi,v)=\la A_2\psi
, v \ra_{L^2(G\times S\times I)}=\la -{\s H}_2\big((\ol{\s K}_2\psi)(x,\omega,\cdot,E)(E),v \ra_{L^2(G\times S\times I)}\nonumber\\
&
=
-
\int_{G\times S\times I}\Big({\rm p.f.}\int_E ^{E_m}{1\over{(E'-E)^2}}\Big(
\hat{\sigma}_{2}(x,E',E) \nonumber\\
&
\cdot
\int_{0}^{2\pi}\psi(x,\gamma(E',E,\omega)(s),E')ds\Big)dE'\Big)v(x,\omega,E)dx d\omega dE\nonumber\\
&
=
-
\int_{G}\Big[\int_I\Big({\rm p.f.}\int_E ^{E_m}{1\over{(E'-E)^2}}\Big(
\hat{\sigma}_{2}(x,E',E)\nonumber\\
&
\cdot
\int_S\int_{0}^{2\pi}\psi(x,\omega,E')v(x,\gamma(E',E,\omega)(s),E)ds d\omega\Big)dE'\Big)dE\Big] dx\nonumber\\
&
=
-
\int_{G}\Big[\int_I\Big({\rm p.f.}\int_E ^{E_m}{1\over{(E'-E)^2}}
f(x,E',E)dE'\Big)dE\Big] dx\nonumber\\
&
=-
\int_{G}\Big[\int_I\Big(
{\rm p.f.}\int_E^{E_m}{1\over{E'-E}}{\p f{E'}}(x,E',E)dE'+{\p f{E'}}(x,E,E)_+
\nonumber\\
&
-{1\over{E_m-E}}f(x,E_m,E)\Big)dE\Big] dx.
\eea
We find that in due to the initial condition (\ref{var-for-3-a})
\[
f(x,E_m,E)=
\hat{\sigma}_{2}(x,E',E)
\int_S\int_{0}^{2\pi}\psi(x,\omega,E_m)v(x,\gamma(E_m,E,\omega)(s),E)ds d\omega
=0
\] 
and so 
\bea\label{B2-v}
B_2(\psi,v)&=
-\int_{G}\int_I\Big(
{\rm p.f.}\int_E^{E_m}{1\over{E'-E}}{\p f{E'}}(x,E',E)dE'\Big)dEdx-
\int_{G}\int_I{\p f{E'}}(x,E,E)_+ dE dx
\nonumber\\
&
=:I_1+I_2.
\eea

By Lemma \ref{ad-le-1}
\be\label{var-for-8-a}
I_1=-\int_{G}\int_{I'}\Big(
{\rm p.f.}\int_{E_0}^{E'}{1\over{E'-E}}{\p f{E'}}(x,E',E)dE\Big)dE' dx.
\ee
Moreover, for $E'\not=E$ we have
\bea\label{var-for-9-a}
&
{\p f{E'}}(x,E',E)=
{\p {\hat{\sigma}_{2}}{E'}}(x,E',E)
\int_S\int_{0}^{2\pi}\psi(x,\omega,E')v(x,\gamma(E',E,\omega)(s),E)ds d\omega
\nonumber\\
&
+
\hat{\sigma}_{2}(x,E',E)
\int_S\int_{0}^{2\pi}{\p \psi{E'}}(x,\omega,E')v(x,\gamma(E',E,\omega)(s),E)ds d\omega
\nonumber\\
&
+
\hat{\sigma}_{2}(x,E',E)
\int_S\int_{0}^{2\pi}\psi(x,\omega,E')\la (\nabla_S v)(x,\gamma(E',E,\omega)(s),E),{\p \gamma{E'}}(E',E,\omega)(s)\ra ds d\omega
\eea
and so by (\ref{var-for-12}) 
\bea\label{var-for-8-b}
&
I_1=-
\int_{G}\int_{I'}\Big(
{\rm p.f.}\int_{E_0}^{E'}{1\over{E'-E}}
{\p {\hat{\sigma}_{2}}{E'}}(x,E',E) \nonumber\\
&
\cdot
\int_S\int_{0}^{2\pi}\psi(x,\omega,E')v(x,\gamma(E',E,\omega)(s),E)ds 
dE\Big) dE'd\omega dx
\nonumber\\
&
-\int_{G}\int_{I'}\Big({\rm p.f.}\int_{E_0}^{E'}{1\over{E'-E}}
\hat{\sigma}_{2}(x,E',E) \nonumber\\
&
\cdot
\int_S\int_{0}^{2\pi}{\p \psi{E'}}(x,\omega,E')v(x,\gamma(E',E,\omega)(s),E)ds  dE\Big) dE'd\omega dx
\nonumber\\
&
-\int_{G}\int_{I'}\Big({\rm p.f.}\int_{E_0}^{E'}{1\over{E'-E}}
\hat{\sigma}_{2}(x,E',E)
\int_S\int_{0}^{2\pi}\psi(x,\omega,E')\nonumber\\
&
\cdot
\la (\nabla_S v)(x,\gamma(E',E,\omega)(s),E),{\p \gamma{E'}}(E',E,\omega)(s)\ra ds d\omega
dE\Big)dE' dx
\nonumber\\
&
=
-
\int_{G}\int_{I'}\int_S\psi(x,\omega,E')\Big(
{\rm p.f.}\int_{E_0}^{E'}{1\over{E'-E}}
{\p {\hat{\sigma}_{2}}{E'}}(x,E',E) \nonumber\\
&
\cdot
\int_{0}^{2\pi}v(x,\gamma(E',E,\omega)(s),E)ds 
dE\Big) dE' d\omega dx\nonumber\\
&
-\int_{G}\int_{I'}\int_S\psi(x,\omega,E')\Big({\rm p.f.}\int_{E_0}^{E'}{1\over{E'-E}}
\hat{\sigma}_{2}(x,E',E)
\nonumber\\
&
\cdot \int_{0}^{2\pi}
\la (\nabla_S v)(x,\gamma(E',E,\omega)(s),E),{\p \gamma{E'}}(E',E,\omega)(s)\ra ds 
dE\Big)dE' d\omega dx\nonumber\\
&
-\int_{G}\int_{I'}\int_S{\p \psi{E'}}(x,\omega,E')\Big({\rm p.f.}\int_{E_0}^{E'}{1\over{E'-E}}
\hat{\sigma}_{2}(x,E',E) \nonumber\\
&
\cdot
\int_{0}^{2\pi}v(x,\gamma(E',E,\omega)(s),E)ds  dE\Big) dE' d\omega dx\nonumber\\
&
=:I_{1,1}+I_{1,2}+I_{1,3}
.
\eea

The last term 
\[
I_{1,3}&=-\int_{G}\int_{I'}\int_S{\p \psi{E'}}(x,\omega,E')\\
&
\cdot
\Big({\rm p.f.}\int_{E_0}^{E'}{1\over{E'-E}}
\hat{\sigma}_{2}(x,E',E)
\int_{0}^{2\pi}v(x,\gamma(E',E,\omega)(s),E)ds  dE\Big) dE' d\omega dx
\]
requires still refining. Let
\[
w(x,\omega,E'):=
{\rm p.f.}\int_{E_0}^{E'}{1\over{E'-E}}
\hat{\sigma}_{2}(x,E',E)
\int_{0}^{2\pi}v(x,\gamma(E',E,\omega)(s),E)ds  dE.
\] 
Then
\be \label{var-for-8-c}
I_{1,3}
=-
\int_{G}\int_{I'}\int_S{\p \psi{E'}}(x,\omega,E')w(x,\omega,E') dE' d\omega dx
\ee
By partial integration we obtain (partial integration is legitimate due to Lemma \ref{le-reg}; we omit details here)
\be\label{B2-12}
I_{1,3}=-\int_G\int_S\Big(
\Big|_{E_0}^{E_m}\psi(x,\omega,E')w(x,\omega,E')-
\int_{I'}\psi(x,\omega,E'){\p w{E'}}(x,\omega,E') dE'\Big)d\omega dx
\ee
Here we interpret 
\[
&
\Big|_{E_0}^{E_m}\psi(x,\omega,E')w(x,\omega,E')
:=\psi(x,\omega,E_m)w(x,\omega,E_m)-\lim_{E'\to E_0}\psi(x,\omega,E')w(x,\omega,E')\nonumber\\
&
=0-0=0
\]
where we used Lemma \ref{le-for-w} below.

Furthermore, we have by Lemma \ref{hadale}
\bea\label{var-for-8-d}
&
{\p w{E'}}(x,\omega,E')
=
{\partial\over{\partial E'}}\Big(
{\rm p.f.}\int_{E_0}^{E'}{1\over{E'-E}}
\hat{\sigma}_{2}(x,E',E)
\int_{0}^{2\pi}v(x,\gamma(E',E,\omega)(s),E)ds  dE\Big)
\nonumber\\
&
=
-{\rm p.f.}\int_{E_0}^{E'}{1\over{(E'-E)^2}}
\hat{\sigma}_{2}(x,E',E)
\int_{0}^{2\pi}v(x,\gamma(E',E,\omega)(s),E)ds  dE
\nonumber\\
&
+
{\rm p.f.}\int_{E_0}^{E'}{1\over{E'-E}}
{\p {\hat{\sigma}_{2}}{E'}}(x,E',E)
\int_{0}^{2\pi}v(x,\gamma(E',E,\omega)(s),E)ds  dE
\nonumber\\
&
+
{\rm p.f.}\int_{E_0}^{E'}{1\over{E'-E}}
\hat{\sigma}_{2}(x,E',E)
\int_{0}^{2\pi}\la (\nabla_S v)(x,\gamma(E',E,\omega)(s),E),
{\p {\gamma}{E'}}(E',E,\omega)\ra ds  dE
\nonumber\\
&
-
{\partial\over{\partial E}}\Big(\hat{\sigma}_{2}(x,E',E)\int_{0}^{2\pi}v(x,\gamma(E',E,\omega)(s),E)ds\Big)_{\Big|E=E'}
\eea
where
\[
&
{\partial\over{\partial E}}\Big(\hat{\sigma}_{2}(x,E',E)\int_{0}^{2\pi}v(x,\gamma(E',E,\omega)(s),E)ds\Big)_{\Big|E=E'}\nonumber\\
&
=
2\pi\ {\p {\hat{\sigma}_{2}}{E}}(x,E',E')v(x,\omega,E')
+
\hat{\sigma}_{2}(x,E',E'){\partial\over{\partial E}}\Big(\int_{0}^{2\pi}v(x,\gamma(E',E,\omega)(s),E)ds\Big)_{\Big|E=E'}.
\]

Hence we get 
\bea\label{var-for-14-a}
&
I_{1,3}=
\int_G\int_S\int_{I'}\psi(x,\omega,E')
\Big[-
{\rm p.f.}\int_{E_0}^{E'}{1\over{(E'-E)^2}}
\hat{\sigma}_{2}(x,E',E) \nonumber\\
&
\cdot
\int_{0}^{2\pi}v(x,\gamma(E',E,\omega)(s),E)ds  dE\Big)
\nonumber\\
&
+
{\rm p.f.}\int_{E_0}^{E'}{1\over{E'-E}}
{\p {\hat{\sigma}_{2}}{E'}}(x,E',E)
\int_{0}^{2\pi}v(x,\gamma(E',E,\omega)(s),E)ds  dE\Big)
\nonumber\\
&
+
{\rm p.f.}\int_{E_0}^{E'}{1\over{E'-E}}
\hat{\sigma}_{2}(x,E',E)
\int_{0}^{2\pi}\la (\nabla_S v)(x,\gamma(E',E,\omega)(s),E),
{\p {\gamma}{E'}}(E',E,\omega)\ra ds  dE
\nonumber\\
&
-2\pi\ {\p {\hat{\sigma}_{2}}{E}}(x,E',E')v(x,\omega,E') \nonumber\\
&
-
\hat{\sigma}_{2}(x,E',E'){\partial\over{\partial E}}\Big(\int_{0}^{2\pi}v(x,\gamma(E',E,\omega)(s),E)ds\Big)_{\Big|E=E'}\Big]dE' d\omega dx.
\eea

Consider the term $I_2$. Since $\gamma(E,E,\omega)(s)=\omega$ we have by (\ref{var-for-9}) 
\bea\label{var-for-15-b}
&
I_2=-\int_G\int_I
{\p f{E'}}(x,E,E)_+ dE dx=-2\pi\ \int_G\int_I\int_S
\psi(x,\omega,E){\p {\hat{\sigma}_{2}}{E'}}(x,E,E)
v(x,\omega,E) d\omega dE dx
\nonumber\\
&
-
2\pi\ \int_G\int_I\int_S{\p \psi{E}}(x,\omega,E)\hat{\sigma}_{2}(x,E,E) 
v(x,\omega,E) d\omega dE dx
\nonumber\\
&
-\int_G\int_I
\int_S\psi(x,\omega,E)\hat{\sigma}_{2}(x,E,E)
\nonumber\\
&
\cdot
\Big(\int_{0}^{2\pi}\la (\nabla_S v)(x,\gamma(E',E,\omega)(s),E),{\p \gamma{E'}}(E',E,\omega)(s)\ra ds\Big)_{\Big|E'=E} d\omega dE dx
\eea
where by partial integration 
\bea\label{var-for-15-c}
&
\int_G\int_I\int_S{\p \psi{E}}(x,\omega,E)\hat{\sigma}_{2}(x,E,E)v(x,\omega,E) d\omega dE dx\nonumber\\
&
=\int_G\int_S\Big(\Big|_{E_0}^{E_m}\hat{\sigma}_{2}(x,E,E)\psi(x,\omega,E)v(x,\omega,E)-\int_I\psi(x,\omega,E){\p {(\hat{\sigma}_{2}(x,E,E)v)}{E}}(x,\omega,E)\Big) d\omega.
\eea
Hence due to $\psi(.,.,E_m)=v(.,.,E_0)=0$
\bea\label{var-for-16-a}
&
I_2=
-\int_G\int_I\int_S\psi(x,\omega,E)
\Big(2\pi\ 
{\p {\hat{\sigma}_{2}}{E'}}(x,E,E)
v(x,\omega,E)
-2\pi\ {\p {(\hat{\sigma}_{2}(x,E,E)v)}{E}}(x,\omega,E)\Big) d\omega dE dx
\nonumber\\
&
-
\int_G\int_I\int_S\psi(x,\omega,E)
\hat{\sigma}_{2}(x,E,E)\nonumber\\
&
\cdot
\Big(\int_{0}^{2\pi}\la (\nabla_S v)(x,\gamma(E',E,\omega)(s),E),{\p \gamma{E'}}(E',E,\omega)(s)\ra ds\Big)_{\Big|E'=E} d\omega dE  dx.
\eea

Combining (\ref{var-for-8}), (\ref{var-for-8-b}), (\ref{var-for-14-a}), (\ref{var-for-16-a})  we obtain
\bea\label{var-for-15-a-1}
&
B_2(\psi,v)=
I_{1,1}+I_{1,2}+I_{1,3}+I_2\nonumber\\
&
=
-
\int_{G}\int_{I'}\int_S\psi(x,\omega,E')\Big(
{\rm p.f.}\int_{E_0}^{E'}{1\over{E'-E}}
{\p {\hat{\sigma}_{2}}{E'}}(x,E',E)
\int_{0}^{2\pi}v(x,\gamma(E',E,\omega)(s),E)ds 
dE\Big) dE' d\omega dx\nonumber\\
&
-\int_{G}\int_{I'}\int_S\psi(x,\omega,E')\Big({\rm p.f.}\int_{E_0}^{E'}{1\over{E'-E}}
\hat{\sigma}_{2}(x,E',E)
\nonumber\\
&
\cdot \int_{0}^{2\pi}
\la (\nabla_S v)(x,\gamma(E',E,\omega)(s),E),{\p \gamma{E'}}(E',E,\omega)(s)\ra ds 
dE\Big)dE' d\omega dx\nonumber\\
&
-
\int_G\int_S\int_{I'}\psi(x,\omega,E')
\Big(
{\rm p.f.}\int_{E_0}^{E'}{1\over{(E'-E)^2}}
\hat{\sigma}_{2}(x,E',E)
\int_{0}^{2\pi}v(x,\gamma(E',E,\omega)(s),E)ds  dE\Big)dE' d\omega dx
\nonumber\\
&
+\int_G\int_S\int_{I'}\psi(x,\omega,E')
\Big(
{\rm p.f.}\int_{E_0}^{E'}{1\over{E'-E}}
{\p {\hat{\sigma}_{2}}{E'}}(x,E',E)
\int_{0}^{2\pi}v(x,\gamma(E',E,\omega)(s),E)ds  dE\Big)dE' d\omega dx
\nonumber\\
&
+\int_G\int_S\int_{I'}\psi(x,\omega,E')\nonumber\\
&
\cdot
\Big(
{\rm p.f.}\int_{E_0}^{E'}{1\over{E'-E}}
\hat{\sigma}_{2}(x,E',E)
\int_{0}^{2\pi}\la (\nabla_S v)(x,\gamma(E',E,\omega)(s),E),
{\p {\gamma}{E'}}(E',E,\omega)\ra ds  dE\Big)dE' d\omega dx
\nonumber\\
& 
-2\pi\ 
\int_G\int_S\int_{I'}\psi(x,\omega,E')
{\p {\hat{\sigma}_{2}}{E}}(x,E',E')v(x,\omega,E')dE' d\omega dx\nonumber\\
&
-
\int_G\int_S\int_{I'}\psi(x,\omega,E')
\hat{\sigma}_{2}(x,E',E'){\partial\over{\partial E}}\Big(\int_{0}^{2\pi}v(x,\gamma(E',E,\omega)(s),E)ds\Big)_{\Big|E=E'}dE' d\omega dx
\nonumber\\
&
-
2\pi\ 
\int_G\int_I\int_S\psi(x,\omega,E)
{\p {\hat{\sigma}_{2}}{E'}}(x,E,E)
v(x,\omega,E) dE d\omega dx\nonumber\\
&
+2\pi\ \int_G\int_I\int_S\psi(x,\omega,E){\p {(\hat{\sigma}_{2}(x,E,E)v)}{E}}(x,\omega,E)\Big) dE d\omega dx\nonumber\\
&
-\int_G\int_I\int_S\psi(x,\omega,E)
\hat{\sigma}_{2}(x,E,E)\nonumber\\
&
\cdot
\Big(\int_{0}^{2\pi}\la (\nabla_S v)(x,\gamma(E',E,\omega)(s),E),{\p \gamma{E'}}(E',E,\omega)(s)\ra ds\Big)_{\Big|E'=E}
\Big)d\omega  dE dx
.
\eea

In virtue of the proof Theorem \ref{ad-le-6})
\[
&
{\partial\over{\partial E}}\Big(\int_{0}^{2\pi}v(x,\gamma(E',E,\omega)(s),E)ds\Big)_{\Big|E=E'}
=
2\pi\ {\p v{E'}}(x,\omega,E')\nonumber\\
&
+
\Big(\int_{0}^{2\pi}\la (\nabla_S v)(x,\gamma(E',E,\omega)(s),E),{\p \gamma{E}}(E',E,\omega)(s)\ra ds\Big)_{\Big|E=E'}.
\]
Furthermore,
\[
&
{\p {(\hat{\sigma}_{2}(x,E,E)v)}{E}}
={\p {(\hat{\sigma}_{2}(x,E,E))}{E}}v+\hat{\sigma}_{2}(x,E,E){\p v{E}}\\
&
=\big({\p {\hat{\sigma}_{2}}{E'}}(x,E,E)+{\p {\hat{\sigma}_{2}}{E}}(x,E,E)\big)
v+\hat{\sigma}_{2}(x,E,E){\p v{E}},
\]
\[
{\p \gamma{E}}(E',E,\omega)(s)=R(\omega)\big(-{{\mu\ \partial_E\mu}\over{\sqrt{1-\mu^2}}}\cos(s),-{{\mu\ \partial_E\mu}\over{\sqrt{1-\mu^2}}}\sin(s), \partial_E\mu\big),
\]
\[
{\p \gamma{E'}}(E',E,\omega)(s)=R(\omega)\big(-{{\mu\ \partial_{E'}\mu}\over{\sqrt{1-\mu^2}}}\cos(s),-{{\mu\ \partial_{E'}\mu}\over{\sqrt{1-\mu^2}}}\sin(s), \partial_{E'}\mu)\big),
\]
\[
(\partial_E\mu)(E',E)={1\over{2\mu}}\Big({1\over{E'(E+2)^2}}(2E'+4)\Big),
\ (\partial_{E'}\mu)(E',E)={1\over{2\mu}}\Big({1\over{E'^2(E+2)}}(-2E)\Big).
\]
Denote 
\bea\label{theta}
&
\theta(E',E,\omega)(s):={\p \gamma{E'}}(E',E,\omega)(s)+{\p \gamma{E}}(E',E,\omega)(s)
\nonumber\\
&
=
 (\partial_{E'}\mu+\partial_{E}\mu)R(\omega)
\big({{-\mu}\over{\sqrt{1-\mu^2}}}\cos(s),{{-\mu}\over{\sqrt{1-\mu^2}}}\sin(s),1\big),
\eea

Hence removing the cancelling terms and reorganizing  we conclude from (\ref{var-for-15-a-1}) that
\bea\label{var-for-15-a-2}
&
B_2(\psi,v)\nonumber\\
&
=
-
\int_G\int_S\int_{I'}\psi(x,\omega,E')
\Big(
{\rm p.f.}\int_{E_0}^{E'}{1\over{(E'-E)^2}}
\hat{\sigma}_{2}(x,E',E) \nonumber\\
&
\cdot
\int_{0}^{2\pi}v(x,\gamma(E',E,\omega)(s),E)ds  dE\Big)dE' d\omega dx\nonumber\\
&
-\int_G\int_I\int_S
\hat{\sigma}_{2}(x,E,E)\psi(x,\omega,E)\nonumber\\
&
\cdot
\Big(\int_{0}^{2\pi}\la (\nabla_S v)(x,\gamma(E',E,\omega)(s),E),\theta(E',E'\omega)(s)\ra ds\Big)_{\Big|E'=E}
\Big) dE d\omega dx
.
\eea

C.
Note that $\theta(E',E,\omega)(s)={\p {\gamma}{E}}(E',E,\omega)(s)+{\p {\gamma}{E'}}(E',E,\omega)(s))$ has a singularity of the form $(E'-E)^{-1/2}$ as well. 
Applying the techniques used in the proof of Theorem \ref{ad-le-6} (cf. (\ref{var-for-14})) we have
\bea 
&
\Big(\int_{0}^{2\pi}\la (\nabla_S v)(x,\gamma(E',E,\omega)(s),E),\theta(E',E,\omega)(s)\ra ds\Big)_{\Big|E'=E}
\nonumber\\
&
=
2\pi \big((\partial_{E}\mu)(E,E)+(\partial_{E'}\mu)(E,E)\big)(\omega\cdot\nabla_Sv)(x,\omega,E)\nonumber
+
\sum_{|\alpha|\leq 2}b_{\alpha}(E,\omega)(\partial_{\omega}^\alpha v)(x,\omega,E)
\eea
where
\bea\label{cc}
&
\sum_{|\alpha|\leq 2}b_\alpha (\omega,E)
(\partial_{\omega}^\alpha v)(x,\omega,E)\nonumber\\
&
=
\lim_{E'\to E}\int_0^{2\pi}\sum_{j=1}^2\partial_j \big(\la (\nabla_S\psi\circ  H_\omega)(x,.,E'),\theta(E',E,\omega,s)\ra\big)(0) \xi_j((E',E,\omega,s) ds
\eea
and
$\theta(E',E,\omega,s):={\p {\gamma}{E'}}(E',E,\omega)(s)+{\p {\gamma}{E}}(E',E,\omega)(s)$. 
Now, similarly as in Lemma \ref{aij}
\[
&
\lim_{E'\to E}\int_0^{2\pi}\sum_{j=1}^2\partial_j \big(\la (\nabla_S\psi\circ  H_\omega)(x,.,E'),\theta(E',E,\omega,s)\ra\big)(0) \xi_j((E',E,\omega,s) ds 
\nonumber\\
&
=
\big((\partial_{E'}\mu)(E,E)+(\partial_{E}\mu)(E,E)\big) 
\cdot
\int_0^{2\pi}
\sum_{j=1}^2\partial_j \big(\la (\nabla_S\psi\circ  H_\omega)(x,.,E),R(\omega)\big(\cos(s),\sin(s),0\big)\ra\big)(0)\\
&
\cdot
\la R(\omega)\big(\cos(s),\sin(s),0\big)
,\ol\Omega_j\ra  ds.
\]
Recall that by (\ref{nab0})
\[
(\omega\cdot\nabla_Sv)(x,\omega,E)=0.
\]
Hence we finally obtain 
\bea\label{var-for-15-a-3}
&
B_2(\psi,v)\nonumber\\
&
=
-
\int_G\int_S\int_{I'}\psi(x,\omega,E')
\Big(
{\rm p.f.}\int_{E_0}^{E'}{1\over{(E'-E)^2}}
\hat{\sigma}_{2}(x,E',E) \nonumber\\
&
\cdot
\int_{0}^{2\pi}v(x,\gamma(E',E,\omega)(s),E)ds  dE\Big)dE' d\omega dx
\nonumber\\
&
-\int_G\int_I\int_S\hat{\sigma}_{2}(x,E,E)\psi(x,\omega,E)
\sum_{|\alpha|\leq 2}b_{\alpha}(E,\omega)(\partial_{\omega}^\alpha v)(x,\omega,E)
dE d\omega 
\eea
where
\bea\label{bij-a}
&
\sum_{|\alpha|\leq 2}b_\alpha (\omega,E)
(\partial_{\omega}^\alpha\psi)(x,\omega,E)
:=
\big((\partial_{E'}\mu)(E,E)+(\partial_{E}\mu)(E,E)\big)\\
&
\cdot
\sum_{j=1}^2\int_0^{2\pi}
\partial_j \big(\la (\nabla_S\psi\circ  H_\omega)(x,.,E),R(\omega)\big(\cos(s),\sin(s),0\big)\ra\big)(0)
\nonumber\\
&
\cdot
\la R(\omega)\big(\cos(s),\sin(s),0\big)
,\ol\Omega_j\ra  ds.
\eea
Since
\[
&
(\mu\ \partial_{E'}\mu)(E',E)={E\over{E'(E+2)}}-{{E(E'+2)(E+2)}\over{E'^2((E+2)^2}},\\
&
(\mu\ \partial_{E}\mu)(E',E)={{E'2}\over{E'(E+2)}}-{{E(E'+2)E'}\over{E'^2((E+2)^2}}
\]
we find that
\[
( \partial_{E'}\mu)(E',E)+(\partial_{E}\mu)(E',E)
={1\over {E+2}}-{1\over{E}}+{1\over E}-{1\over{E+2}}=0.
\]
That is why $\sum_{|\alpha|\leq 2}b_\alpha (\omega,E)
(\partial_{\omega}^\alpha\psi)(x,\omega,E)=0$ which completes the proof. 
\end{proof}

We show the above used lemma.

\begin{lemma}\label{le-for-w}
A.
Assume that $\hat\sigma_2\in C(\ol G\times I',C^\alpha(I))$ and that 
$v\in C(\ol G\times I',C^{\alpha}(S))\cap C(\ol G\times S, C^\alpha(I')),\ \alpha >0$ for which $v(.,.E_0)=0$. Then
\bea
&
\lim_{E'\to E_0}\Big(\int_{E_0}^{E'}{1\over{E'-E}}\hat{\sigma}_{2}(x,E',E)
\int_{0}^{2\pi}v(x,\gamma(E',E,\omega)(s),E)ds dE\Big)
\nonumber\\
&
=
2\pi\ \lim_{E'\to E_0}\ln(E'-E_0)\hat{\sigma}_{2}(x,E',E')
v(x,\omega,E')=0.
\eea

B. 
Assume that $\hat\sigma_2\in C(\ol G\times I,C^\alpha(I'))$  that $\psi\in C(\ol G\times I,C^{\alpha}(S))\cap C(\ol G\times S, C^\alpha(I)),\ \alpha >0$ for which $\psi(.,.E_m)=0$ 
Then
\bea
&
\lim_{E\to E_m}\Big(\int_{E}^{E_m}{1\over{E'-E}}\hat{\sigma}_{2}(x,E',E)
\int_{0}^{2\pi}\psi(x,\gamma(E',E,\omega)(s),E)ds dE'\Big)
\nonumber\\
&
=
2\pi\ \lim_{E\to E_m}\ln(E_m-E)\hat{\sigma}_{2}(x,E,E)
\psi(x,\omega,E)=0.
\eea
\end{lemma}

\begin{proof}
A.
Let 
\[
h(x,\omega,E',E):=\hat{\sigma}_{2}(x,E',E)
\int_{0}^{2\pi}v(x,\gamma(E',E,\omega)(s),E)ds .
\]
Then
\[
& 
{\rm p.f.}\int_{E_0}^{E'}{{h(x,\omega,E',E)}\over{E'-E}}dE\\
&
=
{\rm p.f.}\int_{E_0}^{E'}{{h(x,\omega,E',E')}\over{E'-E}}dE
+
\int_{E_0}^{E'}{{h(x,\omega,E',E)-h(x,\omega,E',E')}\over{E'-E}}dE.
\]
In virtue of the assumptions on $\hat\sigma_2$ and $v$ there exits $\delta>0$ such that (we omit details but recall that $\n{\gamma(E',E,\omega)(s)-\omega}\leq C|E'-E|^{1\over 2}$)
\[
\Big|{{h(x,\omega,E',E)-h(x,\omega,E',E')}\over{E'-E}}\Big|\leq C(E'-E)^{\delta-1}
\]
and so 
\[
\lim_{E'\to E_0}\int_{E_0}^{E'}{{h(x,\omega,E',E)-h(x,\omega,E',E')}\over{E'-E}}dE=0.
\]
Moreover, (recall that $\gamma(E',E',\omega)(s)=\omega$)
\[
&
{\rm p.f.}\int_{E_0}^{E'}{{h(x,\omega,E',E')}\over{E'-E}}dE=
\ln(E'-E_0)h(x,\omega,E',E')=2\pi\ \ln(E'-E_0)\hat\sigma_2(x,E',E')v(x,\omega,E')
\]
Hence we get
\bea
&
\lim_{E'\to E_0}
\int_{E_0}^{E'}{1\over{E'-E}}\hat{\sigma}_{2}(x,E',E)
\int_{0}^{2\pi}v(x,\gamma(E',E,\omega)(s),E)ds dE\nonumber\\
&
=2\pi\ \lim_{E'\to E_0}
\ln(E'-E_0)\hat\sigma_2(x,E',E')v(x,\omega,E')
+
\lim_{E'\to E_0}
\int_{E_0}^{E'}{{h(x,\omega,E',E)-h(x,\omega,E',E')}\over{E'-E}}dE
\nonumber\\
&
=
2\pi\ 
\lim_{E'\to E_0}
\ln(E'-E_0)\hat\sigma_2(x,E',E')v(x,\omega,E')=0
\eea
since by (\ref{limv}) $\lim_{E'\to E_0}\ln(E'-E_0)v(x,\omega,E')=0$.
This completes the proof.

B. The proof runs analogously as in Part A.
\end{proof}

Note that
\be
\int_{\Gamma_-}(\omega\cdot\nu)_-{g} v\ d\sigma d\omega dE=
\la { g},\gamma_-(v)\ra_{T^2(\Gamma_-)}
\ee
and
\be
\int_{\Gamma_+}(\omega\cdot\nu)_+\psi v\ d\sigma d\omega dE=
\la \gamma_+(\psi),\gamma_+(v)\ra_{T^2(\Gamma_+)}.
\ee

\subsection{Lowering the degree of singularities in the bilinear form $B_2(.,.)$}\label{lowering}

The above bilinear form $B_2(.,.)$ contains singularities of the form ${1\over{(E'-E)^2}}dE dE'$  which may be problematic e.g. in numerical treatments. The degree of this singularity can be lowered as follows.
Consider the term in question
\[
I
&
:=-
\int_G\int_S\int_{I'}\psi(x,\omega,E')
\Big(
{\rm p.f.}\int_{E_0}^{E'}{1\over{(E'-E)^2}}
\hat{\sigma}_{2}(x,E',E) \\
&
\cdot
\int_{0}^{2\pi}v(x,\gamma(E',E,\omega)(s),E)ds  dE\Big)dE' d\omega dx.
\]
Let $f(x,E',E)$ be as above in the proof of Theorem \ref{var-th-a}, Part B that is,
\[
f(x,E',E):=\hat{\sigma}_{2}(x,E',E)
\int_S\int_{0}^{2\pi}\psi(x,\omega,E')v(x,\gamma(E',E,\omega)(s),E)ds d\omega.
\]
Then
we have by Lemmas \ref{ad-le-5} and \ref{ad-le-3-aa}
\bea\label{lo-1}
&
I=-
\int_{G}\Big[\int_{I'}\Big({\rm p.f.}\int_{E_0} ^{E'}{1\over{(E'-E)^2}}\Big(
\hat{\sigma}_{2}(x,E',E) \nonumber\\
&
\cdot
\int_S\int_{0}^{2\pi}\psi(x,\omega,E')v(x,\gamma(E',E,\omega)(s),E)ds d\omega\Big)dE\Big)dE'\Big] dx\nonumber\\
&
=-
\int_{G}\Big[\int_{I'}\Big({\rm p.f.}\int_{E_0} ^{E'}{1\over{(E'-E)^2}}
f(x,E',E)dE\Big)dE'\Big] dx\nonumber\\
&
=-
\int_{G}\Big[\int_{I'}\Big(
-{\rm p.f.}\int_{E_0}^{E'}{1\over{E'-E}}{\p f{E}}(x,E',E)dE \nonumber\\
&
-{\p f{E}}(x,E',E')_-
-{1\over{E'-E_0}}f(x,E',E_0)\Big)dE'\Big] dx\nonumber\\
&
=
\int_{G}\int_{I'}\Big(
{\rm p.f.}\int_{E_0}^{E'}{1\over{E'-E}}{\p f{E}}(x,E',E)dE\Big)dE' dx
+
\int_G\int_{I'}
{\p f{E}}(x,E',E')_- dE' dx\nonumber\\
&
+\int_G\int_{I'} {1\over{E'-E_0}}f(x,E',E_0)dE' dx=:I_{1}+I_{2}+I_{3}
\eea
where we assume that the last integral $I_3$ exists. 

Moreover, we have
\bea\label{lo-2}
&
{\p f{E}}(x,E',E)=
{\p {\hat{\sigma}_{2}}{E}}(x,E',E)
\int_S\int_{0}^{2\pi}\psi(x,\omega,E')v(x,\gamma(E',E,\omega)(s),E)ds d\omega
\nonumber\\
&
+
\hat{\sigma}_{2}(x,E',E)
\int_S\int_{0}^{2\pi}\psi(x,\omega,E'){\p v{E}}(x,\gamma(E',E,\omega)(s),E)ds d\omega
\nonumber\\
&
+
\hat{\sigma}_{2}(x,E',E)
\int_S\int_{0}^{2\pi}\psi(x,\omega,E')\la (\nabla_S v)(x,\gamma(E',E,\omega)(s),E),{\p \gamma{E}}(E',E,\omega)(s)\ra ds d\omega
\eea
and so by (\ref{var-for-12}) 
\[
&
I_{1}
=
\int_{G}\int_{I'}\Big(
{\rm p.f.}\int_{E_0}^{E'}{1\over{E'-E}}
{\p {\hat{\sigma}_{2}}{E}}(x,E',E) \nonumber \\
&
\cdot
\int_S\int_{0}^{2\pi}\psi(x,\omega,E')v(x,\gamma(E',E,\omega)(s),E)ds 
dE\Big) dE'd\omega dx
\nonumber\\
&
+\int_{G}\int_{I'}\Big({\rm p.f.}\int_{E_0}^{E'}{1\over{E'-E}}
\hat{\sigma}_{2}(x,E',E) \nonumber \\
&
\cdot
\int_S\int_{0}^{2\pi}\psi(x,\omega,E'){\p v{E}}(x,\gamma(E',E,\omega)(s),E)ds  dE\Big) dE'd\omega dx
\nonumber\\
&
+\int_{G}\int_{I'}\Big({\rm p.f.}\int_{E_0}^{E'}{1\over{E'-E}}
\hat{\sigma}_{2}(x,E',E)
\int_S\int_{0}^{2\pi}\psi(x,\omega,E')\nonumber\\
&
\cdot
\la (\nabla_S v)(x,\gamma(E',E,\omega)(s),E),{\p \gamma{E}}(E',E,\omega)(s)\ra ds d\omega
dE\Big)dE' dx,
\]
i.e.
\bea\label{lo-3}
&
I_1
=
\int_{G}\int_{I'}\int_S\psi(x,\omega,E')\Big(
{\rm p.f.}\int_{E_0}^{E'}{1\over{E'-E}}
{\p {\hat{\sigma}_{2}}{E}}(x,E',E) \nonumber \\
&
\cdot
\int_{0}^{2\pi}v(x,\gamma(E',E,\omega)(s),E)ds d\omega
dE\Big) dE' dx\nonumber\\
&
+\int_{G}\int_{I'}\int_S\psi(x,\omega,E')\Big({\rm p.f.}\int_{E_0}^{E'}{1\over{E'-E}}
\hat{\sigma}_{2}(x,E',E)
\nonumber\\
&
\cdot \int_{0}^{2\pi}
\la (\nabla_S v)(x,\gamma(E',E,\omega)(s),E),{\p \gamma{E}}(E',E,\omega)(s)\ra ds d\omega
dE\Big)dE' dx\nonumber\\
&
+\int_{G}\int_{I'}\int_S\psi(x,\omega,E')\Big({\rm p.f.}\int_{E_0}^{E'}{1\over{E'-E}}
\hat{\sigma}_{2}(x,E',E) \nonumber \\
&
\cdot
\int_{0}^{2\pi}{\p v{E}}(x,\gamma(E',E,\omega)(s),E)ds  dE\Big) dE' d\omega dx
.
\eea

Since $\gamma(E',E',\omega)(s)=\omega$ we have by (\ref{lo-2}) 
\bea\label{lo-4}
&
I_{2}=
\int_G\int_{I'}
{\p f{E}}(x,E',E')_- dE dx\nonumber \\
&
=2\pi\ \int_G\int_I\int_S
\psi(x,\omega,E'){\p {\hat{\sigma}_{2}}{E}}(x,E',E')
v(x,\omega,E') d\omega dE' dx
\nonumber\\
&
+
2\pi\ \int_G\int_{I'}\int_S\psi(x,\omega,E')\hat{\sigma}_{2}(x,E',E') 
{\p v{E}}(x,\omega,E') d\omega dE' dx
\nonumber\\
&
+\int_G\int_{I'}
\int_S\psi(x,\omega,E')\hat{\sigma}_{2}(x,E',E')
\nonumber\\
&
\cdot
\Big(\int_{0}^{2\pi}\la (\nabla_S v)(x,\gamma(E',E,\omega)(s),E),{\p \gamma{E}}(E',E,\omega)(s)\ra ds\Big)_{\Big|E=E'} d\omega dE' dx
\eea

In addition,
\be\label{lo-5}
I_{3}=
\int_G\int_{I'}\int_S {1\over{E'-E_0}}\psi(x,\omega,E')\Big(
\hat{\sigma}_{2}(x,E',E_0)
\int_{0}^{2\pi}v(x,\gamma(E',E_0,\omega)(s),E_0)ds
\Big)dE' dx d\omega=0.
\ee

Substituting $I=I_1+I_2+I_3$ to (\ref{b2-a}) we obtain by (\ref{lo-3}), (\ref{lo-4}), (\ref{lo-5})
\bea\label{lo-6}
&
B_2(\psi,v)\nonumber\\
&
=
\int_{G}\int_{I'}\int_S\psi(x,\omega,E')\Big(
{\rm p.f.}\int_{E_0}^{E'}{1\over{E'-E}}
{\p {\hat{\sigma}_{2}}{E}}(x,E',E)
 \nonumber \\
&
\cdot
\int_{0}^{2\pi}v(x,\gamma(E',E,\omega)(s),E)ds d\omega
dE\Big) dE' dx\nonumber\\
&
+\int_{G}\int_{I'}\int_S\psi(x,\omega,E')\Big({\rm p.f.}\int_{E_0}^{E'}{1\over{E'-E}}
\hat{\sigma}_{2}(x,E',E)
\nonumber\\
&
\cdot \int_{0}^{2\pi}
\la (\nabla_S v)(x,\gamma(E',E,\omega)(s),E),{\p \gamma{E}}(E',E,\omega)(s)\ra ds d\omega
dE\Big)dE' dx\nonumber\\
&
+\int_{G}\int_{I'}\int_S\psi(x,\omega,E')\Big({\rm p.f.}\int_{E_0}^{E'}{1\over{E'-E}}
\hat{\sigma}_{2}(x,E',E) \nonumber \\
&
\cdot
\int_{0}^{2\pi}{\p v{E}}(x,\gamma(E',E,\omega)(s),E)ds  dE\Big) dE' d\omega dx\nonumber\\
&
+
2\pi\ \int_G\int_{I'}\int_S
\psi(x,\omega,E'){\p {\hat{\sigma}_{2}}{E}}(x,E',E')
v(x,\omega,E') d\omega dE' dx
\nonumber\\
&
+
2\pi\ \int_G\int_{I'}\int_S\psi(x,\omega,E')\hat{\sigma}_{2}(x,E',E') 
{\p v{E}}(x,\omega,E') d\omega dE' dx
\nonumber\\
&
+\int_G\int_{I'}
\int_S\psi(x,\omega,E')\hat{\sigma}_{2}(x,E',E')
\nonumber\\
&
\cdot
\Big(\int_{0}^{2\pi}\la (\nabla_S v)(x,\gamma(E',E,\omega)(s),E),{\p \gamma{E}}(E',E,\omega)(s)\ra ds\Big)_{\Big|E=E'} d\omega dE' dx.
\eea

As in (\ref{var-for-14}) one finds that
\bea\label{lo-6-a}
&
\Big(\int_{0}^{2\pi}\la (\nabla_S v)(x,\gamma(E',E,\omega)(s),E),{\p \gamma{E}}(E',E,\omega)(s)\ra ds\Big)_{\Big|E=E'} 
\nonumber\\
&
=
2\pi\ (\partial_E\mu)(E',E')(\omega\cdot\nabla_S v)(x,\omega,E')
+\sum_{|\alpha|\leq 2}c_\alpha(E',\omega)(\partial_\omega^\alpha v)(x,\omega,E')
\eea
where
\bea\label{mue'}
&
\sum_{|\alpha|\leq 2}c_\alpha (\omega,E')
(\partial_{\omega}^\alpha v)(x,\omega,E')\nonumber\\
&
:=
-(\partial_{E}\mu)(E',E')
\cdot
\int_0^{2\pi}
\sum_{j=1}^2\partial_j \big(\la (\nabla_Sv\circ  H_\omega)(x,.,E'),R(\omega)\big(\cos(s),\sin(s),0\big)\ra\big)(0)\nonumber\\
&
\cdot
\la R(\omega)\big(\cos(s),\sin(s),0\big)
,\ol\Omega_j\ra  ds\nonumber\\
&
=-\pi\ (\partial_{E}\mu)(E',E')(\Delta_Sv)(x,\omega,E')
\eea

By (\ref{nab0})
\[
(\omega\cdot\nabla_Sv)(x,\omega,E')=0.
\]  
In addition, it is well-known that for the Laplace-Beltrami operator it holds
\[
\la \psi,\Delta_Sv\ra_{L^2(G\times S\times I)}=
-\int_G\int_S\int_{I'}\la \nabla_S\psi(x,\omega,E),\nabla_Sv(x,\omega,E)\ra dx d\omega dE'
\]
and so 
\bea\label{add-a} 
&
\int_G\int_{I'}
\int_S\psi(x,\omega,E')\hat{\sigma}_{2}(x,E',E')
\nonumber\\
&
\cdot
\Big(\int_{0}^{2\pi}\la (\nabla_S v)(x,\gamma(E',E,\omega)(s),E),{\p \gamma{E}}(E',E,\omega)(s)\ra ds\Big)_{\Big|E=E'} d\omega dE' dx
\nonumber\\
&
=
\pi\ \int_G\int_{I'}
\int_S\psi(x,\omega,E')\hat{\sigma}_{2}(x,E',E')(\partial_{E}\mu)(E',E')
\la \nabla_S\psi(x,\omega,E),\nabla_Sv(x,\omega,E)\ra dx d\omega dE.
\eea
Hence
we conclude from (\ref{lo-6-a}) and (\ref{lo-6}) 
(by removing the cancelling terms and rearranging the terms) finally 

\begin{theorem}\label{final-bi-form}
Suppose that the assumptions of Theorem \ref{var-th-a} are valid.
Then for $v\in \mc D_{(0)}$
\bea\label{lo-6-f}
&
B_2(\psi,v)=\nonumber\\
&
\int_{G}\int_{I'}\int_S\psi(x,\omega,E')\Big({\rm p.f.}\int_{E_0}^{E'}{1\over{E'-E}}
\hat{\sigma}_{2}(x,E',E) \nonumber\\
&
\cdot
\int_{0}^{2\pi}{\p v{E}}(x,\gamma(E',E,\omega)(s),E)ds  dE\Big) dE' d\omega dx\nonumber\\
&
+
2\pi\ \int_G\int_{I'}\int_S\psi(x,\omega,E')\hat{\sigma}_{2}(x,E',E') 
{\p v{E}}(x,\omega,E') d\omega dE' dx
\nonumber\\
&
+
\int_G\int_S\int_{I'}\psi(x,\omega,E')
\Big(
{\rm p.f.}\int_{E_0}^{E'}{1\over{E'-E}}
\hat{\sigma}_{2}(x,E',E)\nonumber\\
&
\cdot
\int_{0}^{2\pi}\la (\nabla_S v)(x,\gamma(E',E,\omega)(s),E),
{\p \gamma{E}}(E',E,\omega)(s)\ra ds  dE\Big)dE' d\omega dx
\nonumber\\
&
+
\pi\ \int_G\int_{I'}
\int_S\psi(x,\omega,E')\hat{\sigma}_{2}(x,E',E')(\partial_{E}\mu)(E',E')
\la \nabla_S\psi(x,\omega,E),\nabla_Sv(x,\omega,E)\ra dx d\omega dE\nonumber\\
&
+\int_G\int_S\int_{I'}\psi(x,\omega,E')
\Big(
{\rm p.f.}\int_{E_0}^{E'}{1\over{E'-E}}{\p {\hat{\sigma}_{2}}{E}}(x,E',E) \nonumber\\
&
\cdot
\int_{0}^{2\pi}v(x,\gamma(E',E,\omega)(s),E)ds  dE\Big)dE' d\omega dx
\nonumber\\
&
+2\pi\ \int_G\int_I\int_S\psi(x,\omega,E){\p {\hat{\sigma}_{2}}{E}}(x,E,E)v(x,\omega,E)dE d\omega dx .
\eea
\end{theorem}

The variational problem (\ref{var-form}) becomes
\be\label{var-form-aa} 
B(\psi,v)=F(v),\ v\in \mc D_{(0)}
\ee
The bilinear form $B(.,.)$ is
\be\label{fb-aa}
B(\psi,v)=B_2(\psi,v)+B_1(\psi,v)+B_0(\psi,v)
\ee
in which $B_2(.,.)$ is now given by (\ref{final-bi-form}).

\begin{remark} \label{testf}
The above considerations suggest that the relevant space of test functions is
$\mc D_{(0)}$. Closer treatments reveal that we can additionally demand that the adjoint inflow boundary condition
\be\label{ad-bc}
\gamma_+(v)=0
\ee
is valid for test functions $v$
(cf. \cite{tervo18-up}, the proof of Theorem  6.7).  
\end{remark}

\subsection{Remarks on boundedness of the bilinear form}\label{bilin-bound}

We give  some preliminary 
notes  concerning for the boundedness of the bilinear forms $B_j(.,.), j=0,1,2$.

\emph{Boundedness of $B_0$}.
Define  an inner product 
\be \label{inH0}
\la \psi,v\ra_{H_0}:=\la \psi,v\ra_{L^2(G\times S\times I)}+
\la\gamma(\psi),\gamma(v)\ra_{T^2(\Gamma)}
\ee
in $C(\ol G\times S\times I)$ and an inner product
\bea\label{inhatH0}
\la \psi,v\ra_{\hat H_0}
&:=\la \psi,v\ra_{\tilde W^2(G\times S\times I)}\nonumber\\
&
=\la \psi,v\ra_{L^2(G\times S\times I)}
+
\la\omega\cdot\nabla_x \psi,\omega\cdot\nabla_x v\ra_{L^2(G\times S\times I)}+
\la\gamma(\psi),\gamma(v)\ra_{T^2(\Gamma)}
\eea
in $C^1(\ol G,C(S\times I))$.

The bilinear form $B_0:C^1(\ol G,C( S\times I))\times C^1(\ol G,C(S\times I))\to\R$ obeys the following  boundedness  condition (see \cite{tervo18-up}, Theorem 5.4):

Suppose that the assumptions  (\ref{sigma}), (\ref{ass5-a}), (\ref{ass7}), (\ref{ass-8})  are valid.
Then there exists a constant $C_0>0$  such that 
\be\label{csda29}
|B_0(\psi,v)|\leq C_0\n{\psi}_{H_0}\n{v}_{\hat H_0}\quad \forall\ \psi,\ v\in C^1(\ol G,C( S\times I)).
\ee

\emph{On boundedness of $B_1$}.
The Banach space
$L^2(G\times I,W^{\infty,\alpha}(S))$ 
and its norm is standardly defined.
The Hilbert space  $L^2(G\times S,H^{s}(I)), \ s\geq 0$ is standardly defined as well and  it is  equipped with the usual inner product.

Let $0<\kappa<{1\over 2}$, $1-2\kappa<\alpha<1$.
Define in $C(\ol G\times I,C^\alpha(S)))\cap C(\ol G\times S,C^{{1\over 2}-\kappa}(I))$ 
the weighted norm
\[
\n{\psi}_{L^2(G\times S\times I,w)}
:= 
\big(\n{\sqrt{\ln(E-E_0)}\psi}_{L^2(G\times S\times I)}^2+\n{\sqrt{\ln(E_m-E)}\psi}_{L^2(G\times S\times I)}^2
\big)^{1/2}
\]
and the norm
\[
\n{v}_{H_{\kappa,\alpha}}:=\big(\n{v}_{L^2(G\times S\times I)}^2+\n{v}_{L^2(G\times S\times I,w)}^2+
\n{v}_{L^2(G\times S,H^{{1\over 2}-\kappa}(I))}^2+
\n{v}_{L^2(G\times I, W^{\infty,\alpha}(S))}^2\big)^{1/2}
\]

The following estimate can be shown after some (tedious) computations:

Suppose that $\hat\sigma_1\in C^1(I',L^\infty( G\times I))$.
Then for  $\psi\in C(\ol G\times S\times I),\ v\in C(\ol G\times I,C^\alpha(S)))\cap C(\ol G\times S,C^{{1\over 2}-\kappa}(I))$
\be\label{bound-B1}
|B_1(\psi,v)|\leq C\n{\psi}_{L^2(G\times S\times I,w)} \n{v}_{H_{\kappa,\alpha}}.
\ee

\emph{On boundedness of $B_2$}.
The boundedness criteria for $B_2(.,.)$ can retrieved on the basis of the previous paragraph. 
For example, the first term  
\[
&
B_{2,1}(\psi,v):=\\
&
\int_{G}\int_{I'}\int_S\psi(x,\omega,E')\Big({\rm p.f.}\int_{E_0}^{E'}{1\over{E'-E}}
\hat{\sigma}_{2}(x,E',E)
\int_{0}^{2\pi}{\p v{E}}(x,\gamma(E',E,\omega)(s),E)ds  dE\Big) dE' d\omega dx
\]
appearing in (\ref{lo-6-f}
obeys under relevant assumptions
\be\label{bound-B21}
|B_{2,1}(\psi,v)|\leq C\n{\psi}_{L^2(G\times S\times I,w)} \n{{\p v{E}}}_{H_{\kappa,\alpha}}.
\ee
The only term in $B_2(.,.)$ which requires essentially further study is the third one,
\[
&
B_{2,3}(\psi,v):=
-\int_G\int_S\int_{I'}\psi(x,\omega,E')
\Big(
{\rm p.f.}\int_{E_0}^{E'}{1\over{E'-E}}
\hat{\sigma}_{2}(x,E',E)\nonumber\\
&
\cdot
\int_{0}^{2\pi}\la (\nabla_S v)(x,\gamma(E',E,\omega)(s),E),
{\p \gamma{E}}(E',E,\omega)(s)\ra ds  dE\Big)dE' d\omega dx.
\] 
We neglect here any formulations for $B_2(.,.)$.
We, however, conjecture that the total bilinear form obeys
\be\label{ex-bound}
|B(\psi,v)|\leq C\n{\psi}_{\mc H}\n{v}_{\widehat{\mc H}}
\ee
for $v\in\mc D_{(0)}$,
$
\psi\in \hat{\mc D}_{(0)}:=\{\psi\in C^1(\ol G,C^2(I,C^3(S)))|\ \psi(x,\omega,E_m)=0\}.
$
where 
\[
\n{\psi}_{\mc H}^2:=\n{\psi}_{H_0}^2+\n{\psi}_{L^2(G\times S\times I,w)}+\n{\psi}_{L^2(G\times I,H^1(S)}^2
+\n{\psi(.,.,E_0)}_{L^2(G\times S)}^2+\n{\psi(.,.,E_m)}_{L^2(G\times S)}^2
\]
and
\[
&
\n{v}_{\widehat{\mc H}}^2
=\n{v}_{\hat H_0}^2+\n{v}_{H_{\kappa,\alpha}}^2+
\n{v}_{L^2(G\times I,H^1(S)}^2+\sum_{|\alpha|\leq 2}\n{\partial_\omega^\alpha v}_{H_{\kappa,\alpha}}^2\\
&
+\n{v(.,.,E_0)}_{L^2(G\times S)}^2+\n{v(.,.,E_m)}_{L^2(G\times S)}^2
+\n{{\p {v}E}}_{H_{\kappa,\alpha}}^2.
\]

For example, in the case when one applies 
existence results founded on \emph{Lions-Lax-Milgram type Theorems} (e.g. \cite{treves}),  one needs Hilbert spaces structures.
To obtain Hilbert space setting the space $ W^{\infty,\alpha}(S)$ must be replaced by a Hilbert space space e.g. along the Sobolev imbedding
$
W^{\infty,\alpha}(S)\subset W^{2,s}(S)=H^s(S)
$
where $s>{2\over 2}+\alpha=1+\alpha$.
Maybe even more sophisticated function spaces and estimates are needed.
Additionally to the boundedness criteria one needs relevant coercivity (accretivity) criteria for $B(.,.)$ and a careful treatise of relevant   inflow trace and (generalized) Green  theorems.

Alternatively, to show the existence of solutions it might be useful to apply the truncating process. Using the notations of the below section \ref{tr-app} the proofs might be based on the following items.

\begin{itemize}
\item
Solve the truncated problem $T_\kappa\psi_\kappa=f,\ {\psi_\kappa}_{|\Gamma_-}=g,\ \psi_\kappa(.,.,E_m)=0$. Because of the factor $\ln(\kappa E-E)$ in the front of some coefficients of $T_\kappa$, the above figured weighted Sobolev spaces
must be utilized as solution spaces.

\item 
Show that $\n{\psi_\kappa}_{L^2(G\times S\times I)}\leq C$. Then there exists a subsequence $\{\psi_{\kappa_j}\}$ and $\psi\in L^2(G\times S\times I)$  such that $\psi_{\kappa_j}\to \psi$ weakly in $L^2(G\times S\times I)$ (Banach-Alaoglu Theorem for Hilbert spaces).

\item
Then for $v\in C_0^\infty(G\times S\times I^\circ)$
\[
&
\la\psi,T^*v\ra_{L^2(G\times S\times I)}
=\lim_{j\to\infty}\la \psi_{\kappa_j},T^*v\ra_{L^2(G\times S\times I)}\\
&
=
\lim_{j\to\infty}\la \psi_{\kappa_j},T_{\kappa_j}^*v+T^*v-T_{\kappa_j}^*v\ra_{L^2(G\times S\times I)}
\\
&
=
\lim_{j\to\infty}\la \psi_{\kappa_j},T_{\kappa_j}^*v\ra_{L^2(G\times S\times I)}
+
\lim_{j\to\infty}\la \psi_{\kappa_j},T^*v-T_{\kappa_j}^*v\ra_{L^2(G\times S\times I)}
\\
&
=\la f,v\ra_{L^2(G\times S\times I)}
\]
since 
\[
\la \psi_{\kappa_j},T_{\kappa_j}^*v\ra_{L^2(G\times S\times I)}
=\la f,v\ra_{L^2(G\times S\times I)}
\]
and since by Theorem \ref{err-th-2} below
\[
\lim_{j\to\infty}\la \psi_{\kappa_j},T^*v-T_{\kappa_j}^*v\ra_{L^2(G\times S\times I)}=0.
\]
Hence $T\psi=f$ weakly. The initial and inflow boundary conditions require further study.

\end{itemize}

Each of these items requires much further study and so we omit  detailed treatises.
As a summary, we notice that
the existence and uniqueness analysis of the  non-classical transport problem treated in this paper requires a considerable further study.

\section{On approximative transport   problem}\label{tr-app}

\subsection{Truncation }\label{trunc}

The following approximation scheme shows that the exact transport operator (related to M\o ller scattering) is CSDA-Fokker-Planck type operator. 
We shall below apply the strongly hyper-singular form (\ref{exth-1}) of $T$.
Equally well we could use  the expression (\ref{exth-2}) by utilizing the same techniques in cutting and approximation of hyper-singular integrals.

Recall from (\ref{exth-1}) that
\bea\label{exth-1-aa}
&
T\psi=-
{\s H}_2\big((\ol{\s K}_2\psi)(x,\omega,\cdot,E)\big)(E) 
+
{\s H}_1\big((\ol{\s K}_1\psi)(x,\omega,\cdot,E)\big)(E)\nonumber\\
&
+\omega\cdot\nabla_x\psi+\Sigma\psi
-K_r\psi
\eea
where the collision operator to be approximated is
\be\label{exth-1-bb}
K\psi=
{\s H}_2\big((\ol{\s K}_2\psi)(x,\omega,\cdot,E)\big)(E) 
-
{\s H}_1\big((\ol{\s K}_1\psi)(x,\omega,\cdot,E)\big)(E)
+K_r\psi
\ee
that is, 
\bea\label{mol-a}
&
(K\psi)(x,\omega,E)=
{\rm p.f.}\int_E^{E_m}\hat\sigma_{2}(x,E',E){1\over{(E'-E)^2}}
\int_{0}^{2\pi}\psi(x,\gamma(E',E,\omega)(s),E')dsdE'
\nonumber\\
&
-
{\rm p.f.}\int_E^{E_m}\hat\sigma_{1}(x,E',E){1\over{E'-E}}
\int_{0}^{2\pi}\psi(x,\gamma(E',E,\omega)(s),E')dsdE'
+(K_r\psi)(x,\omega,E).
\eea
Denote as in \cite{tervo18-up}, \cite{tervo18}
\be\label{psio-0a}
(K_j\psi)(x,\omega,E):=
{\rm p.f.}\int_E^{E_m}\hat\sigma_{j}(x,E',E){1\over{(E'-E)^j}}
\int_{0}^{2\pi}\psi(x,\gamma(E',E,\omega)(s),E')dsdE',\ j=1,2.
\ee

The restricted collision operator $K_r$  is an ordinary partial Schur integral operator and so it suffices only to truncate the first and second terms  in (\ref{mol-a}) that is, the operators $(K_j\psi)(x,\omega,E)$.
We assume that a "cut-off energy for primary particles" is $E'=\kappa E$ where $\kappa>1$.
Decompose the integration in (\ref{psio-0a}) as follows (\cite{tervo18})
\bea\label{psio-0c}
&
(K_j\psi)(x,\omega,E)=
{\rm p.f.}\int_E^{E_m}\hat\sigma_{j}(x,E',E){1\over{(E'-E)^j}}
\int_{0}^{2\pi}\psi(x,\gamma(E',E,\omega)(s),E')dsdE'
\nonumber\\
={}&
{\rm p.f.}\int_E^{\kappa E}\hat\sigma_{j}(x,E',E){1\over{(E'-E)^j}}
\int_{0}^{2\pi}\psi(x,\gamma(E',E,\omega)(s),E')dsdE'\nonumber\\
&
+
\int_{\kappa E}^{E_m}\hat\sigma_{j}(x,E',E){1\over{(E'-E)^j}}
\int_{0}^{2\pi}\psi(x,\gamma(E',E,\omega)(s),E')dsdE'\ {\rm for} \ j=1,2
\eea
where we noticed that the last integral is an ordinary improper integral (and so no p.f.-integration is needed for it).
Let
\bea
(K_{j,1,\kappa}\psi)
&
(x,\omega,E)\nonumber\\
:= {}&
{\rm p.f.}\int_E^{\kappa E}{1\over{(E'-E)^j}}\hat\sigma_{j}(x,E',E)
\int_{0}^{2\pi}\psi(x,\gamma(E',E,\omega)(s),E')dsdE', \label{def-K_{j}-a} \\[2mm] 
(K_{j,0,\kappa}\psi)
&
(x,\omega,E)\nonumber\\
:= {}&
\int_{\kappa E}^{E_m}\hat\sigma_{j}(x,E',E){1\over{(E'-E)^j}}
\int_{0}^{2\pi}\psi(x,\gamma(E',E,\omega)(s),E')dsdE'. \label{def-K_{j}-b}
\eea
Then
\be\label{de}
(K_j\psi)(x,\omega,E)=
(K_{j,1,\kappa}\psi)(x,\omega,E)+(K_{j,0,\kappa}\psi)(x,\omega,E)
\ee
and
\be\label{de-T}
T\psi=-K_{2,1,\kappa}\psi+K_{1,1,\kappa}\psi+\omega\cdot\nabla_x\psi+\Sigma\psi-(K_{2,0,\kappa}\psi-K_{1,0,\kappa}\psi+K_r\psi).
\ee
Since
\bea\label{psio-0d}
&
(K_{j,0,\kappa}\psi)(x,\omega,E)\nonumber\\
&
=
\int_{I'}\chi_{\R_+}(E'-\kappa E)\hat\sigma_{j}(x,E',E){1\over{(E'-E)^j}}
\int_{0}^{2\pi}\psi(x,\gamma(E',E,\omega)(s),E')dsdE'
\eea
we  find   that the operators $K_{j,0,\kappa}$ are partial Schur integral operators and hence they are bounded operators $L^2(G\times S\times I)\to L^2(G\times S\times I)$ (section \ref{ad-sec}). 
Hence it suffices to consider approximations only for the operators 
$K_{j,1,\kappa}$.

\subsection{Approximation of hyper-singular integrals for primary particles}\label{app}

Let
\[
f_1(x,\omega,E',E):=
\hat\sigma_{1}(x,E',E)
\int_{0}^{2\pi}\psi(x,\gamma(E',E,\omega)(s),E')ds.
\]
For $K_{1,1,\kappa}$ we apply the approximation
\be\label{trunc-2}
f_1(x,\omega,E',E)\approx  f_1(x,\omega,E,E)
\ee
\emph{on the interval $[E,\kappa E]$} which leads  to the approximation
\be\label{trunc-1}
(K_{1,1,\kappa}\psi)(x,\omega,E)\approx (\widetilde K_{1,1,\kappa}\psi)(x,\omega,E):=
2\pi\ \ln(\kappa E-E)\hat\sigma_{1}(x,E,E)\psi(x,\omega,E)
\ee
where we recalled that $\gamma(E,E,\omega)(s)=\omega$ and by (\ref{hada1})
\[
{\rm p.f.}\int_E^{\kappa E}{1\over{E'-E}}dE'= \ln(\kappa E-E).
\]

Next we approximate $K_{2,1,\kappa}$.
Let
\[
f_2(x,\omega,E',E):=
\hat\sigma_{2}(x,E',E)
\int_{0}^{2\pi}\psi(x,\gamma(E',E,\omega)(s),E')ds.
\]
For $K_{2,1,\kappa}$ we apply the approximation
\be\label{trunc-2-k}
f_2(x,\omega,E',E)\approx  f_2(x,\omega,E,E)+{\p {f_2}{E'}}(x,\omega,E,E)(E'-E)
\ee
\emph{on the interval $[E,\kappa E]$}. 
In virtue of  Theorem \ref{ad-le-6}, Lemma \ref{aij}, Remark \ref{rem-aij} 
(recall that $\gamma(E,E,\omega)(s)=\omega$)
\[
&
{\p {f_2}{E'}}(x,\omega,E,E)
=
2\pi\ {\p {\hat\sigma_{2}}{E'}}(x,E,E)\psi(x,\omega,E)\\
&
+\hat\sigma_{2}(x,E,E)(A(x,\omega,E,\partial_\omega)\psi)(x,\omega,E)
+2\pi\ \hat\sigma_{2}(x,E,E){\p \psi{E}}(x,\omega,E)
\]
where
\be\label{pd-term-A}
(A(x,\omega,E,\partial_\omega)\psi)(x,\omega,E):=
\sum_{|\alpha|\leq 2}
a_{\alpha}(\omega,E)(\partial_{\omega}^\alpha \psi)(x,\omega,E)
=
-\pi\ (\partial_{E'}\mu)(E,E)
(\Delta_S\psi)(x,\omega,E).
\ee
Hence we obtain the approximation
\bea\label{trunc-4}
&
(K_{2,1,\kappa}\psi)(x,\omega,E)
={\rm p.f.}\int_E^{\kappa E}{1\over{(E'-E)^2}}
 f_2(x,\omega,E',E)dE'\nonumber\\
&
\approx
(\widetilde K_{2,1,\kappa}\psi)(x,\omega,E):=
-2\pi\ {1\over{\kappa E-E}}\hat\sigma_{2}(x,E,E)\psi(x,\omega,E) \nonumber\\
&
+\ln(\kappa E-E)\hat\sigma_{2}(x,E,E)(A(x,\omega,E,\partial_\omega)\psi)(x,\omega,E)\nonumber\\
&
+2\pi\ \hat\sigma_{2}(x,E,E)\ln(\kappa E-E){\p \psi{E}}(x,\omega,E)
+2\pi\ \ln(\kappa E-E){\p {\hat\sigma_{2}}{E'}}(x,E,E)\psi(x,\omega,E)
\eea
where we used (recall  (\ref{hada1}), (\ref{hada2})) 
\[
{\rm p.f.}\int_E^{\kappa E}{1\over{E'-E}}dE'=\ln(\kappa E-E),\
{\rm p.f.}\int_E^{\kappa E}{1\over{(E'-E)^2}}dE'=-{1\over{\kappa E-E}}.
\]

Denote
\[
S_\kappa(x,E):=2\pi\ \hat\sigma_{2}(x,E,E)\ln(\kappa E-E),
\]
\[
&
\Sigma_\kappa(x,E):=\Sigma(x,E)+2\pi\ \hat\sigma_{2}(x,E,E){1\over{\kappa E-E}}\\
&
-2\pi\ \ln(\kappa E-E){\p {\hat\sigma_{2}}{E'}}(x,E,E)
+2\pi\ \ln(\kappa E-E)\hat\sigma_{1}(x,E,E),
\]
\[
(Q_\kappa(x,\omega,E,\partial_\omega)\psi)(x,\omega,E):=
\ln(\kappa E-E)\hat\sigma_{2}(x,E,E)(A(x,\omega,E,\partial_\omega)\psi)(x,\omega,E),
\]
\[
K_{r,\kappa}:=K_r+K_{2,0,\kappa}-K_{1,0,\kappa}.
\]
Then the truncated and approximated transport operator, say $T_\kappa$, is
(by (\ref{de-T}), (\ref{trunc-1}), (\ref{trunc-4}))
\bea\label{trunc-6}
&
T_\kappa\psi:=-\widetilde K_{2,1,\kappa}\psi+\widetilde K_{1,1,\kappa}\psi
+\omega\cdot\nabla_x\psi+\Sigma\psi-K_{r,\kappa}\psi
\nonumber\\
&
=
-S_\kappa(x,E){\p \psi{E}}
-Q_\kappa(x,\omega,E,\partial_\omega)\psi
+\omega\cdot\nabla_x\psi+\Sigma_\kappa\psi-K_{r,\kappa}\psi.
\eea

The restricted collision operator $K_{r,\kappa}$ is
\bea
& 
(K_{r,\kappa}\psi)(x,\omega,E)=((K_r+K_{2,0,\kappa}-K_{1,0,\kappa})\psi)(x,\omega,E)\nonumber\\
&
=
\int_{I'}\chi_{\R_+}(E'-E)\hat\sigma_{0}(x,E',E)
\int_{0}^{2\pi}\psi(x,\gamma(E',E,\omega)(s),E')dsdE'\nonumber\\
&
-
\int_{I'}\chi_{\R_+}(E'-\kappa E)\hat\sigma_{1}(x,E',E){1\over{E'-E}}
\int_{0}^{2\pi}\psi(x,\gamma(E',E,\omega)(s),E')dsdE'\nonumber\\
&
+\int_{I'}\chi_{\R_+}(E'-\kappa E)\hat\sigma_{2}(x,E',E){1\over{(E'-E)^2}}
\int_{0}^{2\pi}\psi(x,\gamma(E',E,\omega)(s),E')dsdE'\nonumber\\
&
=
\int_{I'}\hat\sigma_{r,\kappa}(x,E',E)
\int_{0}^{2\pi}\psi(x,\gamma(E',E,\omega)(s),E')dsdE'
\eea
where
\[
&
\hat\sigma_{r,\kappa}(x,E',E):=
\chi_{\R_+}(E'-E)\hat\sigma_{0}(x,E',E)
-\chi_{\R_+}(E'-\kappa E)\hat\sigma_{1}(x,E',E){1\over{E'-E}}\\
&
+
\chi_{\R_+}(E'-\kappa E)\hat\sigma_{2}(x,E',E){1\over{(E'-E)^2}}.
\]

As a conclusion we see that the \emph{approximative transport operator is of the type}
\bea\label{app-tr-op}
&
T\psi=a(x,E){\p {\psi}E}
+\sum_{|\alpha|\leq 2}b_\alpha(x,\omega,E)\partial_\omega^\alpha\psi
+\omega\cdot\nabla_x\psi+\Sigma(x,\omega,E)\psi-K_r\psi.
\eea

\subsection{Error analysis}\label{error}

The only approximations which we used above were
\[
K_{1,1,\kappa}\approx \widetilde K_{1,1,\kappa}\ {\rm and}\ K_{2,1,\kappa}\approx \widetilde K_{2,1,\kappa}
\]
where $\widetilde K_{j,1,\kappa},\ j=1,2$ are given by (\ref{trunc-1}) and (\ref{trunc-4}), respectively. The approximation $T_\kappa$ (given by (\ref{trunc-6}))  of the exact transport operator $T$ is 
\[
T_\kappa=-\widetilde K_{2,1,\kappa}+\widetilde K_{1,1,\kappa}+\omega\cdot\nabla_x+\Sigma-(K_r+K_{2,0,\kappa}-K_{1,0,\kappa}).
\]
Recalling that
\[
T=- K_{2,1,\kappa}+ K_{1,1,\kappa}+\omega\cdot\nabla_x+\Sigma-(K_r+K_{2,0,\kappa}-K_{1,0,\kappa})
\]
we see that the error is
\be\label{err+1}
T-T_\kappa=- (K_{2,1,\kappa}-\widetilde K_{2,1,\kappa})+ (K_{1,1,\kappa}-
\widetilde K_{1,1,\kappa}).
\ee

Let
\[
h_1(x,\omega,E',E):=\int_0^{2\pi} \psi(x,\gamma(E',E,\omega)(s),E')ds
\]
and for $E'\not=E$
\[
h_2(x,\omega,E',E):=
\int_{0}^{2\pi}\la \nabla_S\psi(x,\gamma(E',E,\omega)(s),E'),{\p \gamma{E'}}(E',E,\omega)(s)\ra ds .
\]
By   the proof of Theorem \ref{ad-le-6} $h_2$ is defined also for $E'=E$ and 
\[
h_2(x,\omega,E,E)=\lim_{E'\to E}h_2(x,\omega,E',E)
= 
(A(x,\omega,E,\partial_\omega)\psi)(x,\omega,E).
\]

We show

\begin{theorem}\label{err-th-1}
Suppose that $\hat\sigma_1\in C(\ol G, C^1(I'\times I))$ 
and 
$\hat\sigma_2\in C(\ol G, C^{2}(I'\times I)$. Then for $\psi\in C^1(\ol G,C^{2}(I,C^{3}(S))$
\be\label{err-es}
\n{T\psi-T_\kappa\psi}_{L^\infty(G\times S\times I)}
\leq C\n{\psi}_{C(\ol G,C^{2}(I,C^{3}(S))}(\kappa-1)^{1\over 2}.
\ee
Hence $T_\kappa\psi\to T\psi$ uniformly in $\ol G\times S\times I$ as $\kappa\to 1^+$.
\end{theorem}

\begin{proof}
Recall that
\[
( K_{j,1,\kappa}\psi)(x,\omega,E)={\rm p.f.}\int_E^{\kappa E}{1\over {(E'-E)^j}}f_j(x,\omega,E',E) dE'
\]
where
\[
f_j(x,\omega,E',E)=
\hat\sigma_{j}(x,E',E)
\int_{0}^{2\pi}\psi(x,\gamma(E',E,\omega)(s),E') .
\]
For $K_{1,1,\kappa}$ we used the approximation
\be\label{trunc-2-a}
f_1(x,\omega,E',E)\approx  f_1(x,\omega,E,E)
\ee
{on the interval $[E,\kappa E]$}. 
Hence the error is
\be\label{trunc-1-a}
|(K_{1,1,\kappa}\psi)(x,\omega,E)-(\widetilde K_{1,1,\kappa}\psi)(x,\omega,E)|
=
\Big|{\rm p.f.}\int_E^{\kappa E}{{f_1(x,\omega,E',E)-f_1(x,\omega,E,E)}\over{E'-E}}dE'\Big|
\ee
In virtue of Corollary \ref{x-le-cor} for $E'\geq E$ (we omit technical details emerging from the factor $\hat\sigma_{1}(x,E',E)$)
\[
|f_1(x,\omega,E',E)-f_1(x,\omega,E,E)|\leq C\n{\psi}_{C(\ol G,C^1(S\times I))}(E'-E)^{1\over 2}.
\]
Hence the p.f.-integral in (\ref{trunc-1-a}) is an ordinary improper integral and
\bea\label{trunc-1-b}
&
|(K_{1,1,\kappa}\psi)(x,\omega,E)-(\widetilde K_{1,1,\kappa}\psi)(x,\omega,E)|
\leq 
\int_E^{\kappa E}\Big|{{f_1(x,\omega,E',E)-f_1(x,\omega,E,E)}\over{E'-E}}\Big|dE'\nonumber\\
&
\leq 
\int_E^{\kappa E}C_1' \n{\psi}_{C(\ol G,C^1(S\times I))}(E'-E)^{{1\over 2}-1}dE' ={2}C_1'\n{\psi}_{C(\ol G,C^1(S\times I))}(\kappa-1)^{1\over 2}E^{\alpha\over 2}\nonumber\\
&
\leq {2}C_1'E_m^{1\over 2}\n{\psi}_{C(\ol G,C^1(S\times I))}(\kappa-1)^{1\over 2}.
\eea

For $K_{2,1,\kappa}$ we used the approximation
\be\label{err-2}
f_2(x,\omega,E',E)\approx  f_2(x,\omega,E,E)+{\p {f_2}{E'}}(x,\omega,E,E)(E'-E)
\ee
{on the interval $[E,\kappa E]$}.
In virtue of Taylor's formula
\bea\label{err-1}
&
f_2(x,\omega,E',E)=  f_2(x,\omega,E,E)+\int_0^1{\p {f_2}{E'}}(x,\omega,E+t(E'-E),E)dt\ (E'-E)\nonumber\\
&
=
f_2(x,\omega,E,E)+
{\p {f_2}{E'}}(x,\omega,E,E)(E'-E)\nonumber\\
&
+
\int_0^1\big({\p {f_2}{E'}}(x,\omega,E+t(E'-E),E)-{\p {f_2}{E'}}(x,\omega,E,E)\big)dt\ (E'-E)
\eea
Hence the error is
\bea\label{err-3}
&
|(K_{2,1,\kappa}\psi)(x,\omega,E)-(\widetilde K_{2,1,\kappa}\psi)(x,\omega,E)|
\nonumber\\
&
=
\Big|{\rm p.f.}\int_E^{\kappa E} \int_0^1{{{\p {f_2}{E'}}(x,\omega,E+t(E'-E),E)-{\p {f_2}{E'}}(x,\omega,E,E)}\over{E'-E}} dE'dt\Big|
\eea

By Theorem \ref{ad-le-6} for $E\not=E'$
\bea\label{err-4}
&
{\p {f_2}{E'}}(x,\omega,E',E) 
=
{\p {\hat\sigma_{2}}{E'}}(x,E',E)\int_{0}^{2\pi}\psi(x,\gamma(E',E,\omega)(s),E')
\nonumber\\
&
\hat\sigma_{2}(x,E',E)
\int_0^{2\pi}\la (\nabla_{S}\psi)(x,\gamma(E',E,\omega)(s),E'),
{\p {\gamma}{E'}}(E',E,\omega)(s)\ra ds
\nonumber\\
&
+
\hat\sigma_{2}(x,E',E)
\int_0^{2\pi}{\p {\psi}{E}}(x,\gamma(E',E,\omega)(s),E')ds
\eea
and
for $E'=E$
\bea\label{err-5}
&
{\p {f_2}{E'}}(x,\omega,E,E) 
=
2\pi\ {\p {\hat\sigma_{2}}{E'}}(x,E,E)\psi(x,\omega,E)
\nonumber\\
& 
+
\hat\sigma_{2}(x,E,E)(A(x,\omega,E,\partial_\omega)\psi)(x,\omega,E)
+
2\pi\ \hat\sigma_{2}(x,E,E) {\p {\psi}{E}}(x,\omega,E).
\eea

In virtue of Corollary \ref{le-reg-cor} we find by (\ref{err-3}) as in Part A
that (again we omit technical details; note that $E+t(E'-E)\geq E$)
\bea\label{err-6}
&
|(K_{2,1,\kappa}\psi)(x,\omega,E)-(\widetilde K_{2,1,\kappa}\psi)(x,\omega,E)|
\nonumber\\
&
=
\int_E^{\kappa E} \int_0^1\Big|{{{\p {f_2}{E'}}(x,\omega,E+t(E'-E),E)-{\p {f_2}{E'}}(x,\omega,E,E)}\over{E'-E}} \Big|dE'dt
\nonumber\\
&
\leq 
C_2'\n{\psi}_{C(\ol G,C^{2}(I,C^{3}(S))}
\int_E^{\kappa E}\int_0^1(t(E'-E))^{{1\over 2}-1}dE'  dt\nonumber\\
&
\leq
C_2'{4}E_m^{1\over 2}\n{\psi}_{C(\ol G,C^{2}(I,C^{3}(S))}(\kappa-1)^{1\over 2}.
\eea
This completes the proof.
\end{proof}

The formal adjoint $T^*$ of $T$ is given by (see section \ref{ad-hada-op})
\bea\label{ad-T}
&
(T^*v)(x,\omega',E')=
-
{\rm p.f.}\int_{E_0}^{E'}{1\over{(E'-E)^2}}
\hat{\sigma}_{2}(x,E',E) 
\cdot
\int_{0}^{2\pi}v(x,\gamma(E',E,\omega)(s),E)ds  dE
\nonumber\\
&
+
{\rm p.f.}\int_{E_0}^{E'}{1\over{E'-E}}\hat\sigma_1(x,E',E)
\int_0^{2\pi}
v(x,\gamma(E',E,\omega')(s),E)dsdE
\nonumber\\
&
-\omega\cdot\nabla_xv+\Sigma^*v-K_r^*v,\ v\in \mc D_{(0)}
\eea
where $\mc D_{(0)}$ is a relevant space of test functions
\[
\mc D_{(0)}:=\{v\in C^1(\ol G,C^2(I,C^3(S)))|\ v(x,\omega,E_0)=0\}.
\]

The formal adjoint $T_\kappa^*$ of $T_\kappa$ is (say for $v\in C_0^2(G\times S\times I^\circ)$)
\bea\label{err-7}
&
T_\kappa^*v:=
{\p {(S_\kappa v)}{E}}
-Q_\kappa^*(x,\omega,E,\partial_\omega)v
-\omega\cdot\nabla_x v+\Sigma^*_\kappa v-K^*_{r,\kappa}v\nonumber\\
&
=S_\kappa {\p v{E}}-Q_\kappa^*(x,\omega,E,\partial_\omega)v
-\omega\cdot\nabla_x v+{\p {S_\kappa }{E}}v+ \Sigma^*_\kappa v-K^*_{r,\kappa}v.
\eea
Here
\[
&
\Big({\p {S_\kappa }{E}}+ \Sigma^*_\kappa \Big)(x,E')
=\Sigma(x,E')+2\pi\ \Big({\p {\hat\sigma_2}{E'}}(x,E',E')+
{\p {\hat\sigma_2}{E}}(x,E',E')\Big)\ln(\kappa E'-E')
\\
&
+2\pi\ \hat\sigma_2(x,E',E'){1\over {E'}}
+
2\pi\ \hat\sigma_2(x,E',E'){1\over {\kappa E'-E'}}
\\
&
-
2\pi\ {\p {\hat\sigma_2}{E'}}(x,E',E')\ln(\kappa E'-E')
+2\pi\ \hat\sigma_1(x,E',E')\ln(\kappa E'-E')
\]
and by Theorem \ref{Kr-adjoint}
\[
(K^*_{r,\kappa}v)(x,\omega,E')=
\int_{I'}\hat\sigma_{r,\kappa}(x,E',E)
\int_{0}^{2\pi}v(x,\gamma(E',E,\omega)(s),E)dsdE
\]

\begin{theorem}\label{err-th-2}
Suppose that $\hat\sigma_1\in C(\ol G, C^1(I'\times I))$ and 
$\hat\sigma_2\in C(\ol G, C^{2}(I'\times I))$. Then for $v\in C^1(\ol G,C^{2}(I,C^{3}(S))$
\be\label{err-es-a}
\n{T^*v-T^*_\kappa v}_{L^\infty(G\times S\times I)}
\leq C\n{v}_{C^1(\ol G,C^{2}(I,C^{3}(S))}(\kappa-1)^{1\over 2}.
\ee
Hence $T^*_\kappa v\to T^*v$ uniformly in $\ol G\times S\times I$ as $\kappa\to 1^+$.
\end{theorem}

\begin{proof}
Decompose the p.f.-integrations in (\ref{ad-T}) as follows (for  $j=1,2$)
\bea\label{err-8}
&
{\rm p.f.}\int_{E_0}^{E'}{1\over{(E'-E)^2}}
\hat{\sigma}_{2}(x,E',E) 
\cdot
\int_{0}^{2\pi}v(x,\gamma(E',E,\omega)(s),E)ds  dE
\nonumber\\
&
=
{\rm p.f.}\int_{{{E'}/\kappa}}^{E'}{1\over{(E'-E)^2}}
\hat{\sigma}_{2}(x,E',E) 
\cdot
\int_{0}^{2\pi}v(x,\gamma(E',E,\omega)(s),E)ds  dE
\nonumber\\
&
+
\int_{E_0}^{{{E'}/\kappa}}{1\over{(E'-E)^2}}
\hat{\sigma}_{2}(x,E',E) 
\cdot
\int_{0}^{2\pi}v(x,\gamma(E',E,\omega)(s),E)ds  dE
\nonumber\\
&
=:(K^*_{2,1,\kappa}v)(x,\omega,E')+(K^*_{2,0,\kappa}v)(x,\omega,E')
\eea
and
\bea\label{err-9} 
&
{\rm p.f.}\int_{E_0}^{E'}{1\over{E'-E}}\hat\sigma_1(x,E',E)
\int_0^{2\pi}
v(x,\gamma(E',E,\omega')(s),E)dsdE
\nonumber\\
&
=
{\rm p.f.}\int_{{{E'}/\kappa}}^{E'}{1\over{E'-E}}\hat\sigma_1(x,E',E)
\int_0^{2\pi}
v(x,\gamma(E',E,\omega')(s),E)dsdE
\nonumber\\
&
+
\int_{E_0}^{{{E'}/\kappa}}{1\over{E'-E}}\hat\sigma_1(x,E',E)
\int_0^{2\pi}
v(x,\gamma(E',E,\omega')(s),E)dsdE\nonumber\\
&
=:(K^*_{1,1,\kappa}v)(x,\omega,E')+(K^*_{1,0,\kappa}v)(x,\omega,E').
\eea

Let
\[
f_j^*(x,\omega,E',E):=\hat\sigma_j(x,E',E)
\int_0^{2\pi}
v(x,\gamma(E',E,\omega)(s),E)ds,\ j=1,2.
\]
Similarly to Theorem \ref{ad-le-6} one can show that
\be\label{der-a}
{\partial\over{\partial E}}\big(\int_0^{2\pi}v(x,\gamma(E',E,\omega)(s),E)ds\big)_{\Big|E=E'}
=
-(A(x,\omega,E',\partial_\omega)v)(x,\omega,E')
+
2\pi {\p {v}{E}}(x,\omega,E')
\ee
where the "-" sign in the front of $(A(x,\omega,E',\partial_\omega)v)(x,\omega,E')$ is due to the relation 
\be\label{mu-E}
\partial_{E}\mu(E,E)= {1\over E}-{1\over{E+2}}=-\partial_{E'}\mu(E,E)
\ee
Using the approximations
\[
f_1^*(x,\omega,E',E')\approx
f_1^*(x,\omega,E',E'),
\]
\[
f_2^*(x,\omega,E',E')\approx
f_2^*(x,\omega,E',E')+
{\p {f_2}{E}}(x,\omega,E',E')(E-E')
\]
on the interval $[{{E'}/\kappa},E']$ and 
noting that
\[
\lim_{\kappa\to 1^+}\ln(\kappa)=0,\ E'-{{E'}\over \kappa}={{(\kappa-1)E'}\over\kappa}
\]
the estimate (\ref{err-es-a}) is shown similarly to the proof of Theorem \ref{err-th-1}. We omit further details.

\end{proof}

\begin{remark}
The assertions of Theorem \ref{err-th-1} and \ref{err-th-2} are valid under weaker assumptions. For example, it suffices to assume only that 
$\hat\sigma_1\in C(\ol G, C^\alpha(I'\times I))$,
$\hat\sigma_2\in C(\ol G, C^{1+\alpha}(I'\times I))$ and $\psi,\ v\in C^1(\ol G,C^{1+\alpha}(I,C^{2+\alpha}(S))$
for $\alpha>0$.
\end{remark}

\section{Discussion}\label{discus}

The  paper considers an exact  linear Boltzmann transport operator related to the charged particle transport.
The hyper--singularities in the differential cross-sections of certain charged particle collisions lead to extra singular integral-partial differential terms in kinematic equations. 
The analysis showed that the terms contain the first-order partial derivatives with respect to energy $E$ combined with Hadamard finite part integral operators which are pseudo-differential-like operators. With respect to angle
$\omega$ also second-order partial derivatives appear. Additionally, some mixed terms appear (see \ref{intro-exact}).
Note especially that in the expression (\ref{intro-exact})  there is the term $\sum_{|\alpha|\leq 2}a_\alpha(\omega,E)\partial_\omega^\alpha\psi$ which is corresponding to the second order partial differential term with respect to angular variables appearing in Fokker-Plank equation. In many cases this term turns out to be an elliptic operator (on $S$) which helps the analysis and in numerical treatments it stabilizes computations. For M\o ller scattering the term is $-\mu_{E'}(E,E)\Delta_S\psi$ where $\Delta_S$ is the Laplace-Beltrami operator (Remark \ref{rem-aij}).

The pseudo-differential-like terms can be  approximated by pseudo-differential operators (\cite{tervo18} or \cite{tervo18-up}).
Moreover, in \cite{tervo18-up} we gave a further approximation  where the resulting transport operator is containing only partial derivatives and Schur partial integral operators. These 
approximations are essentially founded on Taylor's formulas
and  angular approximations of (new) primary particles.
When considering partial differential approximations the resulting approximative transport operator, as we verified in section \ref{tr-app}, is a \emph{partial integro-differential operator} of the form
\bea\label{tr-op-gene}
&
T\psi=a(x,E){\p {\psi}E}
+\sum_{|\alpha|\leq 2}b_\alpha(x,\omega,E)\partial_\omega^\alpha\psi
+\omega\cdot\nabla_x\psi+\Sigma(x,\omega,E)\psi-K_r\psi.
\eea
These approximations are known as \emph{Continuous Slowing Down Approximations} (CSDA) of BTE.
For example, 
in the dose calculation  the scattering events containing hyper--singularities  are the primary M\o ller scattering for electrons and positrons and the Bremsstrahlung. Conventionally these events have needed  the CSDA modelling like (\ref{tr-op-gene}) in numerical computations.

In this paper we additionally exposed the variational formulation for the  transport problem where no approximations are applied. In principle, the obtained variational problem can be numerically solved applying e.g. Galerkin finite element methods (FEM). The numerical computations and approximations of (hyper)-singular integrals has been studied by various methods and for various needs. In particuler, for Galerkin methods them are known in   the field of boundary element methods (BEM) (see e.g. \cite{carley}). BEM considers numerical methods for solutions of the (hyper)-singular integral equations emerging from  solutions of certain initial-boundary value problems.

The well-posedness of the transport problem that is, existence and uniqueness of solutions together with pertinent a priori estimates is a central importance. 
We suggest that the exposed variational formulations together with Lions-Lax-Milgram Theorem give one alternative to investigate well-posedness of the problems where exact operators like (\ref{intro-exact}) are involved in the (coupled) transport system.  
Other successful  methods might be e.g. semigroup-dissipativity-perturbation  methods and fixed-point (contraction) methods. 
In addition,
Sobolev regularity (say, on the scales of Sobolev-Slobodevskij spaces) of solutions  remains  open.
It is known that the transport problems have a limited Sobolev regularity but higher order \emph{weighted (co-normal) Sobolev regularity} can be achieved.
The regularity of solutions are needed e.g. in various approximation and convergence (error analysis) treatments.
We remark that the initial inflow boundary value problems related to transport problems are so called \emph{variable boundary value multiplicity} that is, the dimension of the kernel of the boundary operator is not constant
(e.g. \cite{nishitani98}, \cite{rauch94}). This makes the inflow boundary transport problems
more subtle independently of the applied methods.

\vskip10mm
{\bf Acknowledgement}. The author thanks Dr. Petri Kokkonen for useful discussions while preparing the paper.

\end{document}